\documentclass[a4paper,10pt,reqno]{amsart}
\usepackage{amssymb,latexsym,bbm,amsmath,amsfonts}
\usepackage{amsthm}
\usepackage[mathscr]{eucal}
\usepackage{rotating}
\usepackage{multirow}
\usepackage{slashbox}

\numberwithin{equation}{section}
\newtheorem{theorem}{Theorem}
\newtheorem{lemma}{Lemma}
\newtheorem{proposition}{Proposition}

\theoremstyle{definition}

\newtheorem{example}[theorem]{Example}

\theoremstyle{remark}

\newcommand {\sgn}{\mbox{sgn}}

\DeclareMathOperator{\N}{{\mathbb N}}                
\newcommand{\RN}[1]{\uppercase\expandafter{\romannumeral#1}}
\newcommand{\Rn}[1]{\romannumeral#1\relax}

\begin{document}
\title[Decay Rate Preservation of Regularly Varying ODEs and SDEs]
{On necessary and sufficient conditions for preserving convergence rates to equilibrium in deterministically and stochastically perturbed
differential equations with regularly varying nonlinearity}

\author{John A. D. Appleby}
\address{Edgeworth Centre for Financial Mathematics, School of Mathematical
Sciences, Dublin City University, Glasnevin, Dublin 9, Ireland}
\email{john.appleby@dcu.ie} \urladdr{webpages.dcu.ie/\textasciitilde
applebyj}

\author{Denis D. Patterson}
\address{School of Mathematical
Sciences, Dublin City University, Glasnevin, Dublin 9, Ireland}
\email{denis.patterson2@mail.dcu.ie} \urladdr{sites.google.com/a/mail.dcu.ie/denis-patterson}

\thanks{John Appleby gratefully acknowledges Science Foundation
Ireland for the support of this research under the Mathematics
Initiative 2007 grant 07/MI/008 ``Edgeworth Centre for Financial
Mathematics''. Denis Patterson is supported by the Government of Ireland Postgraduate Scholarship Scheme operated by the Irish Research Council under the project ``Persistent and strong dependence in growth rates of solutions of stochastic and deterministic functional differential equations with applications to finance'', GOIPG/2013/402.} 
\date{22 December 2013}

\begin{abstract}
This paper develops necessary and sufficient conditions for the preservation of asymptotic convergence rates of deterministically and stochastically perturbed ordinary differential equations with regularly varying nonlinearity close to their equilibrium. Sharp conditions are also established which preserve the asymptotic behaviour of the derivative of the underlying unperturbed equation. Finally, necessary and sufficient conditions are established which enable finite difference approximations to the derivative in the stochastic equation to preserve the asymptotic behaviour of the derivative of the unperturbed equation, even 
though the solution of the stochastic equation is nowhere differentiable, almost surely. 
\end{abstract}

\maketitle

\begin{center}
This paper is dedicated to Professor Istv\'an Gy\H{o}ri on the occasion of his 70th birthday.
\end{center}

\section{Introduction}
In this paper we classify the rates of convergence to a limit of the solutions of scalar ordinary and stochastic differential equations
of the form 
\begin{equation} \label{eq.odepert}
x'(t)=-f(x(t))+g(t), \quad t>0; \quad x(0)=\xi,
\end{equation}
and 
\begin{equation} \label{eq.sde}
dX(t)=-f(X(t))\,dt + \sigma(t)\,dB(t), \quad t\geq 0,
\end{equation}
where $B$ is a one--dimensional standard Brownian motion. The asymptotic behaviour of the derivative in the case of \eqref{eq.odepert}, and 
the scaled increment \\ $\left(X(t+h)-X(t)\right)/h$ for fixed $h>0$, in the case of \eqref{eq.sde}, is also classified. 

We assume that the unperturbed equation
\begin{equation} \label{eq.ode}
y'(t)=-f(y(t)), \quad t>0; \quad y(0)=\zeta
\end{equation}
has a unique globally stable equilibrium (which we set to be at zero). This is characterised by the condition
\begin{equation} \label{eq.fglobalstable}
xf(x)>0 \quad\text{for $x\neq 0$,} \quad f(0)=0.
\end{equation}
In order to ensure that \eqref{eq.ode}, \eqref{eq.odepert} and \eqref{eq.sde} have
continuous solutions, we assume
\begin{equation} \label{eq.fgsigmacns}
  f\in C(\mathbb{R};\mathbb{R}), \quad g\in C([0,\infty);\mathbb{R}), \quad \sigma\in C([0,\infty);\mathbb{R}).
\end{equation}
The condition \eqref{eq.fglobalstable} ensures that any solution of \eqref{eq.odepert} or \eqref{eq.sde} is \emph{global} i.e., that
\begin{align*}
\tau_D:=\inf\{t>0\,:\,x(t)\not\in (-\infty,\infty)\}=+\infty, \\ 
\tau_S=\inf\{t>0\,:\, X(t)\not\in (-\infty,\infty)\}=+\infty, \quad \text{a.s.}
\end{align*}
We also ensure that there is exactly one continuous solution of both \eqref{eq.odepert} and \eqref{eq.ode} by assuming
\begin{equation} \label{eq.floclip}
\text{$f$ is locally Lipschitz continuous on $\mathbb{R}$}.
\end{equation}
This condition ensures the existence of a unique continuous adapted process which obeys \eqref{eq.sde}. 

In \eqref{eq.ode}, \eqref{eq.odepert} and \eqref{eq.sde}, we assume that $f(x)$ does \emph{not} have linear leading order behaviour as $x\to 0$; 
moreover, we do not ask that $f$ forces solutions of \eqref{eq.ode} to hit zero in finite time. Since $f$ is continuous, we are free to define
\begin{equation} \label{def.F}
F(x)=\int_x^1 \frac{1}{f(u)}\,du, \quad x>0,
\end{equation}
and avoiding solutions of \eqref{eq.ode} to hitting zero in finite time forces
\begin{equation} \label{eq.Ftoinfty}
\lim_{x\to 0^+} F(x)=+\infty.
\end{equation}
We notice that $F:(0,\infty)\to\mathbb{R}$ is a strictly decreasing function, so it has an inverse $F^{-1}$. Clearly, \eqref{eq.Ftoinfty} implies that
\[
\lim_{t\to\infty} F^{-1}(t)=0.
\]
The significance of the functions $F$ and $F^{-1}$ is that they enable us to determine the rate of convergence of solutions of \eqref{eq.ode} to zero, because
$F(y(t))-F(\zeta)=t$ for $t\geq 0$ or $y(t)=F^{-1}(t+F(\zeta))$ for $t\geq 0$. It is then of interest to ask whether solutions of \eqref{eq.odepert} or of 
\eqref{eq.sde} will still converge to zero as $t\to\infty$, and to determine conditions (on $g$ and $\sigma$) under which the rate of decay of the solution of the underlying unperturbed equation \eqref{eq.ode} is preserved by the solutions of \eqref{eq.odepert} and \eqref{eq.sde}. 

In order to do this with reasonable generality we find it convenient and natural to assume that the function $f$ is regularly varying at zero.
We recall that a measurable function $f:(0,\infty)\to (0,\infty)$ with $f(x)>0$ for $x>0$ is said to be regularly varying at $0$ with index $\beta\in\mathbb{R}$ if 
\[
\lim_{x\to 0^+} \frac{f(\lambda x)}{f(x)}=\lambda^\beta, \quad \text{for all $\lambda>0$}.
\]
In the case that $f$ is regularly varying at zero with index $\beta>1$, the function $t\mapsto F^{-1}(t)$ is regularly varying at infinity with 
index $-1/(\beta-1)$. 
A measurable function $h:[0,\infty)\to [0,\infty)$ with $h(t)>0$ for $t\geq 0$ is said to regularly varying at infinity with index $\alpha\in\mathbb{R}$ if 
\[
\lim_{t\to\infty} \frac{h(\lambda t)}{h(t)}=\lambda^\alpha, \quad \text{for all $\lambda>0$}.
\]
We use the notation $f\in \text{RV}_0(\beta)$ and $h\in\text{RV}_\infty(\alpha)$. Many useful properties of regularly varying functions, including those employed here, are recorded in Bingham, Goldie and Teugels~\cite{BGT}.  

The main results of the paper give (essentially) necessary and sufficient conditions under which the asymptotic rate of decay of solutions of the perturbed 
equations are inherited from those of \eqref{eq.ode}. We consider first the deterministic equation \eqref{eq.odepert}. 
Suppose that $f$ is regularly varying at zero with index $\beta>1$, and is asymptotic at zero to an odd function. Suppose further that $g$ is continuous, and that it is known that $x(t)\to 0$ as $t\to\infty$. 
Then the following statements are equivalent
\begin{itemize}
\item[(a)] The functions $f$ and $g$ obey
\[
\lim_{t\to\infty} \int_0^t g(s)\,ds \text{ exists}, \quad \lim_{t\to\infty} \frac{\int_t^\infty g(s)\,ds}{F^{-1}(t)}=0;
\]
\item[(b)] There is $\lambda\in \{-1,0,1\}$ such that
\[
\lim_{t\to\infty} \frac{x(t)}{F^{-1}(t)}= \lambda.
\]
\end{itemize} 
The cases $\lambda=\pm 1$ reproduce the asymptotic behaviour of the solution $y$ of \eqref{eq.ode} according to whether the initial condition is positive or 
negative. The case $\lambda=0$ means that solutions of the perturbed equation decay more rapidly to zero than those of the unperturbed equation. We believe that this behaviour is rare, but it can arise for special perturbations. It is notable that this result does not require sign or pointwise conditions on the rate of 
decay of $g$; indeed, it can be shown that $g$ need not be absolutely integrable, a strictly weaker condition than the first part of condition (a). Indeed 
one can have that $\limsup_{t\to\infty} |g(t)|/\Gamma(t)=1$ for arbitrarily rapidly growing $\Gamma$, while solutions still obey condition (b). The asymptotic 
oddness of $f$ is assumed so as to ensure that convergence rates from both sides of the equilibrium are the same.  

Once the above result has been established, it is straightforward to characterise conditions under which the solution of \eqref{eq.odepert} and its derivative inherit the asymptotic behaviour of those of \eqref{eq.ode}. In that case, under the same hypotheses as above, we prove that the following statements are equivalent:
\begin{itemize}
\item[(c)] The functions $f$ and $g$ obey 
\[
\lim_{t\to\infty} \frac{g(t)}{f(F^{-1}(t))}=0; 
\]
\item[(d)] There is $\lambda\in \{-1,0,1\}$ such that
\[
\lim_{t\to\infty} \frac{x(t)}{F^{-1}(t)}= \lambda, 
\quad
\lim_{t\to\infty} \frac{x'(t)}{f(F^{-1}(t))}=-\lambda.
\]
\end{itemize}
We notice that solutions of \eqref{eq.ode} with positive initial condition obey (d) with $\lambda=1$, while those with negative initial condition obey (d) 
with $\lambda=-1$. The condition (c), in the case of positive $g$ and positive initial condition $\xi$, was employed in Appleby and Patterson~\cite{apppatt} to establish condition (a) (with $\lambda=1$). However, condition (a) shows that such a pointwise condition is merely sufficient, rather than necessary, to preserve the asymptotic behaviour of solutions of \eqref{eq.ode}. 

Corresponding results apply to the stochastic equation \eqref{eq.sde}. Once again we assume that $f$ is in $\text{RV}_0(\beta)$ for $\beta>1$ and further suppose that $f$ is asymptotic to an odd function at zero. We note first that if $\sigma\not\in L^2([0,\infty);\mathbb{R})$, then 
\[
\mathbb{P}\left[\lim_{t\to\infty} \frac{X(t)}{F^{-1}(t)} \text{ exists and is finite}  \right]=0.
\]
This corresponds to the necessity of the first part of condition (a) to preserve the rate of decay of solutions of \eqref{eq.ode} in the deterministic case. 
In the case when $\sigma\in L^2(0,\infty)$, we have a sharp characterisation of situations under which the solution of \eqref{eq.sde} inherits the decay rate of solutions of \eqref{eq.ode}. Define, for sufficiently large $t>0$ the function $\Sigma:[T,\infty)\to (0,\infty)$ by  
\[
\Sigma^2(t)=2\int_t^\infty \sigma^2(s)\,ds \log\log\left(\frac{1}{\int_t^\infty \sigma^2(s)\,ds}\right), \quad t\geq T,
\] 
and we suppose that 
\[
\mu:=\lim_{t\to\infty} \frac{\Sigma(t)}{F^{-1}(t)} \in  [0,\infty].
\]
Then $\mu\in (0,\infty]$ implies that 
\[
\mathbb{P}\left[\lim_{t\to\infty} \frac{X(t)}{F^{-1}(t)} \text{ exists and is finite} \right] =0 
\]
while $\mu=0$ implies that there exists a $\mathcal{F}^B(\infty)$--measurable random variable $\lambda$ such that $\mathbb{P}[\lambda\in \{-1,0,1\}]=1$ and 
\[
\mathbb{P}\left[\lim_{t\to\infty} \frac{X(t)}{F^{-1}(t)}=\lambda\right]=1.
\]
A result which is less explicit than the above, but parallel to the main result for \eqref{eq.odepert} is the following equivalence:
\begin{itemize}
\item[(e)] $\sigma$ and $f$ obey 
\[
\sigma\in L^2(0,\infty), \quad \lim_{t\to\infty} \frac{\int_t^\infty \sigma(s)\,dB(s)}{F^{-1}(t)}=0, \quad\text{a.s.};
\]
\item[(f)] There exists a $\mathcal{F}^B(\infty)$--measurable random variable $\lambda$ such that $\mathbb{P}[\lambda\in \{-1,0,1\}]=1$ and 
\[
\mathbb{P}\left[\lim_{t\to\infty} \frac{X(t)}{F^{-1}(t)}=\lambda\right]=1.
\]
\end{itemize}
Lastly, we establish a result analogous to the preservation of the asymptotic behaviour of \eqref{eq.ode} by the solution and derivative of \eqref{eq.odepert}. We suppose $\Psi$ is the complementary standard normal distribution function i.e.
\[
\Psi(x)=\frac{1}{\sqrt{2\pi}}\int_x^\infty e^{-y^2/2}\,dy, \quad x\in \mathbb{R}.
\] 
Then the following statements are equivalent:
\begin{itemize}
\item[(g)] For every $\epsilon>0$, 
\[
S_f(\epsilon,h)=\sum_{n=1}^\infty \Psi\left(\frac{\epsilon}
{
\frac{\sqrt{\int_{nh}^{(n+1)h} \sigma^2(s)\,ds}}{f(F^{-1}(nh))}
}\right)<+\infty;
\]
\item[(h)] There exists a $\mathcal{F}^B(\infty)$--measurable random variable $\lambda$ such that $\mathbb{P}[\lambda\in \{-1,0,1\}]=1$ and 
\[
\lim_{t\to\infty} \frac{X(t)}{F^{-1}(t)}=\lambda, \quad\text{a.s.}
\]
and for each $h>0$ 
\[
\lim_{t\to\infty} \frac{\frac{X(t+h)-X(t)}{h}}{f(F^{-1}(t))}=-\lambda, \quad \text{a.s.}
\]
\end{itemize}
It should be noted that this last result has a rather unexpected quality: remember first that provided $\sigma^2(t)>0$ for all $t\geq 0$, 
the sample paths of $X$ are differentiable nowhere with probability one. Therefore, we would not expect a finite difference approximation to 
the derivative of $X$ (which does not exist!) to have smooth asymptotic behaviour. But in fact, that is precisely what this last result predicts: 
if we take $h>0$ as small as we like and fixed, then provided that the noise $\sigma$ decays rapidly enough, the sample path \textit{observed regularly but not continuously} will appear asymptotically differentiable. 

We conjecture in fact that if for any $h_1>0$ we have $S_f(\epsilon,h_1)<+\infty$, then for any $h>0$ we have $S_f(\epsilon,h)<+\infty$ for all $\epsilon>0$, so 
the dependence on the ``step size'' $h$ is not as important as might be guessed from first sight.   

These asymptotic results are proven by constructing appropriate upper and lower solutions to the differential equation \eqref{eq.odepert} as in Appleby and Buckwar~\cite{appbuck}. In this paper, we prove a new result in which the solutions of \eqref{eq.odepert} and \eqref{eq.sde} are related to that of an 
``internally perturbed'' ordinary differential equation of the form $z'(t)=-f(z(t)+\gamma(t))$. The benefit gained from the added difficulty involved in bringing the perturbation inside the argument of the mean--reverting term is that the function $\gamma$ will typically have good pointwise behaviour (obeying for example 
$\gamma(t)\to 0$ or $\gamma(t)/F^{-1}(t)\to 0$ as $t\to\infty$), while the original forcing functions $g$ or $\sigma$ in \eqref{eq.odepert} and \eqref{eq.sde} 
may not have nice pointwise bounds. By means of this reformulation of the problem, we are able to determine a very fine characterisation of the desired asymptotic results. We also speculate that this approach may be very successful for dealing with highly nonlinear equations of the type \eqref{eq.odepert} or \eqref{eq.sde}
in which $f'(x)\to \infty$ as $x\to 0$ with $f$ being in $\text{RV}_0(1)$.     

The paper is a continuation of work by the authors on the deterministic equation \eqref{eq.odepert}, which only covers the case when $g$ is positive, and obeys pointwise asymptotic bounds, but which considered the asymptotic behaviour with respect to ``large'' perturbations. The principal achievement of that paper was 
therefore to give a complete description of ``positively'' perturbed equations in which the $g$ had regular asymptotic behaviour. The goal here, by contrast is to determine necessary and sufficient conditions for the preservation of convergence rates in the presence of more irregular perturbations: ones which on average may be small, but can possess ``large spikes''; perturbations which may oscillate (perhaps rapidly) between being positive and negative; and also stochastic perturbations of It\^o type, which as well as adding uncertainty, remove smoothness and natural monotonicity in the perturbation size. Despite these new complications, however, we are able to capture the key asymptotic features of the solutions. The general question of sharp conditions under which stability of 
perturbed equations  is preserved is examined by Strauss and Yorke in \cite{StrYork:1967,StrYork:1967b}. For other literature in this direction on limiting 
equations, consult the references in~\cite{JAJC:2011szeged}.

We mention some other connections of this work to research in the stochastic literature. The paper builds directly on work of the first author with Mackey 
\cite{appmack2003} which considers the asymptotic behaviour of \eqref{eq.sde} with $f(x)\sim a|x|^\beta\sgn(x)$ as $x\to 0$. In that work, it is shown for sufficiently rapidly decaying noise intensity that the solution inherits the decay rate of the underlying unperturbed ODE \eqref{eq.ode}. However, completely sharp necessary conditions for the preservation of this rate were not found in this earlier work, and the analysis was confined to the case of polynomial $f$.
Moreover, some information about the asymptotic behaviour of the derivative of $f$ close to zero was needed, and this has now been eliminated. Finally, results concerning the finite difference approximation to the ``derivative'' of $X$ were not presented in that work. Inspiration for the summation condition used to prove this result comes from the papers by Appleby, Cheng and Rodkina~\cite{JAJCAR:2011dresden,appchengrod:2012} which deal with the convergence to zero (but not the rate of convergence) of linear and nonlinear stochastic differential equations of the form \eqref{eq.sde}, and scrutiny of the proofs will show how similar arguments have been used to assist in determining the rate of decay, especially by means of an auxiliary affine SDE whose asymptotic behaviour can be determined by an essentially direct computation. Use of ``asymptotic oddness'' of mean--reverting functions in SDEs of the form \eqref{eq.sde} in order to symmetrise the dynamics can be seen in Appleby, Gleeson and Rodkina~\cite{JAJGAR:2009}, while a useful technical lemma concerning the asymptotic behaviour of the family of random variables $(\int_t^\infty \sigma(s)\,dB(s))_{t\geq 0}$ in the case when $\sigma$ is in $L^2([0,\infty);\mathbb{R})$ comes from another work of Appleby, Gleeson and Rodkina~\cite{JADJGAR:integral}. The asymptotic behaviour of discretisations of SDEs of the form \eqref{eq.sde} are studied in~\cite{appmackrod2008},
and some illustrative simulations of our results are given at the end of the paper. 

There is a nice literature on power--like dynamics in solutions of SDEs, and we invite the reader to consult 
those by Mao~\cite{MaoOx,Mao92}, Liu and Mao~\cite{LiuMao98,LiuMao:01a} and in Liu~\cite{Liu01} which deal with highly non--autonomous equations, as well 
as those of Zhang and Tsoi~\cite{ZhangTsoi96, ZhangTsoi97} and Appleby, Rodkina and Schurz \cite{ARS06} which are concerned with autonomous nonlinear equations.

The role of regular variation in the asymptotic analysis of the asymptotic behaviour of differential equations is a very active area. An important monograph summarising themes in the research up to the year 2000 is Maric~\cite{Maric2000}. Another important strand of research on the exact asymptotic behaviour of non--autonomous ordinary differential equations (of first and higher order) in which the equations have regularly varying coefficients has been developed. For recent contributions, see for example work of Evtukhov and co--workers (e.g., Evtukhov and Samoilenko~\cite{EvSam:2011}) and Koz\'ma~\cite{Kozma:2012}, as well as the references in these papers. These papers tend to be concerned with non--autonomous features which are \emph{multipliers} of the regularly--varying state dependent terms, in contrast to the presence of the nonautonomous term $g$ in \eqref{eq.odepert}, which might be thought of as \emph{additive}. Despite this extensive literature and active research concerning regular variation and asymptotic behaviour of ordinary differential equations, and despite the fact that our analysis deals with first--order equations only, it would appear that the results presented in this work are new.

The paper is organised as follows: a short section follows with notation. Section 3 outlines a result concerning the asymptotic behaviour of an equation of 
the form $x'(t)=-f(x(t)+\gamma(t))$ which turns out to be of great importance in establishing results for the solutions of perturbed equations. Its proof is involved, and deferred to Section 8. Section 4 states results concerning the deterministic equation \eqref{eq.odepert}; Section 5 is devoted to 
the stochastic equation~\ref{eq.sde}. Section 6 considers some ramifications of the results and presents examples, including those which demonstrate that the perturbations $g$ and $\sigma$ can have arbitrarily large extreme growth rates, but that solutions of the perturbed equations still inherit the dynamics of the 
underlying ODE \eqref{eq.ode}. Section 7 shows the results of some simulations of \eqref{eq.sde}. Section 9 contains the proofs deferred from Section 4. Section 10 presents most of the proofs concerning \eqref{eq.sde} postponed from Section 5. One proof from Section 5 is granted its own Section: Section 11 presents a result which characterises conditions under which the SDE preserves the asymptotic behaviour of the solution and derivative of \eqref{eq.ode}. Section 12, which  presents proofs of results stated in Section 6, concludes the paper.   

\section{Preliminaries}
In this section we introduce some common notation and list known properties of regular, slow and rapidly varying functions. We also discuss the hypotheses used in 
the paper, and then lay out and discuss the main results of the paper. 
\subsection{Notation and properties of regularly varying functions}
Throughout the paper, the set of real numbers is denoted by $\mathbb{R}$. We let $C(I;J)$ stand for the space of continuous functions which map $I$ onto $J$, where $I$ and $J$ are typically intervals in $\mathbb{R}$. Similarly, the space of differentiable functions with continuous derivative mapping $I$ onto $J$ is denoted by $C^1(I;J)$. If $h$ and $j$ are real--valued functions defined on $(0,\infty)$ and $\lim_{t\to\infty} h(t)/j(t)=1$, we sometimes use the standard asymptotic notation $h(t)\sim j(t)$ as $t\to\infty$. 
We denote the space of (absolutely) integrable functions $h:[0,\infty)\to \mathbb{R}$ which obey $\int_0^\infty |h(t)|\,dt<+\infty$ by $L^1([0,\infty);\mathbb{R})$, and the space of square integrable functions $j:[0,\infty)\to \mathbb{R}$ which obey $\int_0^\infty |h(t)|^2\,dt<+\infty$ by $L^2([0,\infty);\mathbb{R})$. 

Throughout the paper, when we work with stochastic equations, we assume that we are working  on a complete filtered probability space  $(\Omega,\mathcal{F},(\mathcal{F}(t))_{t\geq 0},\mathbb{P})$. The abbreviation \emph{a.s.} stands for
\emph{almost surely}. $B=\{B(t);t\geq 0\}$ is a standard Brownian motion adapted to $(\mathcal{F}(t))_{t\geq 0}$, and in fact, as we choose deterministic 
initial conditions for the solutions of the stochastic equations studied, there is no loss in setting the filtration to be the one naturally generated by $B$: 
\[
\mathcal{F}(t)=\mathcal{F}^B(t)=\sigma\{ B(s); 0\leq s\leq t\}, \quad t\geq 0.
\]
In our analysis, we consider the stochastic differential equation \eqref{eq.sde} with deterministic initial condition $\xi$. For simplicity, we assume throughout that $f$ is locally Lipschitz continuous and obeys $xf(x)>0$ for $x\neq 0$. 

At various points, properties of regularly varying functions are employed. We ask the reader to consult the monograph~\cite{BGT} for these results. Alternatively,
our recent preprint~\cite{apppatt} which concerns the asymptotic behaviour of \eqref{eq.odepert} with $g$ positive incorporates a self--contained section devoted to all relevant properties of regular variation used in these works.  

\section{Asymptotic Behaviour for Ordinary Differential Equations with Internal Perturbations} 
\subsection{Main result and discussion} 
In this section, we deduce the asymptotic behaviour of the ordinary differential equation 
\begin{equation} \label{eq.odeinternal}
x'(t)=-f(x(t)+\gamma(t)), \quad t>0; \quad x(0)=\xi.
\end{equation}
We demonstrate that when the ``internal'' perturbation $\gamma$ decays to zero so rapidly that 
\begin{equation} \label{eq.gammadivFinv}
\lim_{t\to\infty} \frac{\gamma(t)}{F^{-1}(t)}=0,
\end{equation}
and the solution of \eqref{eq.odeinternal} tends to zero as $t\to\infty$, the asymptotic behaviour of \eqref{eq.ode} is preserved.
\begin{theorem} \label{Thm:Lim}
Let $\gamma$ be continuous and $x$ be the continuous solution of \eqref{eq.odeinternal}. Suppose that  
\begin{align} \label{asym}
\text{There exists $\phi$ such that }
\lim_{x \to 0}\frac{f(x)}{\phi(x)}=1, \quad \text{$\phi$ is odd on $\mathbb{R}$}.
\end{align}
and 
\begin{align} \label{RVat0}
f \in RV_0(\beta), \quad \beta>0, \quad \lim_{x \to 0^+}\frac{f(x)}{x}=0, \quad f(0)=0.
\end{align}
with $\beta>1$, $\gamma$ obeys \eqref{eq.gammadivFinv} and that $\lim_{t \to \infty}x(t)=0$. Then 
\begin{align*}
\lim_{t \to \infty} \frac{|x(t)|}{F^{-1}(t)} = 0 \text{ or } 1.
\end{align*}
\end{theorem}
This proof is perhaps of independent interest, as it addresses the situation where the autonomous differential equation \eqref{eq.ode} is perturbed inside 
the argument of $f$ (as opposed to the more commonly studied external perturbation seen in \eqref{eq.odepert}, for example). However, it transpires that 
studying \eqref{eq.odeinternal} and employing Theorem~\ref{Thm:Lim} gives very useful information about the solution of both equations \eqref{eq.odepert} and 
\eqref{eq.sde}, and allows them to be analysed in a form which greatly facilitates the proof of necessary and sufficient conditions for preserving the  asymptotic behaviour of \eqref{eq.ode}.

\subsection{Application of Theorem~\ref{Thm:Lim} to \eqref{eq.odepert} and \eqref{eq.sde}}
We now explore how Theorem~\ref{Thm:Lim} can be applied to determine sufficient conditions for certain asymptotic decay in \eqref{eq.odepert} and \eqref{eq.sde}. Consider first the solution $x$ of \eqref{eq.odepert} which we suppose obeys $x(t)\to 0$ as $t\to\infty$. Introduce the function $u(t)=\int_0^t g(s)\,ds$ and assume that it tends to a finite limit as $t\to\infty$, which we call $u(\infty)$. We are therefore free to define $\gamma(t)=u(t)-u(\infty)$ for $t\geq 0$. Clearly, $\gamma$ is continuous and obeys $\gamma(t)\to 0$ as $t\to\infty$. Of course, $u'(t)=g(t)$. Consider now $z(t)=x(t)-u(t)+u(\infty)=x(t)-\gamma(t)$ for $t\geq 0$. Then $z$ is in $C^1((0,\infty);\mathbb{R})$ and we have that $z(t)\to 0$ as $t\to\infty$. Then $z(0)=\xi+\int_0^\infty g(s)\,ds=:\xi'$ and 
\[
z'(t)=x'(t)-u'(t)=-f(x(t))=-f(z(t)+\gamma(t)), \quad t\geq 0.
\]
Therefore, we see that if $\gamma(t)=\int_t^\infty g(s)\,ds$ obeys \eqref{eq.gammadivFinv}, we can apply Theorem~\ref{Thm:Lim} to $z$ to obtain $z(t)/F^{-1}(t)\to \lambda\in \{0,\pm1\}$ as $t\to\infty$. Then, as $\gamma$ obeys \eqref{eq.gammadivFinv}, we have that $x(t)=z(t)+\gamma(t)$ obeys $x(t)/F^{-1}(t)\to \lambda\in \{0,\pm1\}$ as $t\to\infty$. Therefore, we have established the following result.
\begin{theorem} \label{thm.detpressuff}
Suppose that $f$ is continuous, and obeys \eqref{asym} and \eqref{RVat0} for $\beta>1$. Let $g$ be a continuous function such that  
\begin{equation} \label{eq.intgdivF}
\lim_{t\to\infty} \int_0^t g(s)\,ds \text{ exists and is finite}, \quad 
\lim_{t\to\infty} \frac{\int_t^\infty g(s)\,ds}{F^{-1}(t)}=0. 
\end{equation} 
If the continuous solution $x$ of \eqref{eq.odepert} obeys $\lim_{t\to\infty} x(t)=0$, then 
\begin{equation} \label{eq.xdetperservasy}
\text{There exists a $\lambda\in \{0,\pm1\}=1$ such that } \lim_{t\to\infty} \frac{x(t)}{F^{-1}(t)}=\lambda.
\end{equation}
\end{theorem}
We will see shortly that in the conditions $u(t)$ tends to a finite limit and $(u(\infty)-u(t))/F^{-1}(t)\to 0$ are not only sufficient to ensure the appropriate asymptotic behaviour, but are also necessary. 

We remark in the case that $g(t)$ is ultimately of one sign that the hypothesis $\lim_{t\to\infty} x(t)=0$ is unnecessary, because the first part of \eqref{eq.intgdivF} implies that $g$ is integrable, which suffices to prove under the other hypotheses, that $\lim_{t\to\infty} x(t)=0$.

We show also how Theorem~\ref{Thm:Lim} can assist in determining the asymptotic behaviour of \eqref{eq.sde}. We suppose that $\sigma\in L^2(0,\infty)$. In this case, by the martingale convergence theorem, it follows that the process 
\[
U(t)=\int_0^t \sigma(s)\,dB(s)
\]    
tends to a finite limit almost surely: we call this limit $U(\infty)$. Suppose that this occurs on the (a.s.) event $\Omega_1$; moreover, $t\mapsto U(t,\omega)$ can be taken to be continuous on this event. Let $\Omega_2$ be the a.s. event on which there is a well--defined continuous adapted process $X$ which solves \eqref{eq.sde}. Since $\sigma\in L^2(0,\infty)$, it is well--known that $X(t)\to 0$ as $t\to\infty$ a.s., and denote by $\Omega_3$ the almost sure event on which this convergence occurs. Now set $\Omega_4=\Omega_1\cap\Omega_2\cap\Omega_3$, which is also a.s. Consider 
$V(t)=X(t)-U(t)$ for $t\geq 0$, which is well--defined on $\Omega_4$. Then $V$ obeys
\[
V(t)=\xi -\int_0^t f(X(s))\,ds, \quad t\geq 0.
\]   
Since $f$ is continuous and $t\mapsto X(t)$ is continuous on $\Omega_4$, it follows that for each fixed outcome $\omega\in \Omega_4$ we have that $t\mapsto V(t,\omega)$ is in $C^1((0,\infty);\mathbb{R})$ and in fact
\[
V'(t,\omega)=-f(X(t,\omega)), \quad t\geq 0; \quad V(0,\omega)=\xi.
\]
Now, for each $\omega\in\Omega_4$, define $\gamma(t,\omega)=U(\infty,\omega)-U(t,\omega)$. Then, it follows that $\gamma(t,\omega)\to 0$ as $t\to\infty$ and 
that $t\mapsto\gamma(t,\omega)$ is continuous. Finally, for $\omega\in \Omega_4$, define $Z(t,\omega)=X(t,\omega)-U(t,\omega)+U(\infty,\omega)=X(t,\omega)-\gamma(t,\omega)$ for $t\geq 0$.
Then $Z(t,\omega)  \to 0$ as $t\to\infty$. Furthermore, because we can view $Z(t,\omega)=V(t,\omega)+U(\infty,\omega)$, we have that $t\mapsto Z(t,\omega)$ 
is in $C^1((0,\infty);\mathbb{R})$ and moreover 
\begin{align*}
Z'(t,\omega)&=V'(t,\omega)=-f(X(t,\omega))=-f(Z(t,\omega)+\gamma(t,\omega)), \quad t\geq 0; \\
 Z(0,\omega)&=\xi+\left(\int_0^\infty \sigma(s)\,dB(s)\right)(\omega)=:\xi'(\omega).
\end{align*}
Once again, we see that $t\mapsto Z(t,\omega)$ and $t\mapsto\gamma(t,\omega)$ obeys all conditions of Theorem~\ref{Thm:Lim}, provided that 
\[
\int_t^\infty \sigma(s)\,dB(s)=U(\infty)-U(t)
\]
obeys
\begin{equation} \label{eq.intsigdivF}
\lim_{t\to\infty} \frac{\int_t^\infty \sigma(s)\,dB(s)}{F^{-1}(t)}=0, \quad\text{a.s.}
\end{equation}
Suppose that this last limit is true on the a.s. event $\Omega_5$, and let $\Omega_6=\Omega_5\cap \Omega_4$. In that case, we have that $Z(t,\omega)/F^{-1}(t)\to \lambda(\omega)$ as $t\to\infty$, where $\lambda(\omega)\in \{0,\pm1\}$ for each $\omega\in \Omega_6$. Also, 
since $\gamma(t,\omega)/F^{-1}(t)\to 0$ as $t\to\infty$ for all $\omega\in \Omega_6$, we have that $X(t,\omega)/F^{-1}(t)\to \lambda(\omega)$ as $t\to\infty$ for every $\omega\in \Omega_6$. Since $\Omega_6$ is an almost sure event, we have that $X(t)/F^{-1}(t)\to \lambda$ as $t\to\infty$ a.s., where $\lambda$ must be 
a $\mathcal{F}^B(\infty)$--measurable random variable for which $\mathbb{P}[\lambda\in \{0,\pm1\}]=1$. Accordingly, we see that the following result has been established.
\begin{theorem}  \label{thm.stochpressuff}
Suppose that $f$ is continuous, and obeys \eqref{asym} and \eqref{RVat0} for $\beta>1$. Let $\sigma$ be continuous, in $L^2([0,\infty);\mathbb{R})$ and obey \eqref{eq.intsigdivF}. Then 
\begin{gather} \label{eq.Xperservasy}
\text{There exists an $\mathcal{F}^B(\infty)$--measurable random variable $\lambda$ such that} \nonumber \\
\mathbb{P}[\lambda\in \{0,\pm1\}]=1 \text{ and }
\mathbb{P}\left[\lim_{t\to\infty} \frac{X(t)}{F^{-1}(t)}=\lambda\right]=1.
\end{gather}
\end{theorem}
Once again, we return later to establish that if the solution of \eqref{eq.sde} obeys \eqref{eq.Xperservasy}, then it must be the case that $\sigma\in L^2(0,\infty)$ and obeys \eqref{eq.intsigdivF}. 
 
\section{Main Results for Perturbed ODE}
In this section, we list the main results of the paper. We start with analysis for the deterministic equation \eqref{eq.odepert}, and then consider the stochastic 
equation \eqref{eq.sde}. In each case, we show that the sufficient conditions under which the perturbed equations inherit the asymptotic behaviour of \eqref{eq.ode} are also necessary. We also present results which concern the asymptotic behaviour of the derivative or increment of solutions of the perturbed equation.

A converse of Theorem~\ref{thm.detpressuff} requires that \eqref{eq.xdetperservasy} implies \eqref{eq.intgdivF}. We prove first that \eqref{eq.xdetperservasy}
implies that 
\begin{equation} \label{eq.intgfinite}
\lim_{t\to\infty} \int_0^t g(s)\,ds \text{ exists and is finite}.
\end{equation}
Notice that this is a strictly weaker condition than requiring that $g$ be absolutely integrable.
\begin{theorem} \label{Thm1}
Suppose that $f$ is continuous, and obeys \eqref{asym} and \eqref{RVat0} for $\beta>1$. Let $g$ be continuous. 
If the continuous solution $x$ of \eqref{eq.odepert} obeys \eqref{eq.xdetperservasy}, then $g$ obeys \eqref{eq.intgfinite}.
\end{theorem}

Given that \eqref{eq.intgfinite} (which is the first part of \eqref{eq.intgdivF}) holds when $x$ obeys \eqref{eq.xdetperservasy}, we can define 
\[
\int_t^\infty g(s)\,ds := \lim_{T\to\infty}\int_0^T g(s)\,ds -\int_{0}^t g(s)\,ds, \quad t\geq 0.
\]
We now show that if $x$ obeys \eqref{eq.xdetperservasy}, $t\mapsto \int_t^\infty g(s)\,ds$ must obey both parts of \eqref{eq.intgdivF}.
\begin{theorem}  \label{thm.detpresnecc}
Suppose that $f$ is continuous, and obeys \eqref{asym} and \eqref{RVat0} for $\beta>1$. Let $g$ be continuous. 
If the continuous solution $x$ of \eqref{eq.odepert} obeys \eqref{eq.xdetperservasy}, then $g$ obeys \eqref{eq.intgdivF}.
\end{theorem}

Combining the results of Theorem~\ref{thm.detpressuff} and~\ref{thm.detpresnecc}, we arrive at the following result. 
\begin{theorem} \label{thm.xdetasyiff}
Suppose that $f$ is continuous, and obeys \eqref{asym} and \eqref{RVat0} for $\beta>1$. Let $g$ be continuous. 
If the continuous solution $x$ of \eqref{eq.odepert} obeys $\lim_{t\to\infty} x(t)=0$. Then the following are equivalent:
\begin{itemize}
\item[(a)] The function $g$ obeys \eqref{eq.intgdivF};
\item[(b)] $x$ obeys \eqref{eq.xdetperservasy}; 
\end{itemize}
\end{theorem}
We next consider the situation where the solution and derivative of \eqref{eq.odepert} both inherit their asymptotic behaviour from the solution of \eqref{eq.ode}. 
\begin{theorem} \label{thm.xxprdetasyiff}
Suppose that $f$ is continuous, and obeys \eqref{asym} and \eqref{RVat0} for $\beta>1$. Let $g$ be continuous. 
If the continuous solution $x$ of \eqref{eq.odepert} obeys $\lim_{t\to\infty} x(t)=0$. Then the following are equivalent:
\begin{itemize}
\item[(a)] The function $g$ obeys 
\begin{equation} \label{eq.gdivfF}
\lim_{t\to\infty} \frac{g(t)}{f(F^{-1}(t))}=0.
\end{equation}
\item[(b)] There exists $\lambda\in \{-1,0,1\}$ such that 
\begin{equation} \label{eq.xxprdivFfF}
\lim_{t\to\infty} \frac{x(t)}{F^{-1}(t)}=\lambda, \quad \lim_{t\to\infty}\frac{x'(t)}{f(F^{-1}(t))}=-\lambda.
\end{equation}
\end{itemize}
\end{theorem}
\begin{proof}
The proof is easy and we present it here. We show first that (a) implies (b). \eqref{eq.gdivfF} implies that $g$ is in $L^1([0,\infty);\mathbb{R})$ so the first part of \eqref{eq.intgdivF} holds.
By L'H\^opital's rule, \eqref{eq.gdivfF} implies the second part of \eqref{eq.intgdivF}. Therefore the first part of the limit in \eqref{eq.xxprdivFfF} is
 true, by Theorem~\ref{thm.detpressuff}. In the proof of Theorem~\ref{thm.detpresnecc} it was shown that 
 \[
 \lim_{t\to\infty} \frac{x(t)}{F^{-1}(t)}=\lambda
 \]
 for $\lambda=\pm 1,0$ implies
 \[
 \lim_{t\to\infty} \frac{\varphi(x(t))}{\varphi(F^{-1}(t))}=\lambda.
 \]
where $\varphi$ is the function asymptotic to $f$ introduced in Lemma ~\ref{asym_odd}.
Since $f(x)/\varphi(x)\to 1$ as $x\to 0$, and $x(t)\to 0$ as $t\to\infty$, it follows that  
  \[
 \lim_{t\to\infty} \frac{f(x(t))}{f(F^{-1}(t))}=\lambda.
 \]
 Taking limits in \eqref{eq.odepert}, the last limit and \eqref{eq.gdivfF} yield the second part of \eqref{eq.xxprdivFfF}, as claimed. 
 
 To prove that (b) implies (a), simply rearrange \eqref{eq.odepert} to get
 \begin{equation} \label{eq.gdivfFsetup}
 \frac{g(t)}{f(F^{-1}(t))} = \frac{x'(t)}{f(F^{-1}(t))} + \frac{f(x(t))}{f(F^{-1}(t))}.
 \end{equation}
 By hypothesis, \eqref{eq.xxprdivFfF} holds, so by the argument above, the second term on the right--hand side of \eqref{eq.gdivfFsetup} obeys
 \[
 \lim_{t\to\infty} \frac{\varphi(x(t))}{\varphi(F^{-1}(t))}=\lambda.
 \]
 The first term on the right--hand side of \eqref{eq.gdivfFsetup} has limit $-\lambda$ by hypothesis, so inserting these limits into \eqref{eq.gdivfFsetup} yields 
 \eqref{eq.gdivfF}, as required.
\end{proof}
The pointwise condition \eqref{eq.gdivfF} was used in \cite{apppatt} to obtain the first part of \eqref{eq.xxprdivFfF}; it is a sharp condition, in the sense that if the limit in \eqref{eq.gdivfF} exists and is non--zero, the asymptotic behaviour in the first part of \eqref{eq.xxprdivFfF} does not hold. However, in 
the case that the limit does not exist, it can still be the case that the first part of \eqref{eq.xxprdivFfF} holds, as the condition \eqref{eq.intgdivF} 
(which is of course implied by \eqref{eq.gdivfF}) can be true even when \eqref{eq.gdivfF} is violated. However, Theorem~\ref{thm.xxprdetasyiff} reveals the true significance of the condition \eqref{eq.gdivfF}: it is the critical size of the perturbation $g$ that is allowed in order for the solution of \eqref{eq.odepert}
to be sufficiently well--behaved asymptotically that it inherits the appropriate rate of decay by virtue of the fact that its derivative is well--behaved. Naturally, such a stipulation places greater restrictions on the pointwise behaviour of $g$, by virtue of the form of \eqref{eq.odepert}.

It is tacit in this last statement that the second part of \eqref{eq.xxprdivFfF} drives the behaviour of $x$: in fact, it is easily seen by 
L'H\^opital's rule that the second part \eqref{eq.xxprdivFfF} implies the first.  

We finish this section by noting in the case when $g$ is positive and the initial condition is positive (so that $x(t)>0$ for all $t\geq 0$), the limit in 
\eqref{eq.xdetperservasy} is unity. 
\begin{theorem} \label{thm.limis1gpos}
Suppose that $f$ is continuous and is in $\text{RV}_0(\beta)$ for $\beta>1$. Suppose further that $g$ is continuous and positive. 
If $x$ is the continuous solution \eqref{eq.odepert} with $x(0)=\xi>0$, then 
the following are equivalent
\begin{itemize}
\item[(a)] The function $g$ obeys \eqref{eq.intgdivF};
\item[(b)] $x$ obeys 
\[
\lim_{t\to\infty} \frac{x(t)}{F^{-1}(t)}=1.
\]
\end{itemize} 
\end{theorem}
\begin{proof}
Since $x(t)>0$ for all $t\geq 0$, we do not need to assume that $f$ is asymptotically odd, as in other results in this section. 
If (a) holds, then as $g$ is positive, we have that $g\in L^1([0,\infty);(0,\infty))$. Therefore, we have that $x(t)\to 0$ as $t\to \infty$. Therefore, by 
Theorem~\ref{thm.detpressuff} and the fact that $x$ is positive, we have that 
\[
\lim_{t\to\infty} \frac{x(t)}{F^{-1}(t)}=0 \text{ or $1$}.
\]
On the other hand, define $z'(t)=-f(z(t))$ for $t\geq 0$ and $z(0)=x(0)/2>0$. Then $x(t)>z(t)$ for all $t\geq 0$. Integration yields that $F(z(t))/t\to 1$ as $t\to\infty$, and as $F^{-1}$ is regularly varying, it follows that $z(t)/F^{-1}(t)\to 1$ as $t\to\infty$. Therefore we have that 
\[
\limsup_{t\to\infty} \frac{x(t)}{F^{-1}(t)}\geq 1,
\]
and this forces $x$ to obey (b). If (b) holds, we have that $x(t)\to 0$ as $t\to\infty$, and since all other hypotheses of Theorem~\ref{thm.detpresnecc} are true, we have that (a) holds. 
\end{proof}
Finally, we prove a result concerning the global stability of solutions of \eqref{eq.odepert}, a hypothesis which we require in many of the above results. It is to be noted that to achieve this, no additional conditions are required of $f$ close to the equilibrium of \eqref{eq.ode}. Instead, we require some asymptotic control of $f$ at infinity. The condition we use is 
\begin{align} \label{fatinfinity}
\liminf_{x \to +\infty}|f(x)| >0, \quad \liminf_{x \to -\infty} |f(x)| >0. 
\end{align}
and this was employed in \cite{JAJGAR:2009} to cover the case of equations of the form \eqref{eq.odepert} in which the perturbation $g$ obeys $g(t)\to 0$ as $t\to\infty$. We remark that examples in \cite{JAJGAR:2009} and \cite{JAJC:2011szeged} show that if this condition is violated, it can happen that solutions of \eqref{eq.odepert} tend to $\pm\infty$ as $t\to\infty$. Hence, we can see that this condition is not excessively restrictive. We are of course free to postulate alternative sufficient conditions under which $x(t)\to 0$ 
as $t\to\infty$, so as to allow the use of the above theorems; however, we prefer to separate hypotheses which ensure convergence 
from those explicitly required to preserve exact rates of decay. 

As usual, in the following theorem, we assume without explicitly saying that $f$ obeys \eqref{eq.fglobalstable}. We assume in this result that $f$ is locally Lipschitz continuous, as this enables us to simplify the argument: however, we conjecture that with more effort it is possible to require only that $f$ is continuous. In this case, we may not have uniqueness of continuous solutions, but all solutions $x$ should obey $x(t)\to 0$ as $t\to\infty$.
\begin{theorem}  \label{thm.xto0suff}
Let $f$ be locally Lipschitz continuous and suppose \eqref{eq.intgfinite} holds. Suppose further that \eqref{fatinfinity} holds.
Then the unique, continuous solution $x$ of \eqref{eq.odepert} obeys $x(t) \rightarrow 0$ as $t \to \infty$.
\end{theorem}

\section{Main results for SDEs}
We now present the main results for solutions of \eqref{eq.sde}. 
We have already shown in Theorem~\ref{thm.stochpressuff} that
\begin{equation} \label{eq.intsig2divF}
\sigma\in L^2([0,\infty);\mathbb{R}), \quad \lim_{t\to\infty} \frac{\int_t^\infty \sigma(s)\,dB(s)}{F^{-1}(t)}=0,\quad \text{a.s.}
\end{equation}
imply \eqref{eq.Xperservasy}. We now establish the converse to this result, along the lines used to prove the converse of Theorem~\ref{thm.detpressuff} in 
the deterministic case. Firstly, we prove that \eqref{eq.Xperservasy} implies the first condition in \eqref{eq.intsig2divF}, namely $\sigma\in L^2([0,\infty);\mathbb{R})$. 
\begin{lemma} \label{lemsig2}
Suppose that $f$ is continuous, and obeys \eqref{asym} and \eqref{RVat0} for $\beta>1$. Let $\sigma$ be continuous. 
If the continuous adapted process $X$ which obeys \eqref{eq.sde} also satisfies \eqref{eq.Xperservasy}, then $\sigma\in L^2([0,\infty);\mathbb{R})$.
\end{lemma}
\begin{proof}
Writing \eqref{eq.sde} in integral form and rearranging yields
\begin{equation} \label{eq.sigL2rep}
\int_0^t \sigma(s)\,dB(s)=X(t)-X(0)+\int_0^t f(X(s))\,ds, \quad t\geq 0.
\end{equation}
We have been granted as a hypothesis that $X$ obeys \eqref{eq.Xperservasy}: this implies that $X(t)\to 0$ as $t\to\infty$ a.s. Therefore, if it can be 
shown that the last term on the right--hand side of \eqref{eq.sigL2rep} tends to a finite limit a.s. as $t\to\infty$, we have that 
\[
\lim_{t\to\infty} \int_0^t \sigma(s)\,dB(s) \text{ exists and is finite, a.s.}
\]
By virtue of the martingale convergence theorem, this forces $\sigma$ to be in $ L^2([0,\infty);\mathbb{R})$. Hence it is enough to prove that 
\[
\lim_{t\to\infty} \int_0^t f(X(s))\,ds \text{ exists and is finite a.s.}
\]
under the hypothesis \eqref{eq.Xperservasy}. However, this can be established by employing pathwise (i.e., to each $\omega$ in the almost sure event 
for which \eqref{eq.Xperservasy} holds and for which $\lambda(\omega)=0, 1$, or $-1$) the argument used to prove the convergence of the integral 
\[
\int_0^t f(x(s))\,ds
\] 
in the proof of Theorem~\ref{Thm1}, under the hypothesis that the function $x$ obeys \eqref{eq.xdetperservasy}. 
\end{proof}
Next we show that the second part of \eqref{eq.intsig2divF} is necessary if we are to have that the solution $X$ of \eqref{eq.sde} obeys \eqref{eq.Xperservasy}.
Since $\sigma\in L^2([0,\infty);\mathbb{R})$, we have that 
\[
\lim_{T\to\infty} \int_0^T \sigma(s)\,dB(s) \text{ exists and is finite a.s.}
\] 
Therefore, for each outcome $\omega$ in an a.s. event $\Omega_0$ we can define 
\begin{multline}\label{def.tailmart}
\left(\int_t^\infty \sigma(s)\,dB(s)\right)(\omega)
\\:= \lim_{T\to\infty} \left(\int_0^T \sigma(s)\,dB(s)\right)(\omega) - \left(\int_0^t \sigma(s)\,dB(s)\right)(\omega), \quad t\geq 0.
\end{multline}
It is more useful to view this as an (uncountably) infinite family of $\mathcal{F}^B(\infty)$--measurable random variables that are well--defined on 
$\Omega_0$, rather than as a conventional process. One reason for this is that $I(t):=\int_t^\infty \sigma(s)\,dB(s)$ is not $\mathcal{F}^B(t)$--measurable, as it depends on the Brownian motion $B$ after time $t$: therefore, $I$ is not adapted to the filtration $(\mathcal{F}^B(t))_{t\geq 0}$ on which the solutions of 
the SDE evolves. However, this is not a limitation, as our arguments can be applied path by path, and therefore can be viewed as essentially deterministic, once 
the outcome $\omega$ has been selected in an a.s. event on which desirable asymptotic properties hold.  
\begin{theorem} \label{theorem:Xnecc1}
Suppose that $f$ is continuous, and obeys \eqref{asym} and \eqref{RVat0} for $\beta>1$. Let $\sigma$ be continuous. 
If the continuous adapted process $X$ which obeys \eqref{eq.sde} also satisfies
\[
\mathbb{P}\left[\lim_{t\to\infty} \frac{X(t,\omega)}{F^{-1}(t)} =\lambda(\omega)\in (-\infty,\infty)\right]=1
\]
then $X$ obeys \eqref{eq.Xperservasy} and $\sigma$ obeys \eqref{eq.intsig2divF}.
\end{theorem}
In order to prove this result, and another result later in this section, the following result is needed concerning the asymptotic behaviour of $\int_t^\infty \sigma(s)\,dB(s)$ when $\sigma\in L^2([0,\infty);\mathbb{R})$.
\begin{lemma} \label{lemma.tailsigmart}
Suppose $\sigma$ is a continuous function such that $\sigma \in L^2([0, \infty);\mathbb{R})$ and
\begin{equation} \label{eq.intsigpos}
\int_t^\infty \sigma^2(s)\,ds > 0 \text{ for all $t\geq 0$},
\end{equation}
then 
\begin{align*}
\limsup_{t \to \infty}\frac{\int_t^\infty \sigma(s)\, dB(s)}
{ \sqrt{2\int_t^\infty \sigma^2(s)\,ds \log\log \left(\frac{1}{\int_t^\infty \sigma^2(s)\,ds}\right)} } = 1, \quad\text{a.s.}
\end{align*}
and
\begin{align*}
\liminf_{t \to \infty}\frac{\int_t^\infty \sigma(s)\,dB(s)}
{ \sqrt{2\int_t^\infty \sigma^2(s)\,ds \log\log \left(\frac{1}{\int_t^\infty \sigma^2(s)\,ds}\right)} } = -1, \quad\text{ a.s.}
\end{align*}
\end{lemma}
The proof of this lemma can be found in \cite{JADJGAR:integral}. The condition \eqref{eq.intsigpos} is important; if it did not hold however, the dynamics of the SDE \eqref{eq.sde} would collapse to those of \eqref{eq.ode}. To see this, notice that if  \eqref{eq.intsigpos} does not hold, then 
there exists a deterministic $T>0$ such that $\int_t^\infty \sigma^2(s)\,ds=0$ for all $t\geq T$. Since $\sigma^2$ is non--negative and continuous, this implies that 
$\sigma(t)=0$ a.e. for $t\in [T,\infty)$ and therefore that 
\[
\int_T^t \sigma(s)\,dB(s)=0 \quad\text{for all $t\in [T,\infty)$ a.s.}
\]
Therefore, for $t\geq T$, \eqref{eq.sde} reads
\[
X(t)=X(T)-\int_T^t f(X(s))\,ds,
\]
so $X'(t)=-f(X(t))$ for $t\geq T$ a.s. with ``initial condition'' $X(T)$ being a random variable. Clearly, we have that 
\[
\lim_{t\to\infty} \frac{X(t)}{F^{-1}(t)}=\sgn(X(T)), \quad \text{a.s.}
\]
so in this case, we have \eqref{eq.Xperservasy}. We therefore tacitly assume that \eqref{eq.intsigpos} holds in future, because otherwise the stochastic 
equation \eqref{eq.sde} is simply an equation of the form \eqref{eq.ode} with a random initial condition.  

In the case that \eqref{eq.intsigpos} holds, we see that the function 
\[
\Sigma(t)=\sqrt{2\int_t^\infty \sigma^2(s)\,ds \log\log \left(\frac{1}{\int_t^\infty \sigma^2(s)\,ds}\right)}
\]
is positive for all $t$ sufficiently large and that $\Sigma(t)\to 0$ as $t\to\infty$. Therefore, by Lemma~\ref{lemma.tailsigmart}, because 
\[
\liminf_{t\to\infty} \frac{\int_t^\infty \sigma(s)\,dB(s)}{\Sigma(t)}=-1, \quad \limsup_{t\to\infty} \frac{\int_t^\infty \sigma(s)\,dB(s)}{\Sigma(t)}=1,
\quad \text{a.s}
\]
then as $F^{-1}(t)>0$ for all $t\geq 0$ and $F^{-1}(t)\to 0$ as $t\to\infty$, we have
\begin{equation} \label{eq.zeroeventSig}
\mathbb{P}\left[
\lim_{t\to\infty} \frac{\int_t^\infty \sigma(s)\,dB(s)}{F^{-1}(t)} \text{ exists and is finite and non--zero}\right]=0.
\end{equation}

Combining the results of Theorem~\ref{thm.stochpressuff} and~\ref{theorem:Xnecc1} we obtain the following result. 
\begin{theorem} \label{thm.Xneccsuff}
Suppose that $f$ is continuous, and obeys \eqref{asym} and \eqref{RVat0} for $\beta>1$. Let $\sigma$ be continuous. 
Suppose that $X$ is the continuous adapted process $X$ which obeys \eqref{eq.sde}. Then the following are equivalent:
\begin{itemize}
\item[(a)] $\sigma\in L^2([0,\infty);\mathbb{R})$ and 
\[
\lim_{t\to\infty} \frac{\int_t^\infty \sigma(s)\,dB(s)}{F^{-1}(t)}=0, \quad\text{a.s.}
\] 
\item[(b)] 
\[
\mathbb{P}\left[\lim_{t\to\infty} \frac{X(t)}{F^{-1}(t)}\in (-\infty,\infty) \right]=1;
\]
\item[(c)]
\[
\mathbb{P}\left[ \lim_{t\to\infty} \frac{X(t)}{F^{-1}(t)}\in \{-1,0,1\} \right]=1.
\]
\end{itemize}
\end{theorem}
We have shown that (a) implies (c) in Theorem~\ref{thm.stochpressuff}); (c) clearly implies (b); and by Theorem~\ref{theorem:Xnecc1}, (b) implies (a).

The second condition in \eqref{eq.intsig2divF} is difficult to check a priori. Instead, we may use Lemma~\ref{lemma.tailsigmart} to arrive 
at a more direct theorem, contingent on the following additional assumption on $\sigma$; 
\begin{equation} \label{def.mu}
\text{There exists $\mu\in [0,\infty]$ such that }
\mu^2:= \lim_{t \to \infty} \frac{2\int_t^\infty \sigma^2(s)\,ds \log\log\left(\frac{1}{\int_t^\infty \sigma^2(s)\,ds}\right) }{ F^{-1}(t)^2 }. 
\end{equation}
We state this result now.
\begin{theorem}  \label{thm.Xneccsuff2}
Suppose that $f$ is continuous, and obeys \eqref{asym} and \eqref{RVat0} for $\beta>1$. Let $\sigma$ be continuous. 
Suppose that $X$ is the continuous adapted process $X$ which obeys \eqref{eq.sde}. 
\begin{itemize}
\item[(a)] Suppose $\sigma \notin L^2([0, \infty);\mathbb{R})$. Then
\begin{align*}
\mathbb{P}\left[\lim_{t \to \infty}\frac{X(t)}{F^{-1}(t)}\in (-\infty, \infty)\right]=0.
\end{align*}
\item[(b)] Suppose $\sigma \in L^2([0, \infty);\mathbb{R})$. 
\begin{enumerate}
\item[(i)] If $\mu$ defined by \eqref{def.mu} is zero, then 
\begin{align*}
\mathbb{P}\left[\lim_{t \to \infty} \frac{X(t)}{F^{-1}(t)} \in \{-1, 0,1\}\right] = 1.
\end{align*} 
\item[(ii)] If $\mu$ defined by \eqref{def.mu} in $(0,\infty]$, then
\begin{align*}
\mathbb{P}\left[ \lim_{t \to \infty} \frac{X(t)}{F^{-1}(t)} \in (-\infty, \infty) \right] = 0.
\end{align*}
\end{enumerate}
\end{itemize}
\end{theorem}
We now turn to the situation in which the solution of \eqref{eq.sde} possesses asymptotic behaviour similar to that of \eqref{eq.ode} by virtue of the scaled $h$--increment 
\[
\frac{X(t+h)-X(t)}{h}
\]
possessing good asymptotic behaviour. In fact, we will give necessary and sufficient conditions for the solution of \eqref{eq.sde} to obey 
\begin{equation} \label{eq.Xasydiff}
\mathbb{P}\left[ \lim_{t\to\infty} \frac{X(t)}{F^{-1}(t)}=\lambda, \quad 
\lim_{t\to\infty} \frac{\frac{X(t+h)-X(t)}{h}}{f(F^{-1}(t))}=-\lambda, \quad \lambda\in \{-1,0,1\}\right]=1 
\end{equation} 
for each $h>0$. It turns out if we let $\Psi$ be the complementary normal distribution function and define 
\begin{equation} \label{def.Sf}
S_f(\epsilon,h):= \sum_{n=0}^\infty  \Psi \left(\frac{\epsilon}{\theta(n)} \right),  \quad \theta^2(n) := \frac{ \int_{nh}^{(n+1)h}{\sigma^2(s)ds} }{ (f \circ F^{-1})^2(nh) },
\end{equation} 
then \eqref{eq.Xasydiff} holds if and only if for a fixed $h>0$, we have $S_f(\epsilon,h)<+\infty$ for all $\epsilon>0$. We first establish the sufficiency of 
the finiteness of $S_f(\epsilon,h)$.
\begin{theorem} \label{thm:Xderiv}
Suppose that $f$ is locally Lipschitz continuous, and obeys \eqref{asym} and \eqref{RVat0} for $\beta>1$. Let $\sigma$ be continuous. 
Let $X$ be the continuous adapted solution of \eqref{eq.sde}. 
Let $h>0$ and define $S_f(\epsilon,h)$ as in \eqref{def.Sf}. If $S_f(\epsilon,h)< +\infty$ for all $\epsilon > 0$, then 
\begin{multline} \label{eq:Xderiv}
\text{There exists a $\mathcal{F}^B(\infty)$--measurable random variable $\lambda$ such that }\\
\lambda\in \{-1,0,1\} \quad \text{a.s.}, \quad \lim_{t \to \infty}\frac{X(t)}{F^{-1}(t)} = \lambda, \quad \text{a.s.}, \quad  
\lim_{t \to \infty}\frac{ \frac{X(t+h)-X(t)}{h} }{(f \circ F^{-1})(t)} = -\lambda, \quad \text{a.s.}
\end{multline}
\end{theorem}
Once this result has been secured, we see that it admits a converse.  
\begin{theorem} \label{thm:Xderivconverse}
Suppose that $f$ is locally Lipschitz continuous, and obeys \eqref{asym} and \eqref{RVat0} for $\beta>1$. Let $\sigma$ be continuous. 
Let $X$ be the continuous adapted solution of \eqref{eq.sde}. Let $h>0$ and define $S_f(\epsilon,h)$ as in \eqref{def.Sf}. 
Suppose that there exists an event $A$ with $\mathbb{P}[A]>0$ such that
\begin{multline} \label{eq:XderivA}
A=\biggl\{\omega: \lim_{t \to \infty}\frac{X(t,\omega)}{F^{-1}(t)} = \lambda(\omega), \\
\lim_{t \to \infty}\frac{ \frac{X(t+h,\omega)-X(t,\omega)}{h} }{(f \circ F^{-1})(t)} = -\lambda(\omega), \,
\lambda(\omega)\in \{-1,0,1\}\biggr\}. 
\end{multline}
 Then $S_f(\epsilon,h)< +\infty$ for all $\epsilon > 0$. 
%
\end{theorem}
We now consolidate the last two theorems to demonstrate the necessary and sufficient conditions under which the increments of the process as well as the process itself, enjoy the same convergence rate to zero as the derivative of the solution of \eqref{eq.ode}, as well as the solution itself.
\begin{theorem} \label{thm:Xneccsuffderiv}
Suppose that $f$ is locally Lipschitz continuous, and obeys \eqref{asym} and \eqref{RVat0} for $\beta>1$. Let $\sigma$ be continuous. 
Let $X$ be the continuous adapted solution of \eqref{eq.sde}. Let $h>0$ and define $S_f(\epsilon,h)$ as in \eqref{def.Sf}.
Then the following are equivalent:
\begin{itemize}
\item[(a)] $S_f(\epsilon,h)<+\infty$ for all $\epsilon>0$;
\item[(b)] There exists an event $A$ with $\mathbb{P}[A]>0$ such that
\begin{multline*} 
A=\biggl\{\omega: \lim_{t \to \infty}\frac{X(t,\omega)}{F^{-1}(t)} = \lambda(\omega), \\
\lim_{t \to \infty}\frac{ \frac{X(t+h,\omega)-X(t,\omega)}{h} }{(f \circ F^{-1})(t)} = -\lambda(\omega), \,
\lambda(\omega)\in \{-1,0,1\}\biggr\}. 
\end{multline*}
\item[(c)] There exists an event $A$ with $\mathbb{P}[A]=1$ such that
\begin{multline*} 
A=\biggl\{\omega: \lim_{t \to \infty}\frac{X(t,\omega)}{F^{-1}(t)} = \lambda(\omega), \\
\lim_{t \to \infty}\frac{ \frac{X(t+h,\omega)-X(t,\omega)}{h} }{(f \circ F^{-1})(t)} = -\lambda(\omega), \,
\lambda(\omega)\in \{-1,0,1\}\biggr\}. 
\end{multline*} 
\end{itemize}
\end{theorem}
The proof that the statements are equivalent is now easy: part (a) implies part (c) by Theorem~\ref{thm:Xderiv}; part (c) trivially implies part (b); and part (b) implies (a) by Theorem~\ref{thm:Xderivconverse}.

\section{Examples}
We present in this section some examples to illustrate the scope of the results. 

We note that all limits are possible in Theorem \eqref{thm.detpressuff}. 
To see this, simply take $x(0)>0$ and $g(t)>0$ for $t\geq 0$; then if $\int_t^\infty g(s)\,ds/F^{-1}(t)\to 0$ as $t\to\infty$ we have that 
$x(t)/F^{-1}(t)\to 1$ as $t\to\infty$ by Theorem~\ref{thm.limis1gpos}. 
If $x(0)<0$, $g(t)<0$ for all $t\geq0$, we can  prove similarly that $x(t)/F^{-1}(t)\to 1$ as $t\to\infty$. 

To show that a zero limit can obtain (and indeed that $x$ can decay to zero arbitrarily slowly), 
suppose that $d(0)=1$ and that $d\in C^1((0,\infty);(0,\infty))$ with $d(t)\to 0$ as $t\to\infty$. 
Suppose that $d$ decays to zero faster than any exponential function by assuming that $d'(t)/d(t)\to -\infty$ as $t\to\infty$. Then 
\[
\lim_{t\to\infty} \frac{d'(t)}{f(F^{-1}(t))}=0.
\]
Hence $d(t)/F^{-1}(t)\to 0$ as $t\to\infty$. Let $\xi>0$ and define 
$g(t)= \xi d'(t)+f(\xi d(t))$ for $t\geq 0$. Then $x(t)=\xi d(t)$ for $t\geq 0$ is the solution of \eqref{eq.odepert}, and we have that 
\[
\lim_{t\to\infty} \frac{x(t)}{F^{-1}(t)}=0.
\]
Moreover, as $f(x)/x\to 0$ as $x\to 0$, we have that $|f(\xi d(t))|<d(t)$ for $t$ sufficiently large. Thus $f(\xi d(t))/d'(t)\to 0$ as $t\to\infty$. 
Hence $g(t)/d'(t)\to \xi$ as $t\to\infty$. Therefore, as $d'\in L^1([0,\infty);\mathbb{R}$ with $d(t)=\int_t^\infty -d'(s)\,ds$, we have that 
\[
\lim_{t\to\infty} \frac{\int_t^\infty g(s)\,ds}{F^{-1}(t)}=0.
\]
as predicted by Theorem~\ref{thm.detpresnecc}.

We start with a lemma which can be used to show that $g$ can obey 
\[
\lim_{t\to\infty} \int_0^t g(s)\,ds \text{ exists and is finite},
\]
without being absolutely integrable, and that this gives rise to less conservative stability conditions. In fact, we will use the lemma to demonstrate that there are perturbations $g$ whose extremes can grow arbitrarily fast as $t\to\infty$, and which change sign infinitely often, but which nevertheless satisfy \eqref{eq.intgdivF}. 
\begin{lemma}  \label{lemma.osyexamp}
Let $k\in C^1((0,\infty);(0,\infty))$ be such that $k\not\in L^1([0,\infty);(0,\infty))$. Suppose also that 
\[
\lim_{t\to\infty} \sup_{0\leq s\leq T} \left| \frac{k(t+s)}{k(t)}-1\right|=0, \quad 
\lim_{t\to\infty} \sup_{0\leq s\leq T} \left| \frac{k'(t+s)}{k'(t)}-1\right|=0.  
\]
Define $k_0(t)=k(t)\sin(t)$ for $t\geq 0$. Then
\begin{itemize}
\item[(i)] $\lim_{t\to\infty} \int_0^t k_0(s)\,ds$ is finite, but $\lim_{t\to\infty}\int_0^t |k_0(s)|\,ds=+\infty$;
\item[(ii)] \[
\limsup_{t\to\infty} \frac{\left|\int_t^\infty k_0(s)\,ds\right|}{k(t)}=1, \quad \liminf_{t\to\infty} \frac{\left|\int_t^\infty k_0(s)\,ds\right|}{k(t)}=0. 
\]
\end{itemize}
\end{lemma} 
We note that if $k$ is a positive, non--integrable function with $-k'$ regularly varying, then all the above conditions hold. 

We can use this lemma to demonstrate that there are perturbations $g$ whose extremes can grow arbitrarily fast as $t\to\infty$, and which change sign infinitely often, but which nevertheless satisfy \eqref{eq.intgdivF}.
\begin{theorem} \label{thm.goscill}
Suppose that $f$ is continuous and obeys \eqref{asym} and \eqref{RVat0} for $\beta>1$. 
Let $\Gamma$ be a function that obeys
\[
\Gamma\in C([0,\infty);(0,\infty)), \quad \lim_{t\to\infty} \Gamma(t)=+\infty.
\]
Then there exists a function $g$ which obeys \eqref{eq.intgdivF}, changes sign infinitely often, and satisfies 
\[
\limsup_{t\to\infty} \frac{|g(t)|}{\Gamma(t)}=1,
\]
and hence the continuous solution $x$ of \eqref{eq.odepert} obeys \eqref{eq.xdetperservasy}.
\end{theorem} 
The proof is deferred to Section 12. We notice that the function $g$ constructed to verify this theorem does not merely oscillate, but does so with increasing frequency as $t\to\infty$. In fact, as $t\to\infty$, the number of sign changes of $g$ in the interval $[t,t+1]$ tends to infinity. This rapid ``self--cancellation'' in $g$ is what accounts for 
the good asymptotic behaviour of $\int_t^\infty g(s)\,ds$. Furthermore, for an appreciable proportion of the time as $t\to\infty$, we have that $|g(t)|>\Gamma(t)/2$, so the periods of extreme behaviour of $g$ are common. 

In the case when $g$ is a positive function which has rapidly growing extremes, we cannot rely on such fortuitous self--cancellation to preserve the asymptotic 
behaviour of the solution of \eqref{eq.ode} in \eqref{eq.odepert}. Instead, we show that while arbitrarily rapidly--growing perturbations $g$ can still preserve the rate of decay, such frequent extreme behaviour should be limited to relatively short intervals of time. In other words, short ``spikes'' in $g$ are still admissible.   

In order to demonstrate this, we start by establishing the following lemma. As often, the proof is deferred to the end.
\begin{lemma}  \label{lemma.gspikes}
Suppose $k_s$ is a positive, $C^1(0, \infty)$ function with 
\[
\lim_{t \to \infty}\int_t^\infty{k_s(s)ds}=0
\]
and $\lim_{t \to \infty}k_s(t)=0$. Suppose also that $\Gamma \in C^1(0,\infty)$ is increasing, with $\Gamma(t)\nearrow \infty$.
Define $\Gamma_+(t) = \Gamma(t)+ \bar{k_s}+1, \, t \geq 0$, where $\bar{k_s} = \sup_{t \geq 0}k_s(t)$ and also define the sequence $\{w_j \}_{j=0}^\infty$ by
\begin{align} \label{w_j}
w_j := \frac{1}{2} \wedge \frac{\int_{j+1}^{j+2}{k_s(u)du}}{\Gamma_+(j+1)}.
\end{align}
Suppose $a>0$ and $b>0$ and consider the function
\begin{center}
$h_s(x,a,b) :=
\begin{cases}
b\left( 1-3(\frac{x-a}{a})^2 -2(\frac{x-a}{a})^3 \right), \, x \in [0,a], \\
h(2a-x,a,b), \, x \in (a,2a].
\end{cases}
$
\end{center}
Then the function defined for $t \in [n,n+1]$, for all $n \geq 0$, by 
\begin{align} \label{spikes}
k(t) :=
\begin{cases}
k_s(t) + h(t-n,\frac{w_n}{2},\Gamma_+(t)-k_s(t)), \, t \in [n,n+w_n), \\
k_s(t), \, t \in [n+w_n,n+1].
\end{cases}
\end{align}
is $C^1(0,\infty)$ and, obeys 
\begin{align} \label{intRatio}
\lim_{t \to \infty}\frac{\int_t^\infty{k(s)ds}}{\int_t^\infty{k_s(s)ds}} = 1.
\end{align}
Furthermore, we note that 
\begin{align} \label{max}
\limsup_{t \to \infty}\frac{k(t)}{\Gamma(t)} = 1.
\end{align}
\end{lemma}
Armed with this result we can now prove the following theorem.
\begin{theorem} \label{thm.gspikes}
Suppose that $f$ is continuous and obeys \eqref{asym} and \eqref{RVat0} for $\beta>1$. 
Let $\Gamma$ be a function that obeys
\[
\Gamma\in C^1([0,\infty);(0,\infty)), \quad \lim_{t\to\infty} \Gamma(t)=+\infty.
\]
Then there exists a function $g\in C^1((0,\infty);(0,\infty))$ which obeys \eqref{eq.intgdivF} and satisfies 
\[
\limsup_{t\to\infty} \frac{|g(t)|}{\Gamma(t)}=1,
\]
and hence the continuous solution $x$ of \eqref{eq.odepert} obeys \eqref{eq.xdetperservasy}.
\end{theorem}
\begin{proof}
Let $k_s(t)=(\varphi \circ \Phi^{-1}(t))/(1+t)$ for $t\geq 0$. Notice that $k_s \in C^1((0,\infty);(0,\infty))$ tends to zero. In fact $k_s(t)/(f\circ F^{-1})(t)\to 0$ as $t\to\infty$, so by L'H\^opital's rule we get
\[
\lim_{t\to\infty} \frac{\int_t^\infty k_s(u)\,du}{F^{-1}(t)}=0.
\]
Given this function $k_s$, by Lemma~\ref{lemma.gspikes} there is a positive and $C^1$ function $k$ defined by \eqref{spikes} which additionally satisfies
\[
\lim_{t\to\infty}\frac{\int_t^\infty k(u)\,du}{\int_t^\infty k_s(u)\,du} =1, \quad \limsup_{t\to\infty} \frac{k(t)}{\Gamma(t)}=1.
\]
Now let $g(t)=k(t)$ for all $t\geq 0$, so that $g$ obeys \eqref{eq.intgdivF} and $\limsup_{t\to\infty} g(t)/\Gamma(t)=1$, as required, and hence Theorem~\ref{thm.detpressuff} applies to the solution $x$ of \eqref{eq.odepert}, as claimed.
\end{proof}
We can prove a similar result for the stochastic differential equation.
\begin{theorem} \label{thm.sig2spikes}
Suppose that $f$ is continuous and obeys \eqref{asym} and \eqref{RVat0} for $\beta>1$. 
Let $\Gamma$ be a function that obeys
\[
\Gamma\in C^1([0,\infty);(0,\infty)), \quad \lim_{t\to\infty} \Gamma(t)=+\infty.
\]
Then there exists a function $\sigma^2\in C^1((0,\infty);(0,\infty))$ which obeys \eqref{def.mu} with $\mu=0$ and satisfies 
\[
\limsup_{t\to\infty} \frac{\sigma^2(t)}{\Gamma(t)}=1,
\]
and hence the continuous adapted process $X$ which obeys  \eqref{eq.sde} satisfies the conclusions of Theorem~\ref{thm.stochpressuff}.
\end{theorem}
\begin{proof}
Let $\gamma=1/(\beta-1)+1/2$. Let $k_s(t)=(1+t)^{-2\gamma-\epsilon}$ for $t\geq 0$. Notice that $k_s \in C^1((0,\infty);(0,\infty))$ tends to zero and $k_s\in L^1([0,\infty);\mathbb{R})$. 
We have that $\int_t^\infty k_s(u)\,du\in \text{RV}_{\infty}(-2\gamma+1-\epsilon)$. Hence 
\[
t\mapsto \int_t^\infty k_s(u)\,du \log\log\left(\frac{1}{\int_t^\infty k_s(u)\,du}\right)\in\text{RV}_{\infty}(-2\gamma+1-\epsilon). 
\]
Given this function $k_s$, by Lemma~\ref{lemma.gspikes} there is a positive and $C^1$ function $k$ defined by \eqref{spikes} which additionally satisfies
\[
\lim_{t\to\infty}\frac{\int_t^\infty k(u)\,du}{\int_t^\infty k_s(u)\,du} =1, \quad \limsup_{t\to\infty} \frac{k(t)}{\Gamma(t)}=1.
\]
Now let $\sigma^2(t)=k(t)$ for all $t\geq 0$. Then we have that 
\[
t\mapsto \int_t^\infty \sigma^2(u)\,du \log\log\left(\frac{1}{\int_t^\infty \sigma^2(u)\,du}\right)\in\text{RV}_{\infty}(-2\gamma+1-\epsilon)
\]
Since $F^{-1}\in \text{RV}_\infty(-1/(\beta-1))$, we have that \eqref{def.mu} holds with $\mu=0$, because $2\gamma-1+\epsilon>2/(\beta-1)$. 
Moreover, $\limsup_{t\to\infty} \sigma^2(t)/\Gamma(t)=1$, as required, and hence Theorem~\ref{thm.stochpressuff} applies to the solution $X$ of \eqref{eq.sde}, as claimed.
\end{proof}

In many cases it is straightforward to determine the asymptotic behaviour of differential equations directly, because the asymptotic behaviour of 
$F^{-1}$ can be determined. The following result gives an easily--checked and sufficient condition on $f$ under which the asymptotic behaviour of $F^{-1}$ can be read off.

\begin{proposition} \label{prop.Finvasy}
Suppose that $f\in \text{RV}_0(\beta)$ is continuous and $\beta>1$. Define 
\[
\ell(x)=\left(\frac{f(x)}{x^\beta}\right)^{-1/(\beta-1)}, 
\] 
and assume that 
\begin{equation} \label{eq.ell}
\lim_{x\to 0}\frac{\ell(x\ell(x))}{\ell(x)}=1.
\end{equation}
If $F$ is defined by \eqref{def.F},  
\[
F(x)\sim \frac{1}{\beta-1}\frac{x}{f(x)}, 
\quad\text{as $x\to 0^+$},
\]
and 
\begin{equation} \label{eq.Finvasy}
F^{-1}(t)\sim \left(\frac{1}{\beta-1}\right)^{1/(\beta-1)} t^{-1/(\beta-1)} \ell(t^{-1/(\beta-1)}), \quad\text{as $t\to\infty$}.
\end{equation}
\end{proposition}
\begin{proof}
The proof is not hard and introduces useful notation for the rest of this section, so we give it here.  
Define $l(x)=f(x)/x^\beta$. Then $\ell(x)=l(x)^{-1/(\beta-1)}$. 
Since $f\in \text{RV}_0(\beta)$ for $\beta>1$ it follows that $l$ and $\ell$ are both in $\text{RV}_0(0)$. 
The asymptotic behaviour of $F$ is well--known. Since  $f(x)/x^\beta=l(x)$, we have that $1/l(x)= \ell(x)^{\beta-1}$ as $x\to 0^+$, and so 
it is true that 
\[
F(x)\sim \frac{1}{\beta-1}x^{1-\beta}\ell(x)^{\beta-1}, \quad\text{as $x\to 0^+$}.
\]
Define 
\[
G(t)=\left(\frac{1}{\beta-1}\right)^{1/(\beta-1)} t^{-1/(\beta-1)} \ell(t^{-1/(\beta-1)}), \quad t\geq 1.
\]
If we can show that $\lim_{t\to\infty} F(G(t))/t=1$, then as $F^{-1}\in \text{RV}_\infty(-1/(\beta-1))$, it follows that $G(t)/F^{-1}(t)\to 1$ as $t\to\infty$, 
which proves the claim. 
 
Clearly, as $\ell\in \text{RV}_0(0)$, we have that $G\in \text{RV}_\infty(-1/(\beta-1))$ and thus $G(t)\to 0$ as $t\to\infty$. Hence as $t\to\infty$, the asymptotic behaviour of $F$ at $0$ and the definition of $G$ give  
\begin{equation} \label{eq.FG}
F(G(t))\sim \frac{1}{\beta-1}G(t)^{1-\beta}\ell(G(t))^{\beta-1}\sim t \left(\frac{\ell(G(t))}{\ell(t^{-1/(\beta-1)})} \right)^{\beta-1}.
\end{equation}
Since $\ell\in \text{RV}_0(0)$, we have that 
\[
\lim_{t\to\infty} \frac{\ell(G(t))}{\ell(t^{-1/(\beta-1)}\ell(t^{-1/(\beta-1)}))}
=1.
\]
Therefore
\begin{multline*}
\lim_{t\to\infty} \frac{\ell(G(t))}{\ell(t^{-1/(\beta-1)})} 
\\=\lim_{t\to\infty} \frac{\ell(G(t))}{\ell(t^{-1/(\beta-1)}\ell(t^{-1/(\beta-1)}))} \cdot \frac{\ell(t^{-1/(\beta-1)}\ell(t^{-1/(\beta-1)}))}{\ell(t^{-1/(\beta-1)})}=1,
\end{multline*}
because the second limit is unity, by \eqref{eq.ell}. Returning to \eqref{eq.FG}, we see that $F(G(t))/t\to 1$ as $t\to\infty$, as we required.
\end{proof}
Once a regularly varying function $f$ has been given, $\ell$ is determined. It happens that many regularly varying functions $f$ enjoy the property \eqref{eq.ell}. We give the details now for a parameterised family of such functions. 
\begin{example}  \label{eq.examplepowerloglog}
Suppose for instance that $\beta>1$ and $\beta_1$ and $\beta_2$ are real and $f$ obeys 
\[
\lim_{x\to 0^+} \frac{f(x)}{a|x|^\beta\log^{\beta_1}(1/|x|)\{\log\log(1/|x|)\}^{\beta_2}\sgn(x)}=1. 
\]
Then, in the terminology above, we may take $l(x)=a \log^{\beta_1}(1/x){\log\log(1/x)}^{\beta_2}$ for $x>0$ sufficiently small. 
Then 
\[
\ell(x)=l(x)^{-1/(\beta-1)}=a^{-1/(\beta-1)} \log^{-\beta_1/(\beta-1)}(1/x)\{\log\log(1/x)\}^{-\beta_2/(\beta-1)}.
\]
Hence $x\ell(x)$ is in $\text{RV}_0(1)$ and so $\log(1/(x\ell(x)))/\log(1/x)\to 1$ as $x\to 0^+$. This implies $\log\log(1/(x\ell(x)))/\log\log(1/x)\to 1$ 
as $x\to 0^+$. Armed with these limits and the definition of $\ell$ we get
\begin{align*}
\lefteqn{\lim_{x\to 0^+} \frac{\ell(x\ell(x))}{\ell(x)}}\\
&=
\lim_{x\to 0^+}
\frac
{a^{-1/(\beta-1)} \log^{-\beta_1/(\beta-1)}(1/(x\ell(x)))\{\log\log(1/(x\ell(x)))\}^{-\beta_2/(\beta-1)}}
{a^{-1/(\beta-1)} \log^{-\beta_1/(\beta-1)}(1/x){\log\log(1/x)}^{-\beta_2/(\beta-1)}}\\
&=1,
\end{align*}
as so \eqref{eq.ell} holds. Therefore, by Proposition~\ref{prop.Finvasy}, we have that 
\begin{align*}
F^{-1}(t)&\sim \left(\frac{1}{a(\beta-1)}\right)^{1/(\beta-1)} t^{-1/(\beta-1)}\left(\frac{1}{\beta-1}\log t\right)^{-\beta_1/(\beta-1)}\\
&\qquad\times(\log\log t)^{-\beta_2/(\beta-1)}, \quad\text{as $t\to\infty$}.
\end{align*}
Suppose now that $g$ is continuous such that $\int_t^\infty g(s)\,ds=0$ and obeys 
\[
\lim_{t\to\infty} \frac{\int_t^\infty g(s)\,ds}{t^{-1/(\beta-1)}\left(\log t\right)^{-\beta_1/(\beta-1)}
(\log\log t)^{-\beta_2/(\beta-1)}}=:\mu_D\in [-\infty,\infty].
\]
If we suppose that $f$ has the above asymptotic behaviour at zero, is locally Lipschitz continuous on $\mathbb{R}$, and obeys \eqref{fatinfinity}, then there 
is a unique continuous solution of \eqref{eq.odepert} which obeys $x(t)\to 0$ as $t\to\infty$. Furthermore, if $\mu_D=0$, then 
\[
\lim_{t\to\infty} \frac{x(t)}{t^{-1/(\beta-1)}\left(\frac{1}{\beta-1}\log t\right)^{-\beta_1/(\beta-1)}(\log\log t)^{-\beta_2/(\beta-1)}}\in {-1,0,1}.
\]
If, on the other hand, $\mu_D\neq 0$, then the above limit may not exist, and cannot be $0$ or $\pm1$.

Suppose that $f$ has the same properties (but not necessarily \eqref{fatinfinity}), and consider instead the solution of the stochastic equation \eqref{eq.sde} where $\sigma\in L^2(;0,\infty);\mathbb{R})$ is a continuous function for which 
\[
\lim_{t\to\infty} \frac{\int_t^\infty \sigma^2(s)\,ds \log\log\left(1/\int_t^\infty \sigma^2(s)\,ds\right)}{t^{-2/(\beta-1)}\left(\log t\right)^{-2\beta_1/(\beta-1)} (\log\log t)^{-2\beta_2/(\beta-1)}}=:\mu_S\in [0,\infty].
\]  
Then the unique continuous adapted process $X$ which obeys \eqref{eq.sde} obeys $X(t)\to 0$ as $t\to\infty$ a.s. Furthermore, if $\mu_S=$, we have that 
\[
\frac{X(t)}{t^{-1/(\beta-1)}\left(\frac{1}{\beta-1}\log t\right)^{-\beta_1/(\beta-1)}(\log\log t)^{-\beta_2/(\beta-1)}}\in {-1,0,1}, \quad \text{a.s.,}
\]
while if $\mu_S\neq 0$, we have that $X(t)/F^{-1}(t)$ tends to a limit with zero probability. 
\end{example}

\begin{example} \label{examp.derivcondnversesXcondn}
We have seen that when $S_f(\epsilon,h)<+\infty$ for all $\epsilon>0$ and \textit{some} $h>0$, then $X(t)/F^{-1}(t)$ has limit in $\{-1,0,1\}$ a.s. and 
that $-(X(t+h)-X(t))/h)/(f\circ F^{-1})(t)$ also has the same limit a.s. Therefore, we see that preservation of the asymptotic behaviour of the finite difference approximation to the derivative of \eqref{eq.ode} in the solution of \eqref{eq.sde} requires a condition on $\sigma$ which is not weaker than that required to preserve solely the asymptotic behaviour of the solution of \eqref{eq.ode}. 

In the following example, we explore two aspects of these asymptotic results. First, it is our 
conjecture that if there is \textit{some} $h'>0$ for which $S_f(\epsilon,h')<+\infty$ for all $\epsilon>0$, then it is the case that 
$S_f(\epsilon,h)<+\infty$ for all $\epsilon>0$ and \textit{all} $h>0$. Therefore, if the asymptotic behaviour of the finite difference approximation to the derivative of \eqref{eq.ode} is preserved for any step--size $h'>0$, it will be preserved for any fixed time step $h>0$. Secondly, we see that the condition 
under which the finite difference approximation of the derivative is preserved is \textit{strictly} stronger than that needed to preserve the asymptotic
behaviour of the underlying unperturbed deterministic equation \eqref{eq.ode}.

Let us take for definiteness the simple case when $f(x)\sim a|x|^\beta\sgn(x)$ as $x\to 0$ for some $a>0$ and $\beta>1$. Suppose also that 
$\sigma(t)\sim ct^{-\gamma}$ as $t\to\infty$ for $c\neq0$ and $\gamma>0$. If $\gamma\leq 1/2$, we have that $\sigma\not\in L^2([0,\infty);\mathbb{R})$, so 
the solution of \eqref{eq.sde} cannot inherit the decay properties of the solution of \eqref{eq.ode}. 

Therefore, we let $\gamma>1/2$. We note that elementary considerations, or Proposition~\ref{prop.Finvasy} enable us to show that 
\[
F^{-1}(t)\sim \left(\frac{1}{a(\beta-1)}\right)^{1/\beta-1} t^{-1/(\beta-1)}, \quad\text{as $t\to\infty$}.
\]
and of course $F^{-1}\in \text{RV}_\infty(-1/(\beta-1))$. Clearly $\sigma^2\in \text{RV}_\infty(-2\gamma)$, and so $t\mapsto \int_t^\infty \sigma^2(s)\,ds\in \text{RV}_\infty(-2\gamma+1)$. On  account of the logarithmic factor, we see that for $\gamma>(\beta+1)/(2(\beta-1))$ we have that 
\[
\lim_{t\to\infty} \frac{X(t)}{\left(\frac{1}{a(\beta-1)}\right)^{1/\beta-1} t^{-1/(\beta-1)}}\in \{-1,0,1\}, \quad \text{a.s.}
\]
while for $\gamma\leq (\beta+1)/(2(\beta-1))$ the limit on the left--hand side exists with probability zero. 

Considering $S_f(\varepsilon,h)$, we need to find the asymptotic behaviour of 
\[
\frac{\int_{nh}^{(n+1)h} \sigma^2(s)\,ds}{(f\circ F^{-1})^2(nh)}.
\] 
The numerator scales like $K_1(h)n^{-2\gamma}$ as $n\to\infty$; since $(f\circ F^{-1})(t)\sim K_2 t^{-\beta/(\beta-1)}$ as $t\to\infty$, the denominator 
behaves according to $(f\circ F^{-1})(nh)\sim K_2(h)^2 n^{-2\beta/(\beta-1)}$ as $n\to\infty$. Hence 
\[
\left(
\frac{\int_{nh}^{(n+1)h} \sigma^2(s)\,ds}{(f\circ F^{-1})^2(nh)} \right)^{1/2}
\sim K_3(h) n^{-\gamma+\beta/(\beta-1)}
\]
as $n\to\infty$. Therefore, if $\gamma>\beta/(\beta-1)$ it follows that $S_f(\epsilon,h)<+\infty$ for all $\epsilon>0$ and all $h>0$: thus in this case, 
$(X(t+h)-X(t)/h)/t^{-\beta/(\beta-1)}$ tends to a (known) constant limit with probability one. If however, $\gamma\leq \beta/(\beta-1)$, then for every $h>0$, $S_f(\epsilon,h)=+\infty$ for all $\epsilon>0$, and so $(X(t+h)-X(t)/h)/t^{-\beta/(\beta-1)}$ tends to a finite limit with probability zero. 

Since for $\beta>1$ it is always the case that $(\beta+1)/(2(\beta-1))<\beta/(\beta-1)$, we see that there exist $\gamma$ for which the asymptotic behaviour of 
the approximation to the derivative does not behave like that of the underlying deterministic equation, while the asymptotic behaviour of $X$ itself does.
\end{example}

\section{Simulations}
In order to graphically illustrate our results for stochastic equation, in this section we plot graphs derived from a simulation of a single sample path from the SDE
\begin{align}\label{SDE}
dX(t) = -\sgn(X(t))|X(t)|^{\beta} + (1+t)^{-\gamma}dB(t), \, t \geq 0, 
\end{align}
where $B$ a standard Brownian motion. This is an example of the class of equation examined in Example~\ref{examp.derivcondnversesXcondn}.

In the first instance, we look at the case where $\beta = 3$ and $\gamma = 2.5$. With these parameters we expect to see both ODE--like asymptotics in the solution and the finite difference of approximation of the (non--existent) derivative of $X$ behaving like the those of the ODE as $t\to\infty$. This equation was discretised using the standard explicit Euler-Maruyama method. Convergent solutions of such a numerical scheme are known to possess some asymptotic properties in common with the underlying SDE, as established in \cite{appmackrod2008}, and therefore such simulations should capture faithfully the asymptotic behaviour of the SDE. However, a proof of that this is the case remains to date open: we hope to address this situation for both explicit and implicit Euler--type schemes in a later work. 

It should be remarked that the limit $F(X(t))/t\to 1$ as $t\to\infty$ is observed. Simulations seem to confirm that the limits $X(t)/F^{-1}(t)\to \pm 1$ both occur with positive probability, but that the limit $X(t)/F^{-1}(t)\to 0$ seems to happen with probability zero.
\begin{figure}[h!]
\caption{\small
We observe that the asymptotic regime of the solution quickly settles down to that of the corresponding ODE.}
\centering
\includegraphics[scale=0.38]{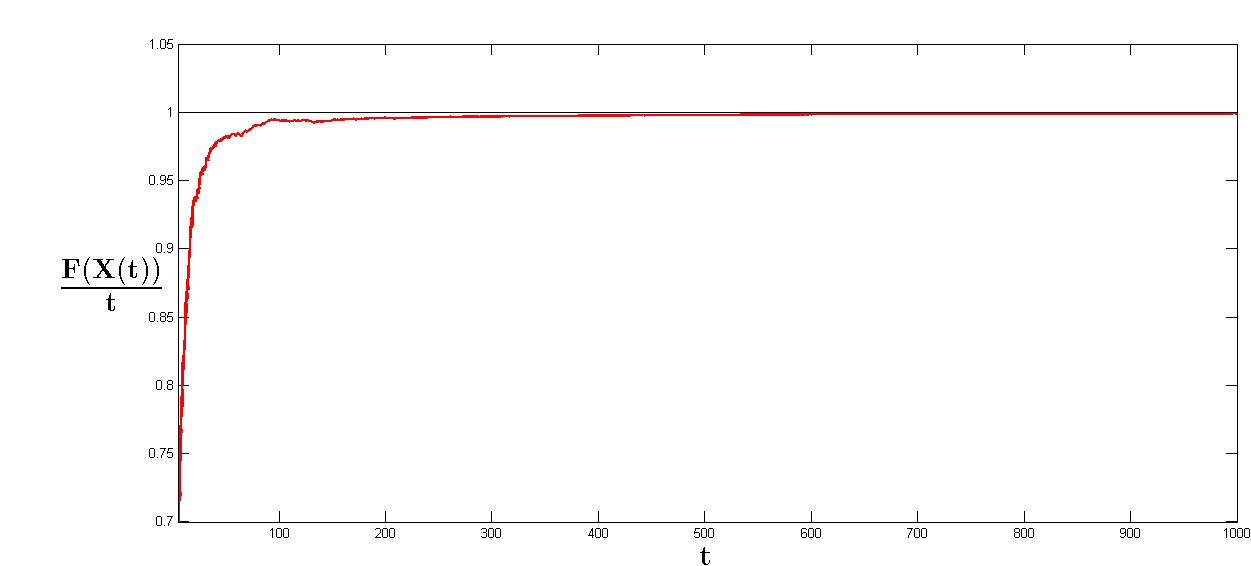}
\caption{\small Once time becomes large enough the approximation to the derivative is well behaved.}
\centering
\includegraphics[scale=0.38]{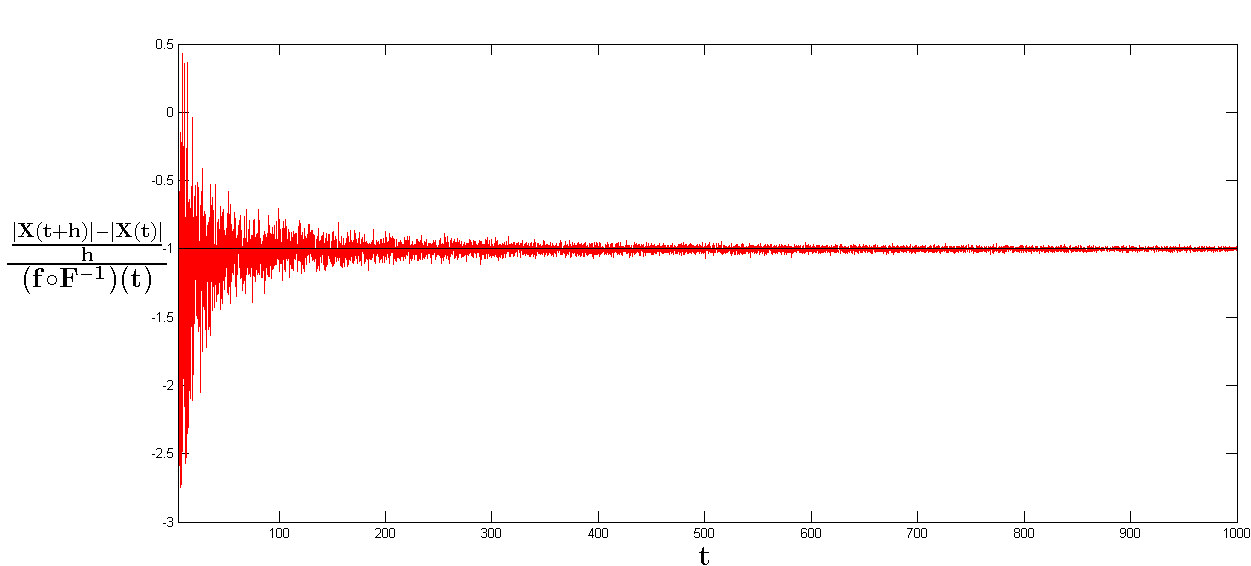}
\end{figure}
\newpage
Below we have two more graphs derived from a single path of (\ref{SDE}) but in this case the parameters are $\beta=3$ and $\gamma = 1.5$. Hence we expect to see the ODE asymptotics preserved but we do not expect to retain the nice asymptotic behaviour of the derivative of the underlying ODE being preserved. The plots confirm this hypothesis. 
\begin{figure}[h!]
\caption{\small
As before the solution settles down to asymptotic regime of the corresponding ODE.}
\centering
\includegraphics[scale=0.38]{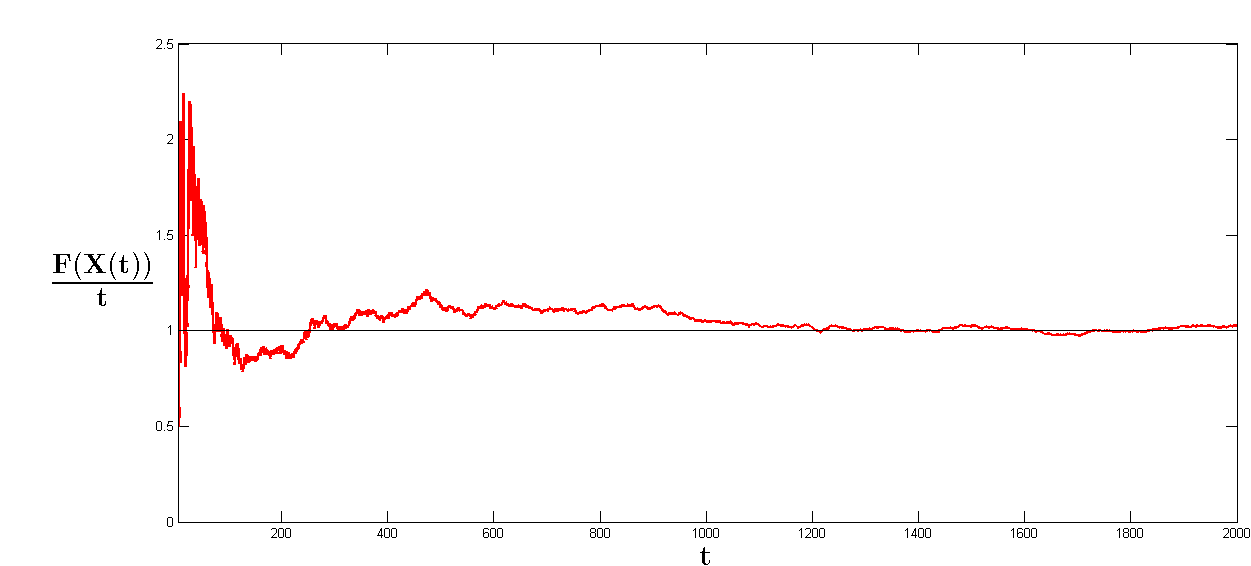}
\caption{\small It is clear that the increased volume of ``noise'' present has caused us to lose the limiting behaviour of the scaled difference in this instance.}
\centering
\includegraphics[scale=0.38]{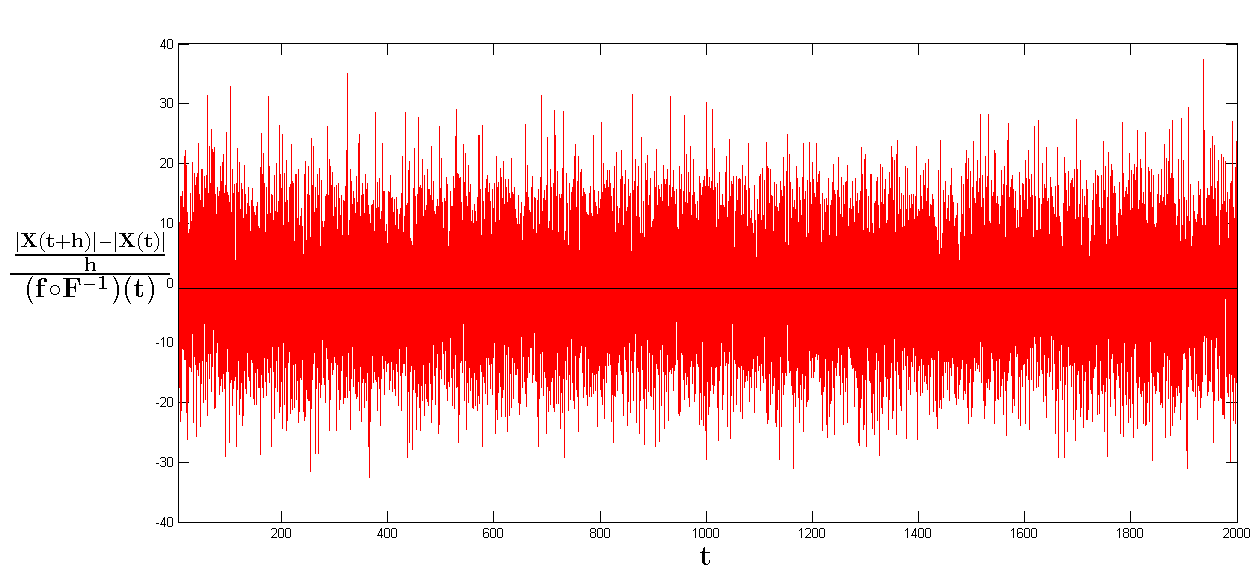}
\end{figure}
\newpage
All graphs presented thus far have been with initial condition $X(0)=1$. We now show two graphs derived from a path of (\ref{SDE}) with $\beta=3$ and $\gamma=2.5$, as before, but with initial condition $X(0)=0$. This helps us to demonstrate some novel behaviour of the scaled finite differences, in particular, the appearance of transient phases which considerably slow convergence to the expected limiting value.
\begin{figure}[h!]
\caption{\small We note here how the scaled differences go through a surprisingly long transient period in which they are very close to zero, almost up to time 1,000. However, upon careful inspection we can see that the graph is in fact beginning to lift away from zero.}
\centering
\includegraphics[scale=0.38]{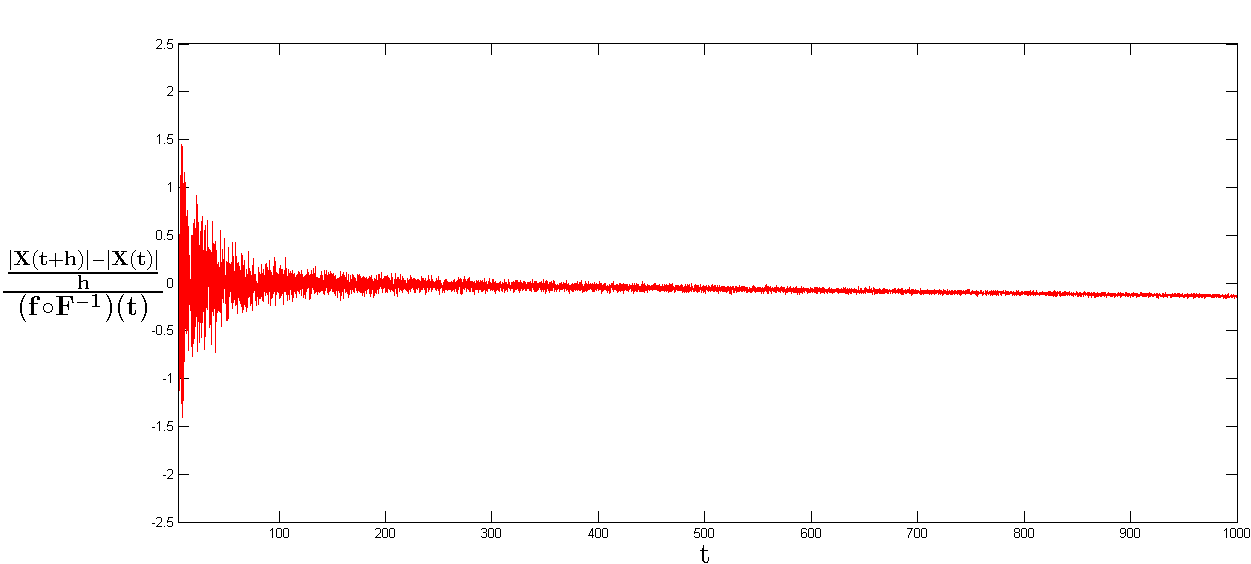}
\caption{\small In this case we show the full path out to time 100,000 and it is clear that the above graph was indeed a transient phase, as claimed, and that the expected limit does eventually prevail.}
\centering
\includegraphics[scale=0.38]{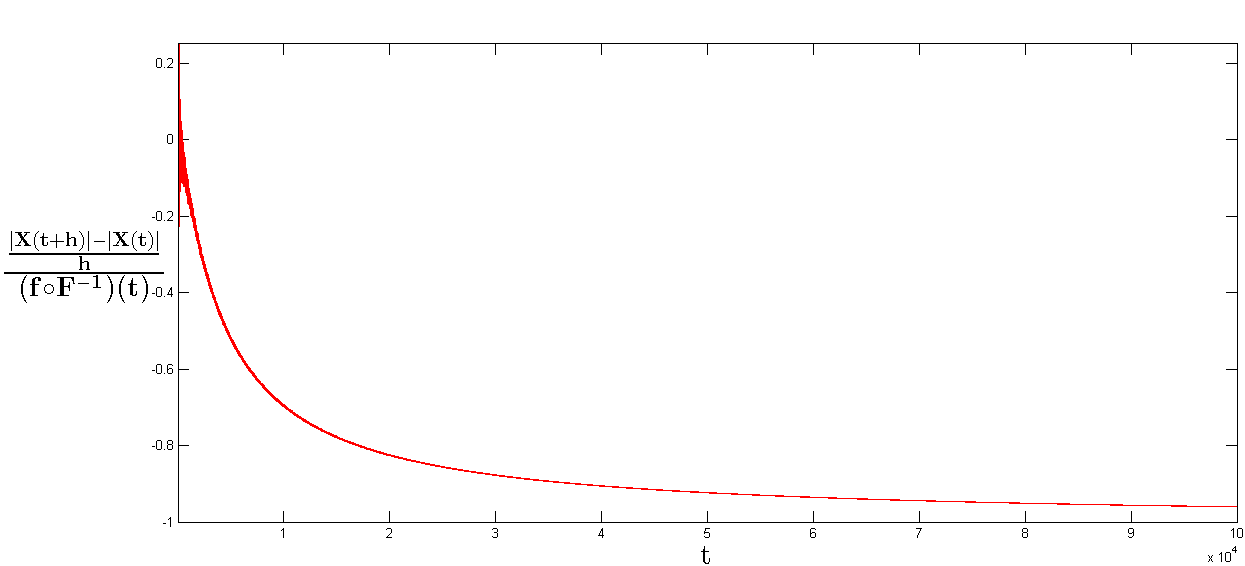}
\end{figure}

\newpage{}

\section{Proof of Theorem~\ref{Thm:Lim}} 
\subsection{Idea and outline of the proof}
Theorem~\ref{Thm:Lim} is the key underlying result of this paper, and its proof relies on careful asymptotic analysis, and a number of interlinked intermediate results. Accordingly, we take a moment to summarise the structure of the proof. in a number of steps. First, we establish that $f$ being asymptotically odd and regularly varying implies that $f$ is asymptotic to a regularly varying, increasing and $C^1$ function $\varphi$ that is odd: in other words, $f$ is asymptotic to a function with improved regularity properties. Then, we show that $t\mapsto |x(t)|$ can be written in terms of the solution of a differential inequality with depends solely on $\varphi$, modulo some small parameter which deals with the asymptotic behaviour of $f$, that $x(t)\to 0$ as $t\to\infty$, and that $\gamma(t)/F^{-1}(t)\to 0$ as $t\to\infty$.

The rest of the proof involves a successive ``ratcheting'' of the asymptotic results: the last two steps of the proof in particular rely on constructing functions 
that are guaranteed to majorise and minorise $x$ for sufficiently large $t$. The majorisation relies on a comparison principle based on the differential 
inequality derived for $t\mapsto |x(t)|$; the minorisation also relies on a comparison argument, but on this occasion the original ODE \eqref{eq.odepert} is employed to make the comparison argument work. In particular, we prove the result through the following steps:
\begin{itemize}
\item[STEP 1:] $\liminf_{t\to\infty} |x(t)|/F^{-1}(t)=0$ or $1$.
\item[STEP 2:] $\limsup_{t\to\infty} |x(t)|/F^{-1}(t)=0$ or $\limsup_{t\to\infty} |x(t)|/F^{-1}(t)\in [1,\infty)$.
\item[STEP 3:] If $\limsup_{t\to\infty} |x(t)|/F^{-1}(t)>0$, then $\limsup_{t\to\infty} |x(t)|/F^{-1}(t)=1$.
\item[STEP 4:] If $\limsup_{t\to\infty} |x(t)|/F^{-1}(t)=1$, then $\liminf_{t\to\infty} |x(t)|/F^{-1}(t)=1$.
\end{itemize}
Of course, it can be seen that STEPs 3 and 4 together imply that the limit of $t\mapsto x(t)/F^{-1}(t)$ must exist and be 0, -1 or 1, which is the desired result.

The sequence of steps mimics those used to determine the asymptotic behaviour in \cite{appmack2003} for stochastic differential equations and in \cite{appmackrod2008} for stochastic difference equations in the special case that $f(x)$ is asymptotic to $a|x|^\beta\sgn(x)$ as $x\to 0$ for $a>0$. 
The proofs of STEPS 1, 3 and 4 differ from those in both papers, although comparison arguments are employed. The proof of STEP 2 is essentially identical 
to that used in both papers. 

The rest of this section is devoted to the 
\subsection{Statement and proofs of technical results}
\begin{lemma} \label{asym_odd}
Suppose that $f$ obeys \eqref{asym} and \eqref{RVat0}. Then there exists a function $\varphi$ such that 
\begin{align} \label{properties}
\varphi \text{ is increasing, in $C^1(\mathbb{R})$, is odd and $\varphi \in RV_0(\beta)$},
\end{align} 
and
\begin{align} \label{asym2}
\lim_{x \to 0}\frac{f(x)}{\varphi(x)}=1.
\end{align}
Moreover, if $\beta>1$ and we define 
\begin{equation}  \label{def.Phi}
\Phi(x)=\int_x^1 \frac{1}{\varphi(u)}\,du, \quad x>0,
\end{equation}
we have that 
\begin{equation} \label{eq.FasyPhi}
\lim_{x\to 0^+} \frac{\Phi(x)}{F(x)}=1, \quad \lim_{t\to\infty} \frac{\Phi^{-1}(t)}{F^{-1}(t)}=1.
\end{equation}
\end{lemma}
\begin{proof}
Recall that $f \in \text{RV}_0(\beta)$ implies that there exists $\phi_+ \in C^1$ such that 
\begin{align*}
\lim_{x \to 0^+}\frac{f(x)}{\phi_+(x)}=1, \,\, \lim_{x \to 0^+}\frac{x\,\phi_+'(x)}{\phi_+(x)} = \beta >0.
\end{align*}
Thus there exists $\delta > 0$ such that $\phi_+$ is increasing and $C^1$ on $(0,\delta)$. We can extend $\phi_+$ to all of $(0,\infty)$ in such a manner that $\phi_+$ is increasing and $C^1$ on all of $(0,\infty)$. 
Define 
\[\varphi(x) =
\begin{cases}
\phi_+(x), \,\, x>0, \\
0, \,\, x=0, \\
-\phi_+(x), \,\, x<0.
\end{cases}
\]
Then $\varphi$ is increasing and odd on $\mathbb{R}$. Moreover, we have
\begin{align*}
\varphi'(0^+) = \lim_{x \to 0^+}\frac{\varphi(x)-\varphi(0)}{x} = \lim_{x \to 0^+}\frac{\phi_+(x)}{x} = 0,
\end{align*}
since $f \sim \phi_+$. Similarly,
\begin{align*}
\varphi'(0^-) = \lim_{x \to 0^-}\frac{\varphi(x)-\varphi(0)}{x} =  \lim_{x \to 0^-}\frac{-\phi_+(-x)}{x} = \lim_{x \to 0^+}\frac{\phi_+(-x)}{-x}=0.
\end{align*}
Hence, as $\varphi$ is in $C^1(0,\infty)$ and $C^1(-\infty,0)$ we conclude that $\varphi \in C^1(\mathbb{R})$. Finally, 
\begin{align*}
\lim_{x \to 0^+}\frac{f(x)}{\varphi(x)} = \lim_{x \to 0^+}\frac{f(x)}{\phi_+(x)} = 1.
\end{align*}
Similarly, we have 
\begin{align*}
\lim_{x \to 0^-}\frac{f(x)}{\varphi(x)} &= \lim_{x \to 0^-}\frac{f(x)}{\varphi(x)}\frac{\varphi(x)}{-\phi_+(-x)} = \lim_{x \to 0^-}\frac{f(x)}{\varphi(x)}\frac{-\varphi(-x)}{-\phi_+(-x)} \\ &= \lim_{x \to 0}\frac{f(x)}{\varphi(x)}\frac{\varphi(-x)}{f(-x)}\frac{f(-x)}{\phi_+(-x)} = 1, 
\end{align*}
as required. For $\beta>1$, the asymptotic behaviour of $\Phi$ defined in \eqref{eq.FasyPhi} is a consequence of the regular variation of $f$ and the fact that 
$f$ is asymptotic to $\varphi$.
\end{proof}
Although $x$ is continuously differentiable, $t\mapsto |x(t)|$ will not be differentiable if $x$ assumes zero values. Since this cannot be ruled out, we derive 
a differential inequality (in terms of Dini derivatives) for $t\mapsto |x(t)|$. Accordingly, we use in the next proof the notation 
\[
D_+ u(t)=\limsup_{h\to 0,h>0} \frac{u(t+h)-u(t)}{h}
\]
for the appropriate Dini derivative. 
\begin{lemma} \label{lemma.Diffineq}
Suppose that $f$ satisfies (\ref{asym}) and (\ref{RVat0}) with $\beta>1$. 
Suppose that $\gamma$ is continuous and $x$ is the unique continuous solution of  
\begin{equation} \label{eq.xfxgamma}
x'(t) = -f( x(t) + \gamma(t) ), \quad t \geq 0,
\end{equation}
such that 
\begin{equation} \label{eq.xto0}
\lim_{t \to \infty}x(t)=0.
\end{equation}
Suppose also that $\gamma$ and $F$ obey \eqref{eq.gammadivFinv}. If $\varphi$ is the function in \eqref{properties} which satisfies (\ref{asym2}), and $\Phi$ is defined by \eqref{def.Phi}, then 
for every $\epsilon\in (0,1)$ there exists $T_1(\epsilon)>0$ and $T(\epsilon)>0$ such that 
\begin{align} \label{gamma}
|\gamma(t)| < \epsilon \Phi^{-1}(t), \quad t \geq T_1(\epsilon), 
\end{align}
and 
\begin{align} \label{DiffInequality}
D_+|x(t)| \leq -\varphi_\epsilon( |x(t)| - \epsilon \Phi^{-1}(t) ), \quad t \geq T(\epsilon),
\end{align} 
where 
\begin{align*}
\varphi_\epsilon(x) := \min\left( (1+\epsilon)\varphi(x) ,\, (1-\epsilon)\varphi(x) \right).
\end{align*}
\end{lemma}
\begin{proof} 
Fix $t>0$ and suppose that $x(t)>0$. 
Then as $x \in C^1$, there exists a $h_1$ small enough so that $x(t+h)>0$ for all $0<h<h_1$. Thus for $h<h_1$  
\begin{align*}
\frac{|x(t+h)| - |x(t)|}{h} = \frac{x(t+h) - x(t)}{h} = \frac{1}{h} \int_t^{t+h}{ -f(x(s)+\gamma(s)) ds} .
\end{align*}
Thus we obtain $D_+|x(t)| = -f( x(t) + \gamma(t) ) = -f( |x(t)| + \gamma(t) )$. 
If $x(t)<0$, then there exists a $h_2>0$ such that $x(t+h)<0$ for all $0<h<h_2$. Similarly, we can write
\begin{align*}
\frac{|x(t+h)| - |x(t)|}{h} &= - \frac{x(t+h) - x(t)}{h} = - \frac{1}{h} \int_t^{t+h}{ -f(x(s)+\gamma(s)) ds} .
\end{align*}
Hence $D_+|x(t)| = f( x(t) + \gamma(t) ) = f( -|x(t)| + \gamma(t) )$. 
Finally, if $x(t)=0$, for $h>0$ we have
\begin{align*}
\frac{|x(t+h)| - |x(t)|}{h} = \left\vert{\frac{x(t+h)}{h}}\right\vert = \left\lvert{\frac{x(t+h)-x(t)}{h}} \right\rvert.
\end{align*}
Thus $D_+|x(t)| = |x'(t)| = |-f( x(t) + \gamma(t) )| = |f( x(t) + \gamma(t) )|$. 
Therefore we have 
\begin{align}
\label{positive}D_+|x(t)| =& -f(|x(t)|+ \gamma(t)), \, x(t)>0. \\
\label{negative}D_+|x(t)| =& f(-|x(t)| + \gamma(t)), \, x(t) <0 .\\
\label{zero}D_+|x(t)| =& | f(x(t) + \gamma(t)) |, \, x(t)=0.
\end{align}
Next, by Lemma~\ref{asym_odd}, there exists a function $\varphi$ satisfying (\ref{properties}) and (\ref{asym2}).
Since $\gamma(t)/F^{-1}(t)\to 0$ as $t\to\infty$, we have that $\gamma(t)/\Phi^{-1}(t)\to 0$ as $t\to\infty$. Hence for every $\epsilon>0$, there exists 
$T_1(\epsilon)>0$ such that $|\gamma(t)|<\epsilon \Phi^{-1}(t)$ for all $t\geq T_1(\epsilon)$, as claimed. 

By \eqref{asym2}, for all $\epsilon > 0$ there is $x_1(\epsilon)>0$ such that 
\[
1-\epsilon<\frac{f(x)}{\varphi(x)}<1+\epsilon, \quad |x|<x_1(\epsilon), x\neq 0.
\]
Therefore, as $f(0)=\varphi(0)=0$ and $x\varphi(x)>0$ for all $x\neq 0$, this implies that 
\begin{align*}
-(1+\epsilon)\varphi(x) &\leq  -f(x) \leq  -(1-\epsilon)\varphi(x), \quad 0\leq x<x_1(\epsilon), \\
-(1-\epsilon)  \varphi(x) &< -f(x) < -(1+\epsilon) \varphi(x), \quad -x_1(\epsilon)<x<0.
\end{align*}
Since $x(t)\to 0$ as $t\to\infty$ and $\gamma(t)\to 0$ as $t\to\infty$, there is a $T^*(\epsilon)$ large enough such that for all $\epsilon>0$ we have $|x(t)| + |\gamma(t)|< x_1(\epsilon)$. Set $T(\epsilon) = 1 + \max(T^*(\epsilon),\, T_1(\epsilon))$. We now deduce that the differential inequality \eqref{DiffInequality}
holds for $t\geq T(\epsilon)$ by considering separately the cases when $x(t)$ is positive, negative and zero. 

(\Rn{1}) If $x(t)>0$, we have from (\ref{positive}) that $D_+|x(t)| = -f(|x(t)|+\gamma(t))$. Therefore for $t\geq T(\epsilon)$, the argument of $-f$ has modulus less than $x_1(\epsilon)$. Hence, if $|x(t)|+\gamma(t)\geq 0$, we have $-f(|x(t)|+\gamma(t))\leq -(1-\epsilon)\varphi(|x(t)|+\gamma(t))\leq 0$. Now 
$-\epsilon \Phi^{-1}(t)<\gamma(t)$, so 
$|x(t)| -\epsilon  \Phi^{-1}(t) < |x(t)|+\gamma(t)$. 
Since $\varphi$ is increasing, we have 
\[
-(1-\epsilon)\varphi(|x(t)| -\epsilon  \Phi^{-1}(t)) > -(1-\epsilon)\varphi(|x(t)|+\gamma(t)). 
\] 
Hence 
\[
D_+|x(t)|=-f(|x(t)|+\gamma(t))\leq -(1-\epsilon)\varphi(|x(t)|+\gamma(t))< -(1-\epsilon)\varphi(|x(t)| -\epsilon  \Phi^{-1}(t)).
\]
Suppose on the other hand that $|x(t)|+\gamma(t)<0$. Since $t>T(\epsilon)$ we have that  $-x_1(\epsilon)<|x(t)|+\gamma(t)<0$. 
Then $-f(|x(t)|+\gamma(t)) < -(1+\epsilon) \varphi(|x(t)|+\gamma(t))$, and it is moreover the case that $0>\gamma(t)>-\epsilon \Phi^{-1}(t)$. Hence 
$|x(t)| -\epsilon  \Phi^{-1}(t) < |x(t)|+\gamma(t)$. 
Since $\varphi$ is increasing, we have 
\[
-(1+\epsilon)\varphi(|x(t)| -\epsilon  \Phi^{-1}(t)) > -(1+\epsilon)\varphi(|x(t)|+\gamma(t)). 
\] 
Hence
\[
D_+|x(t)|=-f(|x(t)|+\gamma(t))< -(1+\epsilon)\varphi(|x(t)|+\gamma(t))< -(1+\epsilon)\varphi(|x(t)| -\epsilon  \Phi^{-1}(t)).
\]
Therefore, when $x(t)>0$, by using the fact that  $a<b$ and $a<c$ implies $a<\max(b,c)$, we have
\begin{align*}
D_+|x(t)| &< \max(-(1-\epsilon)\varphi(|x(t)| -\epsilon  \Phi^{-1}(t)), -(1+\epsilon)\varphi(|x(t)| -\epsilon  \Phi^{-1}(t)) )\\
&=-\min((1-\epsilon)\varphi(|x(t)| -\epsilon  \Phi^{-1}(t)), (1+\epsilon)\varphi(|x(t)| -\epsilon  \Phi^{-1}(t)))\\
&=-\varphi_\epsilon(|x(t)| -\epsilon  \Phi^{-1}(t)),
\end{align*}
where we have used the definition of $\varphi_\epsilon$ at the last step. Hence 
\begin{equation} \label{eq:D1}
D_+|x(t)| <  =-\varphi_\epsilon(|x(t)| -\epsilon  \Phi^{-1}(t)), \quad t\geq T(\epsilon), \quad x(t)>0.
\end{equation}

(\Rn{2}) $x(t)<0$, so $|x(t)| = -x(t) > 0$. First we note that for $t \geq T(\epsilon)$ that 
$D_+|x(t)| = f(-|x(t)| + \gamma(t))$. 
Suppose first that $-|x(t)|+\gamma(t)\geq 0$. Then, we have that $\gamma(t)\geq 0$. Hence $\gamma(t)>-\epsilon \Phi^{-1}(t)$. Therefore, as 
$x_1(\epsilon)>-|x(t)|+\gamma(t)\geq 0$, $-f(-|x(t)|+\gamma(t))\geq -(1+\epsilon)\varphi(-|x(t)|+\gamma(t))$. Hence as $\varphi$ is odd, we get
\[
D_+|x(t)| = f(-|x(t)| + \gamma(t))\leq (1+\epsilon)\varphi(-|x(t)|+\gamma(t))= -(1+\epsilon)\varphi(|x(t)|-\gamma(t)).
\]
Next as $\varphi$ is increasing, $-(1+\epsilon)\varphi(|x(t)| -\epsilon \Phi^{-1}(t)) > -(1+\epsilon)\varphi(|x(t)| -\gamma(t))$, 
so
\[
D_+|x(t)| \leq -(1+\epsilon)\varphi(|x(t)|-\gamma(t)) < -(1+\epsilon)\varphi(|x(t)| -\epsilon \Phi^{-1}(t)).
\]
Suppose next that $-|x(t)|+\gamma(t)< 0$. Then $-x_1(\epsilon)<-|x(t)|+\gamma(t)< 0$, and we have that 
$-f(-|x(t)|+\gamma(t))> -(1-\epsilon)\varphi(-|x(t)|+\gamma(t))$. Hence as $\varphi$ is odd we get
\begin{align*}
D_+|x(t)| = f(-|x(t)| + \gamma(t)) < (1-\epsilon)\varphi(-|x(t)| + \gamma(t)) = -(1-\epsilon)\varphi(|x(t)| - \gamma(t)).
\end{align*}
Now, as $\varphi$ is increasing, we have $-(1-\epsilon)\varphi(|x(t)|-\gamma(t))<-(1-\epsilon)\varphi(|x(t)|-\epsilon \Phi^{-1}(t))$, so
\[
D_+|x(t)| < -(1-\epsilon)\varphi(|x(t)|-\epsilon \Phi^{-1}(t)).
\]
Therefore, regardless of the sign of $-|x(t)|+\gamma(t)$, we have that 
\[
D_+|x(t)| < \max(-(1-\epsilon)\varphi(|x(t)|-\epsilon \Phi^{-1}(t)), -(1+\epsilon)\varphi(|x(t)| -\epsilon \Phi^{-1}(t)))
\]
and so by the definition of $\varphi_\epsilon$ we get
\begin{align} 
\label{eq:D2} D_+|x(t)| < -\varphi_\epsilon(|x(t)| - \epsilon \Phi^{-1}(t)), \quad, t \geq T(\epsilon), \quad x(t)<0.
\end{align}

(\Rn{3}) When $x(t)=0$ we have  
\begin{align*}
D_+|x(t)| = |f(\gamma(t))| \leq (1+\epsilon)|\varphi(\gamma(t))|, \quad t \geq T(\epsilon).
\end{align*}
Therefore, as $\varphi$ is odd and increasing we have
\begin{align*}
|\varphi(\gamma(t))| &= \varphi(|\gamma(t)|) \leq \varphi(\epsilon \Phi^{-1}(t)) \\
&= -\varphi(-\epsilon \Phi^{-1}(t)) = -\varphi(|x(t)|-\epsilon \Phi^{-1}(t)).
\end{align*}
Hence for $t \geq T(\epsilon)$ we obtain
\begin{align*}
 D_+|x(t)| &\leq -(1+\epsilon)\varphi(|x(t)|-\epsilon \Phi^{-1}(t))\\ 
 &\leq \max( -(1+\epsilon)\varphi(|x(t)|-\epsilon \Phi^{-1}(t))  , -(1-\epsilon)\varphi(|x(t)|-\epsilon \Phi^{-1}(t)) )\\
 &=-\varphi_\epsilon(|x(t)|-\epsilon \Phi^{-1}(t)))
\end{align*}
from the definition of $\varphi_\epsilon$. Hence
\begin{align}
\label{eq:D3} D_+|x(t)| &\leq  -\varphi_\epsilon(|x(t)|-\epsilon \Phi^{-1}(t)), \quad t\geq T(\epsilon),\quad x(t) =0.
\end{align}
Combining (\ref{eq:D1}), (\ref{eq:D2}) and (\ref{eq:D3}) we have that, for all $t \geq T(\epsilon)$,
\begin{align} \label{eq:D4}
D_+|x(t)| \leq -\varphi_\epsilon( |x(t)| - \epsilon \Phi^{-1}(t) ), \quad t\geq T(\epsilon),
\end{align}
as required. 
\end{proof}
\begin{lemma} \label{lemma.xliminf}
Suppose that $f$ satisfies (\ref{asym}) and (\ref{RVat0}) with $\beta>1$. 
Suppose that $\gamma$ is continuous and $x$ is the unique continuous solution of \eqref{eq.xfxgamma}
which satisfies \eqref{eq.xto0}. Suppose also that $\gamma$ and $F$ obey \eqref{eq.gammadivFinv}. 
Then 
\begin{align*}
\liminf_{t \to \infty}\frac{|x(t)|}{F^{-1}(t)}=0 \text{ or }1. 
\end{align*}
\end{lemma}
\begin{proof}
Either $\liminf_{t \to \infty} |x(t)|/F^{-1}(t)=0$ holds, or $\liminf_{t \to \infty} |x(t)|/F^{-1}(t) \in (0,\infty]$. Suppose that 
\begin{align*}
\liminf_{t \to \infty}\frac{|x(t)|}{F^{-1}(t)}= M \in (0, \infty), \quad M\neq 1. 
\end{align*}
Then there exists $T_0>0$ such that $|x(t)| > \frac{M}{2}F^{-1}(t)$ for all $t \geq T_0$.
Hence $\lim_{t \to \infty} (x(t)+\gamma(t))/x(t)=1$, since $\gamma(t)/F^{-1}(t) \to 0$ as $t \to \infty$. Thus as $\varphi$ is asymptotic to $f$, we have 
\begin{align*}
\lim_{t \to \infty}\frac{f(x(t)+\gamma(t))}{\varphi(x(t))}=1.
\end{align*}
Hence $\lim_{t \to \infty} x'(t)/\varphi(x(t))=-1$, and integrating yields $\lim_{t \to \infty} \Phi(|x(t)|)/t=1$, which implies that 
$\lim_{t \to \infty} |x(t)|/\Phi^{-1}(t)=1$. Hence $\lim_{t \to \infty} |x(t)|/F^{-1}(t)=1$. Since by supposition $\liminf_{t \to \infty} |x(t)|/F^{-1}(t)=M\neq1$, we have a contradiction. Therefore, if the liminf is finite and non--zero, it must be unity. We now rule out the 
possibility that 
\[
\liminf_{t \to \infty}\frac{|x(t)|}{F^{-1}(t)}= +\infty
\]
Suppose this holds. Then there is $T_0>0$ such that for all $t \geq T_0$, 
$|x(t)| > 2F^{-1}(t)$. Arguing as above, we prove once again that this leads to 
\begin{align*}
\lim_{t \to \infty}\frac{|x(t)|}{F^{-1}(t)}=1, 
\end{align*}
which contradicts our supposition. Therefore, we must have either liminf zero or unity, as all other possibilities have been eliminated.
\end{proof}
\begin{lemma} \label{lemma.xlimsup}
Suppose that $f$ satisfies (\ref{asym}) and (\ref{RVat0}) with $\beta>1$. 
Suppose that $\gamma$ is continuous and $x$ is the unique continuous solution of \eqref{eq.xfxgamma}
which satisfies \eqref{eq.xto0}. Suppose also that $\gamma$ and $F$ obey \eqref{eq.gammadivFinv}. 
Then 
\begin{align*}
\limsup_{t \to \infty}\frac{|x(t)|}{F^{-1}(t)}=0 \mbox{ or } \limsup_{t \to \infty}\frac{|x(t)|}{F^{-1}(t)} \in[1, \infty].
\end{align*}
\end{lemma}
\begin{proof}
Applying Lemma~\ref{asym_odd} to $f$ we know there exists a $\varphi$ satisfying (\ref{properties}) and (\ref{asym2}) with $\beta>1$. Thus 
\begin{align*}
\frac{\varphi((\lambda+\epsilon)x)}{\varphi(x)} < (\lambda+\epsilon)^\beta (1+\epsilon), \quad, |x|< x_0(\epsilon).
\end{align*}
Suppose that 
\begin{align*}
\limsup_{t \to \infty}\frac{|x(t)|}{F^{-1}(t)}=\lambda \in (0, \infty).
\end{align*}
Then, for every $\epsilon>0$, there is $T'(\epsilon)>0$ such that $|x(t)|<(\lambda+\frac{\epsilon}{2})F^{-1}(t)$ for all $t \geq T'(\epsilon)$. 
By \eqref{eq.gammadivFinv}, we also have that there is $T''(\epsilon)>0$ and $T^\ast>0$ such that 
$|\gamma(t)|<\frac{\epsilon}{2}F^{-1}(t)$ for all $t \geq T''(\epsilon)$ and $F^{-1}(t) < x_0(\epsilon)$ for all $t \geq T^\ast$.
Define $T'''(\epsilon)=\max(T'(\epsilon),T''(\epsilon))$, which implies
$|x(t) + \gamma(t)| < (\lambda+\epsilon)F^{-1}(t)$ for all $t \geq T'''(\epsilon)$. 
Hence as $\varphi$ is odd and increasing, for $t \geq T'''(\epsilon)$ we have 
\begin{align*}
|f( x(t)+\gamma(t) )| &< (1+\epsilon)\varphi( |x(t)+\gamma(t)| ) < (1+\epsilon)\varphi( (\lambda+\epsilon)F^{-1}(t) ) \\
&< (1+\epsilon)^2(\lambda+\epsilon)^\beta(\varphi \circ F^{-1})(t). 
\end{align*}
Therefore for $t\geq T'''(\epsilon)$,
\begin{align*}
\left| \int_t^\infty f( x(s)+\gamma(s) )\,ds \right| &\leq (1+\epsilon)^2(\lambda+\epsilon)^\beta \int_t^\infty (\varphi \circ F^{-1})(s)\,ds.
\end{align*}
By \eqref{asym2}, for every $\epsilon\in (0,1)$ there exists $T^\ast(\epsilon)$ such that 
\begin{align*}
\varphi(F^{-1}(t)) < \frac{f(F^{-1}(t))}{1-\epsilon}, \quad t \geq T^*(\epsilon).
\end{align*}
This allows us to write, for $t\geq \max(T'''(\epsilon),T^*(\epsilon))$, the inequality 
\begin{align*}
\left| \int_t^\infty f( x(s)+\gamma(s) )\,ds \right| &\leq \frac{(1+\epsilon)^2(\lambda+\epsilon)^\beta}{1-\epsilon} \int_t^\infty (f \circ F^{-1})(s)\,ds \\
&= \frac{(1+\epsilon)^2(\lambda+\epsilon)^\beta}{(1-\epsilon)}F^{-1}(t).
\end{align*}
Since $x(t)= \int_t^\infty f(x(s)+\gamma(s))\,ds$ we have, for $t\geq \max(T'''(\epsilon),T^*(\epsilon))$,
\begin{align*}
\frac{|x(t)|}{F^{-1}(t)} = \frac{|\int_t^\infty{f(x(s)+\gamma(s))ds}|}{F^{-1}(t)}
\leq \frac{(1+\epsilon)^2(\lambda+\epsilon)^\beta}{(1-\epsilon)}.
\end{align*}
Hence, taking the $\limsup$ yields
\begin{align*}
\lambda \leq \frac{(1+\epsilon)^2(\lambda+\epsilon)^\beta}{(1-\epsilon)}.
\end{align*}
Letting $\epsilon \to 0^+$ gives us $\lambda \leq \lambda^\beta$ or $\lambda^{\beta-1} \geq 1$. Hence $\lambda \geq 1$, as required.
\end{proof}
\begin{lemma} \label{lemma.xlimsup0or1}
Suppose that $f$ satisfies (\ref{asym}) and (\ref{RVat0}) with $\beta>1$. 
Suppose that $\gamma$ is continuous and $x$ is the unique continuous solution of \eqref{eq.xfxgamma}
which satisfies \eqref{eq.xto0}. Suppose also that $\gamma$ and $F$ obey \eqref{eq.gammadivFinv}. 
If  
\begin{align*}
\limsup_{t \to \infty}\frac{|x(t)|}{F^{-1}(t)}>0, 
\end{align*}
then 
\begin{align*} 
\limsup_{t \to \infty}\frac{|x(t)|}{F^{-1}(t)} = 1.
\end{align*}
\end{lemma}
\begin{proof}
From Lemma~\ref{lemma.xlimsup}, if 
\begin{align*}
\limsup_{t \to \infty}\frac{|x(t)|}{F^{-1}(t)}>0, 
\end{align*} 
we have that 
\begin{align*}
\limsup_{t \to \infty}\frac{|x(t)|}{F^{-1}(t)} \geq 1.
\end{align*}
In the case when $\limsup_{t\to\infty} |x(t)|/F^{-1}(t)=1$, we are done. We assume therefore that $\limsup_{t \to \infty} |x(t)|/F^{-1}(t) >1$.  
From Lemma~\ref{lemma.xliminf} we have that either 
\begin{align*}
\liminf_{t \to \infty}\frac{|x(t)|}{F^{-1}(t)}= 0 \mbox{ or } 1.
\end{align*}
If this prevails, for every $\epsilon \in (0,1)$ sufficiently small, there exists $t_n(\epsilon)\nearrow \infty$
such that $|x(t_n)|=(1+\frac{2\beta}{\beta-1}\epsilon)\Phi^{-1}(t_n)$. 
Let $\eta= 3\beta/(\beta-1)$ and define the function 
\begin{align*}
h(\epsilon) := (1+(\eta-1)\epsilon)^\beta (1-\epsilon)^\beta - (1+\eta \epsilon), \quad \epsilon \in [0,1).
\end{align*}
Note that $h(0)=0$ and $h'(0)>0$. Thus there exists $x_1(\beta)>0$ such that $h(\epsilon)>0$ for all $\epsilon< x_1(\beta)$.  
Let $\lambda(\epsilon) = 1+ \frac{3\beta}{\beta-1}\epsilon = 1+ \eta \epsilon$. Therefore 
\begin{align*}
\lambda-\epsilon = 1 + \left(\frac{3\beta}{\beta-1} -1 \right)\epsilon = 1 + \left(\frac{2\beta +1}{\beta-1} \right)\epsilon > 1.
\end{align*}
Furthermore, for $\epsilon<x_1(\beta)$ we have
\begin{align*}
\lefteqn{(\lambda(\epsilon)-\epsilon)^\beta (1-\epsilon)^{\beta-1} - \frac{\lambda}{1-\epsilon}} \\
&=(1+ (\eta-1) \epsilon)^\beta (1-\epsilon)^{\beta-1} - \frac{1+ \eta \epsilon}{1-\epsilon}=\frac{1}{1-\epsilon}h(\epsilon)>0. 
\end{align*}
Since $\varphi \in \text{RV}_0(\beta)$ we have that 
\begin{align*}
\frac{\varphi( (1+\frac{2\beta +1}{\beta-1}\epsilon)x )}{\varphi(x)} > \left(1+\frac{2\beta +1}{\beta-1}\epsilon \right)^\beta (1-\epsilon)^\beta, \quad x< x_2(\epsilon).
\end{align*}
Since $\Phi^{-1}(t)\to 0$ as $t\to\infty$, there exists $T_2(\epsilon)$ such that $\Phi^{-1}(t) < x_2(\epsilon)$ for all $t > T_2(\epsilon)$ 
and as $\gamma$ obeys \eqref{eq.gammadivFinv}, we have that there is $T_1(\epsilon)$ such that $|\gamma(t)| < \epsilon \Phi^{-1}(t)$ for all $t > T_1(\epsilon)$.
Also, by Lemma~\ref{lemma.Diffineq}, there exists $T(\epsilon)>0$ such that we have
\begin{align*} 
D_+|x(t)| \leq -\varphi_\epsilon( |x(t)| - \epsilon \Phi^{-1}(t) ), \quad t \geq T(\epsilon).
\end{align*}
Let $T^\ast(\epsilon) := \inf \{ t_n(\epsilon):t_n > T_1(\epsilon) \vee T_2(\epsilon) \vee T(\epsilon) \}$. Define 
$x_+(t) = \lambda(\epsilon)\Phi^{-1}(t)$ for $t \geq T^\ast(\epsilon)$.
Therefore $x_+'(t) = - \lambda(\epsilon)(\varphi \circ \Phi^{-1})(t)$ for all $t \geq T^\ast(\epsilon)$.
Using the regular variation of $\varphi$ and the fact that $\epsilon<x_1(\beta)$,  for $t\geq T^\ast(\epsilon)$ we have
\begin{align*}
\frac{\varphi((\lambda(\epsilon)-\epsilon)\Phi^{-1}(t))}{\varphi(\Phi^{-1}(t))} > (\lambda(\epsilon)-\epsilon)^\beta (1-\epsilon)^{\beta-1} > \frac{\lambda}{1-\epsilon}.
\end{align*}
Thus
\begin{align*}
-\lambda(\epsilon)\varphi(\Phi^{-1}(t)) > -(1-\epsilon)\varphi((\lambda(\epsilon)-\epsilon)\Phi^{-1}(t)) = -(1-\epsilon)\varphi(x_+(t) -\epsilon \Phi^{-1}(t)).
\end{align*}
Therefore we have
\begin{align*}
x_+'(t) > -(1-\epsilon)\varphi(x_+(t) -\epsilon \Phi^{-1}(t)), \quad t \geq T^*(\epsilon).
\end{align*}
Furthermore, we have
\begin{align*}
x_+(T^\ast) &= \lambda(\epsilon)\Phi^{-1}(T^\ast) = \left(1 + \frac{3\beta}{\beta-1}\epsilon \right)\Phi^{-1}(T^\ast) > \left(1 + \frac{2\beta}{\beta-1}\epsilon \right)\Phi^{-1}(T^\ast) \\
&=|x(T^\ast)|.
\end{align*}
Hence
\begin{align}
\label{Upper} &x_+'(t) > -(1-\epsilon)\varphi( x_+(t)-\epsilon \Phi^{-1}(t) ), \quad t \geq T^\ast(\epsilon) \\
\nonumber &x_+(T^\ast) > |x(T^\ast)|.
\end{align}
Also, by Lemma~\ref{lemma.Diffineq}, we have
\begin{align} \label{Dini} 
D_+|x(t)| \leq -\varphi_\epsilon( |x(t)| - \epsilon \Phi^{-1}(t) ), \quad t \geq T^\ast(\epsilon).
\end{align}
Suppose there is a minimal $t' > T^*(\epsilon)$ such that $|x(t')|=x_+(t')=\lambda(\epsilon)\Phi^{-1}(t')$. Then 
$|x(t')| - \epsilon \Phi^{-1}(t')=x_+(t')-\epsilon \Phi^{-1}(t')=(\lambda-\epsilon)\Phi^{-1}(t')>0$. Hence
\begin{align*}
\lefteqn{\varphi_\epsilon( |x(t')| - \epsilon \Phi^{-1}(t') )}\\
&=\min\{ 
(1+\epsilon)\varphi( |x(t')| - \epsilon \Phi^{-1}(t')),
(1-\epsilon)\varphi( |x(t')| - \epsilon \Phi^{-1}(t'))  
\}\\
&=(1-\epsilon)\varphi( |x(t')| - \epsilon \Phi^{-1}(t')).
\end{align*}
Hence by (\ref{Upper}) and (\ref{Dini}) we get
\[
D_+|x(t)| \leq -(1-\epsilon)\varphi( |x(t')| - \epsilon \Phi^{-1}(t') )=-(1-\epsilon)\varphi(x_+(t')-\epsilon \Phi^{-1}(t'))<x_+'(t').
\]
The minimality of $t'$ implies that $D_+|x(t')| \geq x_+'(t')$, which gives a contradiction. 
Therefore we must have $|x(t)| < x_+(t)$ for all $t \geq T^*(\epsilon)$. Hence,
\begin{align*}
|x(t)| < x_+(t) = \lambda(\epsilon) \Phi^{-1}(t) = \left(1+ \frac{3\beta}{\beta-1}\epsilon\right)\Phi^{-1}(t), \quad t \geq T^*(\epsilon).
\end{align*}
Thus 
\begin{align*}
\limsup_{t \to \infty} \frac{|x(t)|}{\Phi^{-1}(t)} \leq 1+ \frac{3\beta}{\beta-1}\epsilon, 
\end{align*}
so by letting $\epsilon \to 0^+$ and using the fact that $\Phi^{-1} \sim F^{-1}$, we get
\begin{align*}
\limsup_{t \to \infty}\frac{|x(t)|}{F^{-1}(t)} \leq 1.
\end{align*}
This contradicts the supposition that $\limsup_{t\to\infty} |x(t)|/F^{-1}(t)>1$, and so we must have $\limsup_{t\to\infty} |x(t)|/F^{-1}(t)=1$ or $\limsup_{t\to\infty} |x(t)|/F^{-1}(t)=0$, 
as claimed.
%
%
\end{proof}

We are now in a position to prove Theorem~\ref{Thm:Lim}.
\begin{proof}[Proof of Theorem~\ref{Thm:Lim}]
Define 
\begin{align*}
\eta = \frac{3\beta}{\beta-1}, \quad \lambda(\epsilon) = 1 - \eta \epsilon,  \quad 0 < \epsilon < \frac{1}{\eta +1} < \frac{1}{\eta}.
\end{align*}
Then $\lambda(\epsilon) \in (0,1)$ and $\epsilon < \lambda(\epsilon)$. Define 
$h(\epsilon) := (1-\eta \epsilon)-(1-\eta \epsilon + \epsilon)^\beta (1+\epsilon)^\beta$. We note that $h(0)=0$ and $h'(0)>0$. 
Hence there exists $\epsilon'=\epsilon'(\beta)>0$ such that $h(\epsilon)>0$ for all $\epsilon < \epsilon'(\beta)<1$. Therefore
\begin{align*}
(1-\eta \epsilon)-(1-\eta \epsilon + \epsilon)^\beta (1+\epsilon)^\beta > 0, \quad \epsilon < \epsilon',
\end{align*}
or 
\begin{align*}
\lambda(\epsilon) - (\lambda(\epsilon)+\epsilon)^\beta (1+\epsilon)^\beta > 0 , \quad \epsilon < \epsilon'.
\end{align*}
This implies 
\begin{align*}
\lambda(\epsilon) > (\lambda(\epsilon)+\epsilon)^\beta (1+\epsilon)^\beta , \quad \epsilon < \epsilon'.
\end{align*}
For every $\epsilon \in (0,1)$, there is $x_1(\epsilon)>0$ such that $f(x)<(1+\epsilon)\varphi(x)$ for $x<x_1(\epsilon)$ and  
\begin{align*}
\frac{\varphi( (\lambda+\epsilon)x )}{f(x)} < (\lambda+\epsilon)^\beta (1+\epsilon)^{\beta-1}, \quad x < x_1(\epsilon).
\end{align*}
Since $F^{-1}(t)\to 0$ as $t\to\infty$, for every $\epsilon>0$ there is $T_1(\epsilon)>0$ such that $t > T_1(\epsilon)$ implies $F^{-1}(t) < x_1(\epsilon)$
and $F^{-1}(t)<x_2(\epsilon)/(\lambda+\epsilon)$. Also, as $\gamma$ obeys \eqref{eq.gammadivFinv} for every $\epsilon>0$ there exists $T_2(\epsilon)>0$ such that for $t > T_2(\epsilon)$, we have $|\gamma(t)|< \epsilon F^{-1}(t)$. Since $\limsup_{t \to \infty} |x(t)|/F^{-1}(t)=1$, we have that either 
\begin{align*}
(\text{I})\quad \limsup_{t \to \infty}\frac{x(t)}{F^{-1}(t)} = 1,  \quad\mbox{ or } \quad (\text{II}) \quad \limsup_{t \to \infty}\frac{-x(t)}{F^{-1}(t)}=1.
\end{align*}
We consider case (I) first. If it holds, there exists $t_n \nearrow \infty$ such that $x(t_n) > (1 - \frac{\eta}{2}\epsilon)F^{-1}(t_n)$. 
Let $T(\epsilon) = \inf\{ t_n(\epsilon) : t_n(\epsilon) > T_1 \vee T_2 \}$ and define 
$x_-(t) = \lambda(\epsilon)F^{-1}(t)$ for all $t \geq T(\epsilon)$. Then 
$x_-(T) = x_-(t_n) = \lambda(\epsilon)F^{-1}(t_n)$ and we have
\begin{align*}
x(T) &= x(t_n) > (1 - \frac{\eta}{2}\epsilon)F^{-1}(t_n) > (1- \eta \epsilon)F^{-1}(t_n) = x_-(t_n)= x_-(T).
\end{align*}
Hence $x(T) > x_-(T)$. Now for $t \geq T(\epsilon)$, $x_-(t)+ \gamma(t)< (\lambda(\epsilon)+\epsilon)F^{-1}(t)$, 
which implies
\begin{align*}
f(x_-(t)+ \gamma(t)) &< (1+\epsilon)\varphi((\lambda(\epsilon)+\epsilon)F^{-1}(t))\\
&<(1+\epsilon) (\lambda+\epsilon)^\beta(1+\epsilon)^{\beta-1} f( F^{-1}(t) ) 
< 
 \lambda(\epsilon) (f \circ F^{-1})(t), 
\end{align*}
since $F^{-1}(t) < x_1(\epsilon)$, $(\lambda+\epsilon)F^{-1}(t)<x_2(\epsilon)$ and $\epsilon<\epsilon'$. Thus 
\begin{align*}
-f( x_-(t) + \gamma(t) ) > -\lambda(\epsilon)(f \circ F^{-1})(t) = -x_-'(t), \quad t \geq T(\epsilon).
\end{align*}
Therefore $x_-'(t) < -f( x_-(t) +\gamma(t))$ for $t \geq T(\epsilon)$ and $x_-(T) < |x(T)|$.
Now suppose there exists $t'>T$ such that $x(t') = x_-(t')$. Then $x'(t') \leq x_-'(t')$. Hence 
\begin{align*}
x_-'(t') &< -f( x_-(t')+\gamma(t') ) = -f( x(t') + \gamma(t') ) = x'(t') \\
&\leq x_-'(t'), 
\end{align*}
which gives a contradiction. Hence $x_-(t) < x(t), t \geq T(\epsilon)$. Thus 
\begin{align*}
x(t) > x_-(t) = \lambda(\epsilon)F^{-1}(t) = \left(1-\frac{3\beta}{\beta-1}\epsilon \right)F^{-1}(t), \quad t \geq T(\epsilon).
\end{align*}
Therefore we have that 
\begin{align*}
\liminf_{t \to \infty}\frac{x(t)}{F^{-1}(t)} \geq 1 - \frac{3\beta}{\beta-1}\epsilon,
\end{align*}
so by letting $\epsilon \to 0^+$ we get 
\begin{align*}
\liminf_{t \to \infty}\frac{x(t)}{F^{-1}(t)} \geq 1.
\end{align*}
Thus, if $\limsup_{t \to \infty} x(t)/F^{-1}(t)=1$, we have $\liminf_{t \to \infty} x(t)/F^{-1}(t)\geq 1$.
Therefore we have 
\begin{align*}
\limsup_{t \to \infty}\frac{x(t)}{F^{-1}(t)}=1 \, \mbox{ implies } \, \lim_{t \to \infty}\frac{x(t)}{F^{-1}(t)} = 1.
\end{align*}
In case (II), if we have that 
\begin{align*}
\limsup_{t \to \infty}\frac{-x(t)}{F^{-1}(t)} = 1,
\end{align*}
then let $z(t) = -x(t)$ and follow the same argument as before. In this case we let 
\begin{align*}
z_-(t) = \lambda(\epsilon)F^{-1}(t), \text{ for all } t > T^*(\epsilon)
\end{align*} 
and similarly we arrive at $z(t) > z_-(t)$ for all $t > T^*(\epsilon)$. Translating this back to a statement about $x(t)$ we obtain
\begin{align*}
\limsup_{t \to \infty}\frac{-x(t)}{F^{-1}(t)}=1 \, \mbox{ implies } \, \lim_{t \to \infty}\frac{-x(t)}{F^{-1}(t)} = 1,
\end{align*}
as required.
\end{proof} 

\section{Proofs from Section 4}
\subsection{Proof of Theorem~\ref{eq.intgfinite}}
We start by making uniform asymptotic estimates of the terms involving $x$ in the integrated form of \eqref{eq.odepert}, namely
\begin{align} \label{eq:0} 
x(t) = x(0) + \int_0^t f(x(s))\,ds + \int_0^t g(s)\,ds.
\end{align}
This entails making a pointwise estimate of $f(x(t))$. If it can be shown that $\int_0^t f(x(s))\,ds$ tends to a finite limit as $t\to\infty$, the result is secured, because the hypothesis \eqref{eq.xdetperservasy} implies that $x(t)\to 0$ as $t\to\infty$, and therefore that $g$ obeys \eqref{eq.intgfinite}. 

By Lemma~\ref{asym_odd}, there is a function $\varphi$ such that 
\begin{align*}
\frac{1}{2} < \frac{f(x)}{\varphi(x)} < \frac{3}{2}, \,\, |x|<x_1,
\end{align*}
for some $x_1>0$, where $\varphi$ is increasing, odd and $\varphi \in RV_0(\beta)$. Since $\varphi \in RV_0(\beta)$ we also have that
\begin{align*}
\frac{\varphi(x)}{\varphi(\frac{x}{L+1})} < 2(|\lambda|+1)^\beta, \text{ for } |x|<x_2.
\end{align*}
for some $x_2>0$. Thus $|f(x)| < \frac{3}{2}\varphi(|x|)$ for all $|x|<x_1$.
Since $x(t) \rightarrow 0$ as $t \to \infty$, $|x(t)| < x_1$ for all $t \geq T_1$. Since $x$ obeys \eqref{eq.xdetperservasy} and $F^{-1}(t)\to 0$ as $t\to\infty$, we have that there exist $T_2>0$ and $T_3>0$ such that $|x(t)| < (|\lambda|+1)F^{-1}(t)$ for $t \geq T_2$ and $(|\lambda|+1)F^{-1}(t) < x_1$, for $t \geq T_3$. Hence, for $t \geq T := 1 + T_1 \vee T_2 \vee T_3$ we have 
$|f(x(t))| \leq 2\,\varphi(|x(t)|) \leq 2\,\varphi((|\lambda|+1)F^{-1}(t))$. 
Now we estimate the integral involving $f(x(t))$. For $t\geq T$ we have
\begin{align} \label{eq:1}
\left|\int_T^t{f(x(s))ds}\right| &\leq \int_T^t \frac{3}{2}\varphi((|\lambda|+1)F^{-1}(s))\,ds \nonumber \\
&= \frac{3}{2(|\lambda|+1)}\int_{F^{-1}(t)}^{F^{-1}(T)} \frac{\varphi(u)}{\varphi(\frac{u}{|\lambda|+1})}\cdot\frac{\varphi(\frac{u}{|\lambda|+1})}{f(\frac{u}{|\lambda|+1})} \,du.
\end{align}
Now $(|\lambda|+1)F^{-1}(T) \leq (|\lambda|+1)F^{-1}(T_3) < x_1$ so if $0 < u \leq F^{-1}(T)$, then 
\begin{align*}
\frac{u}{|\lambda|+1} \leq \frac{F^{-1}(T)}{|\lambda|+1} < \frac{x_1}{(|\lambda|+1)^2} < x_1.
\end{align*}
Hence
\begin{align} \label{eq:2}
\frac{\varphi(\frac{u}{|\lambda|+1})}{f(\frac{u}{|\lambda|+1})} < 2, \text{ for } u \leq F^{-1}(T).
\end{align}
Next $T>T_3$, so $F^{-1}(T) < F^{-1}(T_3)$ so $(|\lambda|+1)F^{-1}(T) < (|\lambda|+1)F^{-1}(T_3) < x_2$. Hence $u \leq F^{-1}(T)$ implies $u < x_2$. Thus 
\begin{align} \label{eq:3}
\frac{\varphi(u)}{\varphi(\frac{u}{|\lambda|+1})} < 2(|\lambda|+1)^\beta, \text{ for } u \leq F^{-1}(T).
\end{align}
If we insert equations (\ref{eq:2}) and (\ref{eq:3}) into (\ref{eq:1}) we obtain the following inequalities, for $t \geq T$,
\begin{align*}
\left| \int_T^t f(x(s))\,ds \right| \leq \frac{3}{2(|\lambda|+1)} \int_{F^{-1}(t)}^{F^{-1}(T)} 2\cdot 2(|\lambda|+1)^\beta\, du 
\leq 6(|\lambda|+1)^{\beta-1} F^{-1}(T), 
\end{align*}
which is finite. Since $T$ is finite, $\lim_{t \to \infty}\int_0^t f(x(s))\,ds$ is finite, and so $g$ obeys \eqref{eq.intgfinite}, as required. 

\subsection{Proof of Theorem~\ref{thm.detpresnecc}}
By Theorem~\ref{Thm1}, we have that $\lim_{t \to \infty}\int_0^t{g(s)ds}$ exists. 
By \eqref{eq.xdetperservasy}, we have that $\lim_{t \to \infty}x(t)=0$. Also, by Theorem~\ref{Thm1} it follows that 
\[
\lim_{t \to \infty}\int_0^t{f(x(s))}\, ds
\] is finite, so $\int_t^\infty f(x(s))\,ds$ is well defined for all $t \geq 0$. Hence we have 
\begin{equation*}
\int_0^\infty g(s)\,ds = \int_0^\infty f(x(s))\,ds - x(0), \quad \int_0^t g(s)\,ds = \int_0^t f(x(s))\,ds +x(t) - x(0).
\end{equation*}
Therefore we have
\begin{align} \label{eq:+}
\frac{1}{F^{-1}(t)}\int_t^\infty g(s)\,ds = \frac{1}{F^{-1}(t)}\int_t^\infty f(x(s))\,ds - \frac{x(t)}{F^{-1}(t)}, \quad t\geq 0.
\end{align}
We now analyse the asymptotic behaviour of the right--hand side of \eqref{eq:+} to prove the second part of \eqref{eq.intgdivF}. 
Under \eqref{eq.xdetperservasy}, we have either
\begin{align*}
&(\text{\Rn{1}}) \lim_{t \to \infty}\frac{x(t)}{F^{-1}(t)} = 0 \text{ or} \\
&(\text{\Rn{2}}) \lim_{t \to \infty}\frac{x(t)}{F^{-1}(t)} = \pm1. 
\end{align*}
By L'H\^opitals rule, and recalling the properties of the function $\varphi$ introduced in Lemma~\ref{asym_odd}, we may consider 
\begin{align*}
\lim_{t \to \infty}\frac{\int_t^\infty{f(x(s))ds}}{F^{-1}(t)} &= \lim_{t \to \infty}\frac{-f(x(t))}{-f(F^{-1}(t))} = \lim_{t \to \infty}\frac{f(x(t))}{f(F^{-1}(t))} = \lim_{t \to \infty}\frac{\varphi(x(t))}{\varphi(F^{-1}(t))},
\end{align*}
provided that the limit on the right--hand side exists. We now show that it does when \eqref{eq.xdetperservasy}
prevails. In case (i), as $\varphi \in RV_0(\beta)$, $\varphi$ is odd and $|x(t)|/F^{-1}(t) \rightarrow 0$ as $t \to \infty$ we have
\begin{align*}
\lim_{t \to \infty}\frac{|\varphi(x(t))|}{\varphi(F^{-1}(t))} = \lim_{t \to \infty}\frac{\varphi(|x(t)|)}{\varphi(F^{-1}(t))} = 0.
\end{align*}
Thus 
\begin{align*}
\lim_{t \to \infty}\frac{\int_t^\infty{f(x(s))ds}}{F^{-1}(t)} = 0,
\end{align*}
so by taking limits on both sides of (\ref{eq:+}), we have the second part of \eqref{eq.intgdivF}, as required.

In case (ii) the limit 
\begin{align} \label{eq:4}
\lim_{t \to \infty}\frac{\int_t^\infty{f(x(s))ds}}{F^{-1}(t)} = \lim_{t \to \infty}\frac{\varphi(x(t))}{\varphi(F^{-1}(t))},
\end{align}
still obtains, provided the limit on the right--hand side exists. If $x(t)/F^{-1}(t) \rightarrow 1$ as $t \to \infty$ the limit on the right--hand side of (\ref{eq:4}) is 1, so (\ref{eq:4}) and (\ref{eq:+}) combine to yield the second part of \eqref{eq.intgdivF}, as claimed. 
If, on the other hand, $x(t)/F^{-1}(t) \rightarrow 1$ as $t \to \infty$, then from (\ref{eq:4}) we use the fact that $\varphi$ is odd to write 
\begin{align*}
\lim_{t \to \infty}\frac{\int_t^\infty f(x(s))\,ds}{F^{-1}(t)} &= \lim_{t \to \infty}\frac{\varphi(x(t))}{\varphi(F^{-1}(t))} = \lim_{t \to \infty}\frac{-\varphi(-x(t))}{\varphi(F^{-1}(t))} =-1.
\end{align*}
Using this, $x(t)/F^{-1}(t) \rightarrow 1$ as $t \to \infty$ and (\ref{eq:+}) gives \eqref{eq.intgdivF}, as required.

\subsection{Proof of Theorem~\ref{thm.xto0suff}}
We define 
\[
\phi_+=\liminf_{x \to +\infty} |f(x)|, \quad \phi_=\liminf_{x \to -\infty} |f(x)|. 
\]
Note by \eqref{fatinfinity} that $\phi_+,\phi_->0$. 
As before we define $u(t) = \int_0^t g(s)ds$ and thus \eqref{eq.intgfinite} implies that $u(\infty) = \int_0^\infty g(s)\,ds$ is well defined. Hence $\gamma(t) := u(t) - u(\infty) \to 0$ as $t \to \infty$. Define $z(t)=x(t)-u(t)+u(\infty)=x(t)-\gamma(t)$ for $t\geq 0$. Therefore $z(t)\to 0$ as $t\to\infty$ if and only if $x(t)\to 0$ as $t\to\infty$. Moreover $z'(t) = x'(t) - u'(t) = -f(x(t)) = -f(z(t)+\gamma(t))$. Thus we proceed to show that $z$ has the desired limit.
We set $\phi = \min(\phi_+, \, \phi_-)$ and thus there exists $x_1$ such that
\begin{align*}
f(x) \geq \frac{\phi}{2} \text{ for all } x \geq x_1 > 0 \text{ and }
-f(x) \geq \frac{\phi}{2} \text{ for all } x \leq - x_1.
\end{align*}
Now by the Lipschitz continuity of $f$ there exists $K_1$ such that
\begin{align} \label{eq.floclipK1}
|f(x)-f(y)| \leq K_1 |x-y| \text{ for all } x,y \text{ such that } |x|,|y| \leq x_1+1.
\end{align}
Choose $\delta \in (0,1)$ to be small enough that $\phi / 4  > K_1 \delta$ with $\delta/2 < x_1$. Thus for $x \in (x_1-\delta, \, x_1+\delta)$ we have
\[
-K_1 \delta \leq f(x) - f(x_1) \leq K_1 \delta.
\]
Hence we obtain 
\[
f(x) \geq f(x_1) - K_1 \delta \geq \frac{\phi}{2} - \frac{\phi}{4} \geq \frac{\phi}{4}.
\]
Since we know that $\gamma(t) \to 0$ there exists a $T_0>0$ such that $|\gamma(t)|<\delta/2$ for all $t\geq T_0$. 

If there exists $T_2>T_0$ such that $z(T_2) < x_1$, it can be shown that $z(t)<x_1$ for all $t\geq T_2$. We defer the proof of this fact temporarily. Instead, we first assume to the contrary that $z(t)\geq x_1$ for all $t\geq T_0$. Therefore $z(t)+\gamma(t)\geq x_1-\delta/2$ for $t\geq T_0$, and therefore $-z'(t)=f(z(t)+\gamma(t))\geq \phi/4>0$ for $t\geq T_0$. But this implies that $z(t)$ will ultimately lie below $x_1$, a contradiction. 

It remains to prove that if  $z(T_2) < x_1$ for some $T_2>T_0$, then $z(t)<x_1$ for all $t\geq T_2$. Suppose to the contrary that there is a minimal $T_1>T_2$ such $z(T_1)=x_1$. Then $z'(T_1)\geq 0$. On the other hand, $f(z(T_1)+\gamma(T_1))=f(x_1+\gamma(T_1))$. Since $T_1>T_0$, we have $|\gamma(T_1)|\leq \delta/2$, and so it follows that $f(x_1+\gamma(T_1))\geq \phi/4$. Hence 
$0\leq z'(T_1)=-f(x_1+\gamma(T_1))\leq -\phi/4<0$, a contradiction. 

Therefore, we have shown that there exists $T_2>0$ such that $z(t)<x_1$ for all $t\geq T_2$. By a similar argument, it can be shown that there is a $T_3>0$ such that $z(t)>-x_1$ for all $t\geq T_3$. Hence, with $T_4=\max(T_0,T_2,T_3)$, we have that $|z(t)|\leq x_1$ 
for all $t\geq T_4$, and also that $|\gamma(t)|\leq \delta/2$. 

It remains to show that the boundedness of $z$ implies that it tends to zero. Write, for $t\geq 0$, $g(t)=f(z(t))-f(z(t)+\gamma(t))$.
Then $g$ is continuous on $[0,\infty)$, by dint of the continuity of $z$, $\gamma$ and $f$. Since $|\gamma(t)|\leq \delta/2<1/2$ for $t\geq T_4$, it follows that $|z(t)+\gamma(t)|<x_1+1$ and $|z(t)|<x_1+1$ for all $t\geq T_4$. Hence, by \eqref{eq.floclipK1}, we have 
$|g(t)|\leq K_1|\gamma(t)$ for $t\geq T_4$. Since $\gamma(t)\to 0$ as $t\to\infty$, we have that $g(t)\to 0$ as $t\to\infty$. Moreover, 
by the definition of $g$, it follows that $z$ obeys 
\[
z'(t)=-f(z(t))+g(t), \quad t\geq 0.
\] 
The Lipschitz continuity of $f$, the property \eqref{eq.fglobalstable}, \eqref{fatinfinity}, and the fact that $g$ is continuous on $[0,\infty)$ and $g(t)\to 0$ as $t\to\infty$ means, by a result in \cite{JAJGAR:2009}, that $z(t)\to 0$ as $t\to\infty$. This allows us to conclude that $x(t)\to 0$ as $t\to\infty$, as claimed. 

\section{Proofs from Section 5}
We now prove some results from Section 5, up to but not including Theorem~\ref{thm:Xderiv}.
\subsection{Proof of Theorem~\ref{theorem:Xnecc1}}
We rearrange \eqref{eq.sde} and write
\begin{equation} \label{eq.intfrep1sde}
\int_0^t -f(X(s))\,ds = X(t) - \xi - \int_0^t \sigma(s)\,dB(s).
\end{equation}
Since $X$ obeys \eqref{eq.Xperservasy}, it follows from Lemma~\ref{lemsig2}  that $\sigma \in L^2(0,\infty)$. The martingale convergence theorem then implies that the last term on the right--hand side of \eqref{eq.intfrep1sde} has a finite limit as $t\to\infty$ a.s. Moreover, $X$ obeys \eqref{eq.Xperservasy} implies 
that $X(t)\to 0$ as $t\to\infty$ a.s., so from this it follows that all the terms on the right--hand side of \eqref{eq.intfrep1sde} converge to 0 
with probability 1. Therefore 
there is an event $\Omega_1$ such that the limit as $t\to\infty$ of the left--hand side of \eqref{eq.intfrep1sde} is well--defined and we may write 
\begin{align*}
\int_0^\infty -f(X(s))\,ds = - \xi - \int_0^\infty \sigma(s)\,dB(s), \mbox{ on } \Omega_1.
\end{align*}
Taking this identity together with \eqref{eq.intfrep1sde} on $\Omega_1$, we can obtain 
\begin{equation} \label{eq.Xasyintsig}
\int_t^\infty \sigma(s)\,dB(s) = -X(t)+\int_t^\infty f(X(s))\,ds.
\end{equation}
Define the a.s. event on which \eqref{eq.Xasyintsig} holds to be $\Omega^\ast$ and 
\begin{align*}
A := \left\{ \omega : \lim_{t \to \infty}\frac{X(t, \omega)}{F^{-1}(t)}=\lambda(\omega)\in(-\infty,\infty) \right\}\cap \Omega^\ast.
\end{align*}
We have presumed that $\mathbb{P}[A]=1$. We decompose $A=A_+\cup A_-\cup A_0$ where the events $A_\cdot$ are defined by  
\[
A_+=A\cap\{\omega:\lambda(\omega)>0\}, \quad A_-=A\cap\{\omega:\lambda(\omega)<0\}, \quad A_0=A\cap\{\omega:\lambda(\omega)=0\}.
\]
Now, consider $\omega\in A$ so that $\lambda(\omega)\neq 0$; then 
\[
\lim_{t \to \infty}\frac{X(t, \omega)}{\lambda(\omega)F^{-1}(t)}=1. 
\]
Then as $\varphi$ is in $\text{RV}_0(\beta)$ and $f(x)/\varphi(x)\to 1$ as $x\to 0$
\[
\lim_{t\to\infty} \frac{f(X(t))}{\varphi(\lambda(\omega)F^{-1}(t))}
=\lim_{t\to\infty} \frac{\varphi(X(t))}{\varphi(\lambda(\omega)F^{-1}(t))}=1.
\]

In the case when $\omega\in A_+$, since $\varphi\in \text{RV}_0(\beta)$ we have that 
\[
\lim_{t\to\infty} \frac{f(X(t,\omega))}{\varphi(F^{-1}(t))}=\lambda(\omega)^\beta.
\]
Therefore, by L'H\^opital's rule and the fact that $f(x)/\varphi(x)\to 1$ as $x\to 0$, we have 
\[
\lim_{t\to\infty} \frac{\int_t^\infty f(X(s,\omega))\,ds}{F^{-1}(t)}
=\lim_{t\to\infty} \frac{f(X(t,\omega))}{f(F^{-1}(t))}= \lambda(\omega)^\beta.
\]
Rearranging \eqref{eq.Xasyintsig} and taking limits yields for each $\omega\in A_+$ 
\[
\lim_{t\to\infty} \frac{\left(\int_t^\infty \sigma(s)\,dB(s)\right)(\omega)}{F^{-1}(t)}= -\lambda(\omega)+\lambda(\omega)^\beta.
\]
Now write $A_+=A_1\cup A_1'$ where $A_1=A_+\cap \{\lambda=1\}$ and $A_1'=A_+\cap \{\lambda\neq 1\}$. Suppose that $A_1'$ is such that $\mathbb{P}[A_1']>0$. Then 
we have that 
\[
\mathbb{P}\left[\lim_{t\to\infty} \frac{\int_t^\infty \sigma(s)\,dB(s)}{F^{-1}(t)} \text{ exists and is not equal to $0$}   \right]>0,
\]
which contradicts \eqref{eq.zeroeventSig}. Hence $\mathbb{P}[A_1']=0$. Thus $\mathbb{P}[A_+]=\mathbb{P}[A_1]$. Moreover, we have that 
\begin{equation}  \label{eq.A1ok}
\lim_{t\to\infty} \frac{\left(\int_t^\infty \sigma(s)\,dB(s)\right)(\omega)}{F^{-1}(t)}= 0, \quad \omega\in A_1.
\end{equation}

Next, we consider the case when $\omega\in A_-$, so $\lambda(\omega)<0$. As before we have
\[
\lim_{t\to\infty} \frac{f(X(t,\omega))}{\varphi(\lambda(\omega)F^{-1}(t))}=1.
\]
Using this limit and the fact that $\varphi$ is odd, we have 
\[
\lim_{t\to\infty} \frac{f(X(t,\omega))}{\varphi(F^{-1}(t))}
=\lim_{t\to\infty} \frac{\varphi(\lambda(\omega)F^{-1}(t))}{\varphi(F^{-1}(t))}
=\lim_{t\to\infty} \frac{-\varphi(-\lambda(\omega)F^{-1}(t))}{\varphi(F^{-1}(t))},
\]
so because $-\lambda(\omega)>0$, the fact that $\varphi\in \text{RV}_0(\beta)$, and that $f$ and $F^{-1}$ are asymptotic to $\varphi$ and $\Phi^{-1}$ 
respectively implies that 
\[
\lim_{t\to\infty} \frac{\varphi(X(t,\omega))}{\varphi(\Phi^{-1}(t))}=-(-\lambda(\omega))^\beta.
\]
Therefore, by L'H\^opital's rule and the fact that $f(x)/\varphi(x)\to 1$ as $x\to 0$, we have 
\[
\lim_{t\to\infty} \frac{\int_t^\infty f(X(s,\omega))\,ds}{F^{-1}(t)}
=\lim_{t\to\infty} \frac{\int_t^\infty \varphi(X(s,\omega))\,ds}{\Phi^{-1}(t)}
=\lim_{t\to\infty} \frac{\varphi(X(t,\omega))}{\varphi(\Phi^{-1}(t))}= -(-\lambda(\omega))^\beta. 
\]
Rearranging \eqref{eq.Xasyintsig} and taking limits yields for each $\omega\in A_-$ 
\[
\lim_{t\to\infty} \frac{\left(\int_t^\infty \sigma(s)\,dB(s)\right)(\omega)}{F^{-1}(t)}= -\lambda(\omega)-(-\lambda(\omega))^\beta.
\]
Now write $A_-=A_{-1}\cup A_{-1}'$ where $A_{-1}=A_+\cap\{\lambda(\omega)=-1\}$ and $A_1'=A_+\cap \{\lambda(\omega)\neq -1\}$. Suppose that $A_{-1}'$ is such that $\mathbb{P}[A_{-1}']>0$. Then 
we have that 
\[
\mathbb{P}\left[\lim_{t\to\infty} \frac{\int_t^\infty \sigma(s)\,dB(s)}{F^{-1}(t)} \text{ exists and is not equal to $0$}   \right]>0,
\]
which contradicts \eqref{eq.zeroeventSig}. Hence $\mathbb{P}[A_{-1}']=0$. Thus $\mathbb{P}[A_-]=\mathbb{P}[A_{-1}]$. Moreover, we have that 
\begin{equation}  \label{eq.A1minusok}
\lim_{t\to\infty} \frac{\left(\int_t^\infty \sigma(s)\,dB(s)\right)(\omega)}{F^{-1}(t)}= 0, \quad \omega\in A_{-1}.
\end{equation}

Finally, we consider the situation $\omega\in A_0$, so $\lambda(\omega)=0$. Then 
\[
\lim_{t\to\infty} \frac{|X(t,\omega)|}{F^{-1}(t)}=0.
\]
Therefore, as $\varphi$ is odd, and in $\text{RV}_0(\beta)$, we have 
\[
\lim_{t\to\infty} \frac{|\varphi(X(t,\omega)|}{\varphi(F^{-1}(t))}
=
\lim_{t\to\infty} \frac{\varphi(|X(t,\omega)|)}{\varphi(F^{-1}(t))}=0.
\]
Hence
\[
\lim_{t\to\infty} \frac{\varphi(X(t,\omega))}{\varphi(\Phi^{-1}(t))}=0. 
\]
Therefore, because by L'H\^opital's rule and the fact that $f(x)/\varphi(x)\to 1$ as $x\to 0$, we have 
\[
\lim_{t\to\infty} \frac{\int_t^\infty f(X(s,\omega))\,ds}{F^{-1}(t)}
=\lim_{t\to\infty} \frac{\int_t^\infty \varphi(X(s,\omega))\,ds}{\Phi^{-1}(t)}
=\lim_{t\to\infty} \frac{\varphi(X(t,\omega))}{\varphi(\Phi^{-1}(t))}=0.  
\]
Rearranging \eqref{eq.Xasyintsig} and taking limits yields 
\begin{equation}  \label{eq.A0ok}
\lim_{t\to\infty} \frac{\left(\int_t^\infty \sigma(s)\,dB(s)\right)(\omega)}{F^{-1}(t)}= 0, \quad \omega\in A_0.
\end{equation}

Therefore, we have shown that $1=\mathbb{P}[A]=\mathbb{P}[A_1\cup A_1'\cup A_{-1}\cup A_{-1}'\cup A_0]=\mathbb{P}[A_1\cup A_{-1}\cup A_0]$. Therefore, by \eqref{eq.A0ok}, \eqref{eq.A1minusok}, \eqref{eq.A1ok}, and this statement, we have that 
\[
\lim_{t\to\infty} \frac{\int_t^\infty \sigma(s)\,dB(s)}{F^{-1}(t)}=0, \quad \lim_{t\to\infty} \frac{X(t)}{F^{-1}(t)}=\lambda\in \{-1,0,1\}, \quad\text{a.s.},
\]
which proves \eqref{eq.Xperservasy}  and \eqref{eq.intsig2divF}.

\subsection{Proof of Theorem~\ref{thm.Xneccsuff2}}
To prove part (a), we note that the event 
\[
A:=\{\omega\,:\,\lim_{t \to \infty}\frac{X(t,\omega)}{F^{-1}(t)}=\lambda(\omega)\in (-\infty, \infty)\}
\]
is a sub--event of the event $\{\omega\,:\,\lim_{t\to\infty}X(t,\omega)=0\}$. Therefore, if we assume that $\mathbb{P}[A]>0$ it follows that $X(t)$ tends to zero 
with positive probability, and does so for all outcomes in $A$. 

As in the proof of Theorem~\ref{theorem:Xnecc1}, write $A=A_+\cup A_-\cup A_0$. We can use the argument employed in Theorem~\ref{theorem:Xnecc1} to prove that 
\begin{gather*}
\lim_{t\to\infty} \frac{f(X(t,\omega))}{\varphi(\Phi^{-1}(t))}=\lambda(\omega)^\beta, \quad \omega\in A_+,\\
\lim_{t\to\infty} \frac{f(X(t,\omega))}{\varphi(\Phi^{-1}(t))}=-(-\lambda(\omega))^\beta, \quad \omega\in A_-,\\
\lim_{t\to\infty} \frac{f(X(t,\omega))}{\varphi(\Phi^{-1}(t))}=0, \quad \omega\in A_0.
\end{gather*}
Therefore, as $\varphi\circ \Phi^{-1}\in L^1([0,\infty);\mathbb{R})$, it follows that 
\[
\lim_{t\to\infty} \int_0^t f(X(s,\omega))\,ds \text{ exists and is finite for each $\omega\in A$}.
\] 
Since $X(t,\omega)\to 0$ as $t\to\infty$ for each $\omega\in A$, it follows that every term on the right--hand side of \eqref{eq.sigL2rep} 
tends to a finite limit as $t\to\infty$, for each $\omega\in A$. Therefore, it follows that 
\[
\mathbb{P}\left[\lim_{t\to\infty} \int_0^t \sigma(s)\,dB(s) \text{ exists and is finite}\right]>0. 
\]
If $\sigma\not\in L^2([0,\infty);\mathbb{R})$, we have that 
\[
\mathbb{P}\left[\lim_{t\to\infty} \int_0^t \sigma(s)\,dB(s) \text{ exists and is finite}\right]=0, 
\]
a contradiction. Hence the assumption that $\mathbb{P}[A]>0$ must be false, proving part (a). 

To prove (b), part (i), notice that $\mu=0$ in  \eqref{def.mu} together with Lemma~\ref{lemma.tailsigmart} implies 
\begin{align*} 
\lefteqn{
\limsup_{t\to\infty} \frac{\left(\int_t^\infty \sigma(s)\,dB(s)\right)(\omega)^2}{F^{-1}(t)^2}}\\
&=
\limsup_{t\to\infty}
\frac{(\int_t^\infty \sigma(s)\,dB(s))^2(\omega)}{2\int_t^\infty \sigma^2(s)\,ds \log\log\left(\frac{1}{\int_t^\infty \sigma^2(s)\,ds}\right)} \cdot 
\frac{2 \int_t^\infty \sigma^2(s)\,ds \log \log \left(\frac{1}{\int_t^\infty \sigma^2(s)ds}\right)}{F^{-1}(t)^2} \\
&=1\cdot 0=0, \quad\text{a.s.} 
\end{align*}
Therefore, by Theorem~\ref{thm.stochpressuff}, it follows that $X$ obeys \eqref{eq.Xperservasy}, as required. 
To prove part (ii), let us again suppose that the event $A$ defined above is of positive probability. Arguing as in the proof of Theorem~\ref{theorem:Xnecc1}, 
we see that on the event 
\[
A':=A\cap \Omega_2:=A\cap \{\omega\,:\, \left(\lim_{t\to\infty} \int_0^t \sigma(s)\,dB(s)\right)(\omega) \text{ exists and is finite}\}
\]
(which has the same probability as $A$, because $\sigma\in L^2([0,\infty);\mathbb{R})$ ensures that the second event is a.s.) we have 
\[
 \int_t^\infty \sigma(s)\,dB(s)=-X(t)+\int_t^\infty f(X(s))\,ds.
\]
Therefore, defining $A'_+=A_+\cap \Omega_2$, $A'_-=A_-\cap \Omega_2$, and $A_0'=A_0\cap \Omega_2$, we can argue as in Theorem~\ref{theorem:Xnecc1} to show that 
\begin{align*}
\lim_{t\to\infty} \frac{\left(\int_t^\infty \sigma(s)\,dB(s)\right)(\omega)}{F^{-1}(t)}&=-\lambda(\omega)+\lambda(\omega)^\beta, \quad \omega\in A'_+,  \\
\lim_{t\to\infty} \frac{\left(\int_t^\infty \sigma(s)\,dB(s)\right)(\omega)}{F^{-1}(t)}&= -\lambda(\omega)-(-\lambda(\omega))^\beta, \quad \omega\in A'_-,  \\
\lim_{t\to\infty} \frac{\left(\int_t^\infty \sigma(s)\,dB(s)\right)(\omega)}{F^{-1}(t)}&=0, \quad \omega\in A'_0.
\end{align*}
Therefore, it follows that, for all $\omega\in A'$
\[
\lim_{t\to\infty} \frac{\left(\int_t^\infty \sigma(s)\,dB(s)\right)(\omega)}{F^{-1}(t)}=:\Lambda, \quad\text{exists and is finite}. 
\]
Now, by Lemma~\ref{lemma.tailsigmart}, there is an a.s. event $\Omega_3$ such that for all $\omega\in \Omega_3$ we have 
\[
\limsup_{t\to\infty}
 \frac{\left(\int_t^\infty \sigma(s)\,dB(s)\right)(\omega)}{\sqrt{2\int_t^\infty \sigma^2(s)\,ds \log\log\left(\frac{1}{\int_t^\infty \sigma^2(s)\,ds}\right)}}
 =1,
\] 
with the liminf being $-1$. Therefore, for  $\omega\in A'':=A'\cap \Omega_3$, for which $\mathbb{P}[A'']=\mathbb{P}[A]>0$, we have 
\begin{multline*}
\lim_{t\to\infty}
 \frac{\left(\int_t^\infty \sigma(s)\,dB(s)\right)(\omega)}{\sqrt{2\int_t^\infty \sigma^2(s)\,ds \log\log\left(\frac{1}{\int_t^\infty \sigma^2(s)\,ds}\right)}}
 \\
 =
 \lim_{t\to\infty} 
 \frac{\left(\int_t^\infty \sigma(s)\,dB(s)\right)(\omega)}{F^{-1}(t)}
 \cdot 
 \frac{F^{-1}(t)}{\sqrt{2\int_t^\infty \sigma^2(s)\,ds \log\log\left(\frac{1}{\int_t^\infty \sigma^2(s)\,ds}\right)} }
 =\Lambda\frac{1}{\mu}, 
\end{multline*}
where we interpret $1/\mu=0$ in the case when $\mu$ is infinite. But on $A''$ this limit does not exist, giving the required contradiction. 

\subsection{Proof of martingale result}

Our result follows from a number of lemmas. The first is an easily--believable claim.
\begin{lemma} \label{lemma:0}
Let $\delta\in C([0,\infty);(0,\infty))$ and suppose that 
\[
\mathbb{P}\left[ \lim_{t\to\infty} \frac{B(t)}{\delta(t)}=0\right]>0.
\]
Then $\delta^2(t)/t\to \infty$ as $t\to\infty$.
\end{lemma}
Define 
\begin{equation} \label{def.varsigma}
\varsigma(t)=\int_t^\infty \sigma^2(s)\,ds
\end{equation}
If $\varsigma$ is decreasing, there exists $T'>0$ such that we may define 
\begin{equation} \label{def.delta}
\delta(t)=tF^{-1}(\varsigma^{-1}(t)), \quad t\geq T'.
\end{equation}
The assumption that $\varsigma$ be decreasing is relatively mild and quite natural: one important case when it happens is when $\sigma(t)\neq 0$ for all $t\geq T'$, so the equation is always authentically stochastic (more technically, the diffusion coefficient is non--degenerate) for sufficiently large time. 

We now give a result which relates the asymptotic behaviour of $\int_t^\infty \sigma(s)\,dB(s)/F^{-1}(t)$ to that of $B(t)/\delta(t)$ where $\delta$ is defined 
by \eqref{def.delta}. This is clear progress, as it is easier to determine the asymptotic behaviour of $B$ directly rather than the delicate family of 
random variables $\int_t^\infty \sigma(s)\,dB(s)$.
\begin{lemma} \label{lemma:A}
Let $\sigma$ be continuous and in $L^2([0,\infty);\mathbb{R})$. 
Suppose $\varsigma$ in \eqref{def.varsigma} is decreasing. If $\delta$ is defined in \eqref{def.delta}, then 
\[
\mathbb{P}\left[  \lim_{t\to\infty} \frac{\int_t^\infty \sigma(s)\,dB(s)}{F^{-1}(t)}=0\right] 
=\mathbb{P}\left[ \lim_{t\to\infty} \frac{B(t)}{\delta(t)}=0\right].
\]
\end{lemma}
\begin{proof}
Writing $M(t)=\int_t^\infty \sigma(s)\,dB(s)$, the martingale time change theorem asserts the existence of another standard Brownian motion 
$\tilde{B}$ such that $M(t)=\tilde{B}(\langle M\rangle(t))$ for all $t\geq 0$. Notice also that $M$ has an a.s. limit at infinity. Define $T:=\int_0^\infty \sigma^2(s)\,ds=\langle M\rangle(\infty)<+\infty$. Then $\varsigma(t)=T-\langle M\rangle(t)$ for all $t\geq 0$. Since $\int_t^\infty \sigma(s)\,dB(s)=M(\infty)-M(t)$, 
we have
\[
\int_t^\infty \sigma(s)\,dB(s)=M(\infty)-M(t) = \tilde{B}(T)-\tilde{B}(\langle M\rangle(t))=\tilde{B}(T)-\tilde{B}(T-\varsigma(t)).
\]
Since $\varsigma(t)$ is decreasing to $0$ as $t\to\infty$, we have 
\begin{align*}
\lefteqn{
\mathbb{P}\left[\lim_{t\to\infty} \frac{\int_t^\infty \sigma(s)\,dB(s)}{F^{-1}(t)}=0\right]}\\
&=
\mathbb{P}\left[\lim_{t\to\infty} \frac{\tilde{B}(T)-\tilde{B}(T-\varsigma(t))}{F^{-1}(t)}=0\right]\\
&
=
\mathbb{P}\left[\lim_{\tau\downarrow 0} \frac{\tilde{B}(T)-\tilde{B}(T-\tau)}{F^{-1}(\varsigma^{-1}(\tau))}=0\right],
\end{align*}
where we made the substitution $\tau=\varsigma(t)$ at the last step. Now, notice that $B_2(\tau)=\tilde{B}(T-\tau)-\tilde{B}(\tau)$ for $\tau\in [0,T]$ 
is a standard Brownian motion, so therefore
\[
\mathbb{P}\left[\lim_{t\to\infty} \frac{\int_t^\infty \sigma(s)\,dB(s)}{F^{-1}(t)}=0\right]
=
\mathbb{P}\left[\lim_{\tau\downarrow 0} \frac{B_2(\tau)}{F^{-1}(\varsigma^{-1}(\tau))}=0\right].
\]
Since $B_3(t)=tB_2(1/t)$ for $t>0$ and $B_3(0)=0$ is also standard Brownian motion, we make the substitution $\tau=1/t$ and use the definition of $\delta$ in 
\eqref{def.delta} to get
\[
\mathbb{P}\left[\lim_{t\to\infty} \frac{\int_t^\infty \sigma(s)\,dB(s)}{F^{-1}(t)}=0\right]
=
\mathbb{P}\left[\lim_{t\to\infty} \frac{B_3(t)}{\delta(t)}=0\right].
\]
Since $B_3$ has the same distribution as $B$, the claim is proven.
\end{proof}
The first conclusion of the next result can be proven using the argument from Lemma~\ref{lemsig2}.
The second conclusion can be proven in an almost identical manner to Theorem~\ref{theorem:Xnecc1}. The proof is therefore omitted.
\begin{lemma} \label{lemma:B}
Suppose that $f$ is continuous, and obeys \eqref{asym} and \eqref{RVat0} for $\beta>1$. Let $\sigma$ be continuous. 
Suppose that $X$ is the continuous adapted process $X$ which obeys \eqref{eq.sde}. If
\[
\mathbb{P}\left[\lim_{t\to\infty} \frac{X(t)}{F^{-1}(t)}\in (-\infty,\infty)\right]>0,
\]
then $\sigma\in L^2(0,\infty)$ and  
\[
\mathbb{P}\left[  \lim_{t\to\infty} \frac{\int_t^\infty \sigma(s)\,dB(s)}{F^{-1}(t)}=0\right] =1.
\]
\end{lemma}
The following result can be deduced from the Kolmogorov--Erdos characterisation of the law of the iterated logarithm for standard Brownian motion. 
\begin{lemma} \label{lemma:C}
Suppose that $\delta\in C([0,\infty);(0,\infty))$ is such that $t\mapsto \delta^2(t)/t$ is increasing. Then 
the following are equivalent:
\begin{itemize}
\item[(a)] 
\[
\int_1^\infty \frac{1}{t}\exp\left(-\frac{\epsilon^2}{\delta^2(t)/t}\right)\,dt < +\infty, \quad \text{for all $\epsilon>0$};
\]
\item[(b)]
\[
\mathbb{P}\left[ \lim_{t\to\infty} \frac{B(t)}{\delta(t)}=0\right]=1.
\]
\end{itemize}
\end{lemma}
A factor of $s(t):=\delta(t)/\sqrt{t}$, which appears in the Kolmogorov--Erdos characterisation, is omitted here because the dependence of the parameter 
$\epsilon$ enables this subdominant term $s$ to be subsumed into more rapidly decaying the exponential term. 

We make a remarks and then prove our main result. Suppose that $\varsigma$ is increasing. 
Notice from Lemma~\ref{lemma:B} combined with Lemma~\ref{lemma:A} that 
\[
\mathbb{P}\left[\lim_{t\to\infty} \frac{X(t)}{F^{-1}(t)}\in (-\infty,\infty)\right]>0.
\]
implies 
\[
\mathbb{P}\left[ \lim_{t\to\infty} \frac{B(t)}{\delta(t)}=0\right]=1.
\]
Taking this in conjunction with Lemma~\ref{lemma:0} we see that $t\mapsto \delta^2(t)/t$ tends to infinity as $t\to\infty$. Hence, if we want to 
preserve any of the main features of the decay rate of the underlying deterministic equation, even with positive probability, we must demand that
$\delta^2(t)/t\to \infty$ as $t\to\infty$. Strengthening this to ask that the limit is reached monotonically, we may give a deterministic characterisation 
of the preservation of the rate of decay of the solution of $y'(t)=-f(y(t))$ in the solution of \eqref{eq.sde}.
\begin{theorem} \label{theorem:mainsdedetnesssuff}
Suppose that $f$ is continuous, and obeys \eqref{asym} and \eqref{RVat0} for $\beta>1$. Let $\sigma$ be continuous. 
Suppose that $X$ is the continuous adapted process $X$ which obeys \eqref{eq.sde}. Suppose finally that $\varsigma$ defined in \eqref{def.varsigma}
is decreasing and that $\delta$ is the function defined in \eqref{def.delta}.  
\begin{itemize}
\item[(i)] If  
\[
\mathbb{P}\left[\lim_{t\to\infty} \frac{X(t)}{F^{-1}(t)}\in (-\infty,\infty) \right]>0,
\]
then $t\mapsto \delta^2(t)/t\to\infty$ as $t\to\infty$.
\item[(ii)]
If moreover $t\mapsto \delta^2(t)/t$ is increasing,  then the following are equivalent:
\begin{enumerate}
\item[(a)] $\sigma\in L^2([0,\infty);\mathbb{R})$ and 
\[
\int_1^\infty \frac{1}{t}\exp\left(-\frac{\epsilon^2}{\delta^2(t)/t}\right)\,dt < +\infty, \quad \text{for all $\epsilon>0$};
\] 
\item[(b)] 
\[
\mathbb{P}\left[\lim_{t\to\infty} \frac{X(t)}{F^{-1}(t)}\in (-\infty,\infty) \right]>0;
\]
\item[(c)]
\[
\mathbb{P}\left[ \lim_{t\to\infty} \frac{X(t)}{F^{-1}(t)}\in \{-1,0,1\} \right]=1.
\]
\end{enumerate}
\end{itemize}
\end{theorem}
\begin{proof}
We have proved (i) in the discussion above. Now we prove (ii). 
Suppose that (c) holds. Then clearly (b) is true. This implies that $\sigma\in L^2(0,\infty)$. By Lemma~\ref{lemma:B}
we have that
\[
\mathbb{P}\left[  \lim_{t\to\infty} \frac{\int_t^\infty \sigma(s)\,dB(s)}{F^{-1}(t)}=0\right] =1.
\]
Then by Lemma~\ref{lemma:A} we have 
\[
\mathbb{P}\left[ \lim_{t\to\infty} \frac{B(t)}{\delta(t)}=0\right]=1.
\]
Since $t\mapsto \delta^2(t)/t$ is increasing, we have from Lemma~\ref{lemma:C} that 
\[
\int_1^\infty \frac{1}{t}\exp\left(-\frac{\epsilon^2}{\delta^2(t)/t}\right)\,dt < +\infty, \quad \text{for all $\epsilon>0$},
\] 
which proves (a). It remains to show that (a) implies (c). Since (a) holds, by Lemma~\ref{lemma:C}, it follows that 
\[
\mathbb{P}\left[ \lim_{t\to\infty} \frac{B(t)}{\delta(t)}=0\right]=1.
\]
Therefore Lemma~\ref{lemma:A} and the monotonicity of $\varsigma$ gives  
\[
\mathbb{P}\left[  \lim_{t\to\infty} \frac{\int_t^\infty \sigma(s)\,dB(s)}{F^{-1}(t)}=0\right] =1.
\]
Finally, by Theorem~\ref{thm.stochpressuff} it follows that $X$ obeys (c).
\end{proof}

\section{Proof of Theorem~\ref{thm:Xderiv} and~\ref{thm:Xderivconverse}}
To prove Theorem~\ref{thm:Xderiv}, we require preliminary asymptotic estimates on the $h$--increment of the It\^o integral in \eqref{eq.sde}, as well as 
an auxiliary stochastic process with the same diffusion coefficient as \eqref{eq.sde}. We prove that both of these processes are small relative to 
$(f\circ F^{-1})(t)$ as $t\to\infty$ a.s. under the condition that $S_f(\epsilon,h)$ is finite for all $\epsilon>0$. The proof of the converse, Theorem~\ref{thm:Xderivconverse}, is more straightforward and follows in the second subsection.
\subsection{Proof of Theorem~\ref{{thm:Xderiv}}}
As promised, we start with a lemma concerning the asymptotic behaviour of the $h$--increment of the It\^o integral in \eqref{eq.sde}.
\begin{lemma} \label{sigmaIntegral:zero}
Suppose that $f$ is continuous, and obeys \eqref{asym} and \eqref{RVat0} for $\beta>1$. Let $\sigma$ be continuous. 
Let $h>0$ and define $S_f(\epsilon,h)$ as in \eqref{def.Sf}. If $S_f(\epsilon,h)< +\infty$ for all $\epsilon > 0$, then 
\begin{align*}
\lim_{t \to \infty}\frac{ \int_{t}^{t+h}{\sigma(s)dB(s)} }{ (f \circ F^{-1})(t) }=0, \text{ a.s.}
\end{align*}
\end{lemma}
\begin{proof}
Considering $\int_t^{t+h}{\sigma(s)dB(s)}$ we write, for $nh \leq t \leq (n+1)h$,
\begin{align*}
\left|\int_t^{t+h}{\sigma(s)dB(s)}\right| \leq \left|-\int_{(n+1)h}^{t}{\sigma(s)dB(s)}\right| + \left|\int_{(n+1)h}^{t+h}{\sigma(s)dB(s)}\right|.
\end{align*} 
It follows that 
\begin{align*}
\sup_{nh \leq t \leq (n+1)h}\left|\int_t^{t+h}{\sigma(s)dB(s)}\right| \leq &\sup_{nh \leq t \leq (n+1)h} \left|-\int_{(n+1)h}^{t}{\sigma(s)dB(s)}\right| \\ + &\sup_{nh \leq t \leq (n+1)h} \left|\int_{(n+1)h}^{t+h}{\sigma(s)dB(s)}\right|.
\end{align*} 
Similarly we obtain 
\begin{align} \label{eq:***}
\nonumber \sup_{nh \leq t \leq (n+1)h}\left|\int_t^{t+h}{\sigma(s)dB(s)}\right| \leq &\left|-\int_{nh}^{(n+1)h}{\sigma(s)dB(s)}\right| \\ + &\sup_{(n+1)h \leq t \leq (n+2)h} \left|\int_{(n+1)h}^{t}{\sigma(s)dB(s)}\right|.
\end{align}
Here we note that $S_f(\epsilon,h)< +\infty$ implies that 
\begin{align*}
\lim_{t \to \infty}\frac{\int_{nh}^{(n+1)h}{\sigma(s)dB(s)}}{(f \circ F^{-1})(t)} = 0, \quad \text{a.s.}
\end{align*}
by means of the first Borel--Cantelli lemma. Next we proceed to estimate 
\begin{align*}
\mathbb{P}\left[ Z((n+1)h) > \epsilon \right],\,\, \text{for some } \epsilon \in (0,1), 
\end{align*}
where
\begin{align*} 
Z((n+1)h) := \sup_{(n+1)h \leq t \leq (n+2)h} \frac{\left|\int_{(n+1)h}^{t}{\sigma(s)dB(s)}\right|}{(f \circ F^{-1})((n+1)h)}, \,\, n \geq 1.
\end{align*}
We also define the function
\begin{align*}
\tau(t) := \frac{ \int_{(n+1)h}^{t}\sigma^2(s)ds }{ (f \circ F^{-1})^2((n+1)h) }, \text{ for } t \in [(n+1)h, (n+2)h].
\end{align*}
By the Martingale Time Change Theorem there exists a standard Brownian motion $B_n^*$ such that
\begin{align*}
\lefteqn{
\mathbb{P}[ Z((n+1)h) > \epsilon]} \\
&=\mathbb{P}\left[ \sup_{t \in [(n+1)h, (n+2)h]}\left| B_{n+1}^*(\tau(t)) \right| > \epsilon \right] \\
&=\mathbb{P}\left[\sup_{u \in [0, \tau((n+2)h)]} |B_{n+1}^*(u)| > \epsilon \right] \\
&\leq \mathbb{P}\left[\sup_{u \in [0, \tau((n+2)h)]} B_{n+1}^*(u) > \epsilon \right] 
+ \mathbb{P}\left[\sup_{u \in [0, \tau((n+2)h)]} -B_{n+1}^*(u) > \epsilon \right] \\ 
&=\mathbb{P}\left[ |B_{n+1}^*(\tau((n+2)h))| > \epsilon \right] + \mathbb{P}\left[ |B_{n+1}^{**}(\tau((n+2)h))| > \epsilon \right],
\end{align*}
where $-B_{n+1}^* = B_{n+1}^{**}$ is a standard Brownian motion. Thus, as $B_{n+1}^*(\tau((n+2)h))$ is normally distributed with zero mean we have 
\begin{align*}
\lefteqn{
\mathbb{P}[Z((n+1)h) > \epsilon]}\\ 
&\leq 2\mathbb{P}\left[ |B_{n+1}^*(\tau((n+2)h))| > \epsilon \right] 
= 4\mathbb{P}\left[ B_{n+1}^*(\tau((n+2)h)) > \epsilon \right] \\
&= 4\Psi\left(\frac{\epsilon}{\sqrt{\tau((n+2)h)}} \right)  = 4\Psi\left(\frac{\epsilon}{\theta(n+1)} \right).
\end{align*}
But since we assumed that $S_f(\epsilon) < +\infty$, we have
\begin{align*}
\sum_{n=0}^\infty \mathbb{P}[Z((n+1)h) > \epsilon]  \leq 4\sum_{n=0}^\infty \Psi\left( \frac{\epsilon}{\sqrt{\tau((n+2)h)}} \right) 
< +\infty.
\end{align*}
We can then apply the first Borel--Cantelli Lemma to conclude that 
\begin{align*}
\limsup_{n \to \infty}Z((n+1)h)< \epsilon \quad\text{a.s.}
\end{align*}
Thus
\begin{align*}
\lim_{n \to \infty} \sup_{(n+1)h \leq t \leq (n+2)h} 
\frac{\left| \int_{(n+1)h}^{t} \sigma(s)\,dB(s) \right|}{ (f \circ F^{-1})((n+1)h) } = 0 \quad\text{a.s.}
\end{align*}
Combining this with (\ref{eq:***}) we get 
\begin{align*}
\lim_{n \to \infty} \sup_{nh \leq t \leq (n+1)h} \frac{\left| \int_{t}^{t+h} \sigma(s)\,dB(s) \right|}{ (f \circ F^{-1})(nh) } = 0,\quad \text{a.s.},
\end{align*}
as required.
\end{proof}
We next need the asymptotic behaviour of an auxiliary process which solves an affine SDE. 
\begin{lemma} \label{eq.YOUasy}
Suppose that $f$ is continuous, and obeys \eqref{asym} and \eqref{RVat0} for $\beta>1$. Let $\sigma$ be continuous. 
Let $h>0$ and define $S_f(\epsilon,h)$ as in \eqref{def.Sf}. Suppose that $S_f(\epsilon,h)< +\infty$ for all $\epsilon > 0$.
Let $Y$ be the unique continuous adapted process which solves  
\begin{align} \label{Y}
dY(t) = -Y(t)dt + \sigma(t)dB(t), \, t \geq 0, \quad Y(0)=0.
\end{align}
Then
\begin{align}
\lim_{t \to \infty}\frac{Y(t)}{(f \circ F^{-1})(t)} = 0, \quad\text{a.s.}
\end{align} 
\end{lemma}
\begin{proof}
Define 
\[
V_h(n) = \int_{(n-1)h}^{nh} e^{s-nh}\sigma(s)\,dB(s), \quad n\geq 1; \quad \tilde{V}_h(n)=\frac{V_h(n)}{(\varphi \circ \Phi^{-1})(nh)}, \quad n\geq 1.
\]
Then $(\tilde{V}_h(n))_{n\geq 1}$ is a sequence of independent normal random variables with zero mean and variance
\[
\tilde{v}_h^2(n)=\frac{1}{(\varphi\circ \Phi^{-1})(nh)^2}\int_{(n-1)h}^{nh} e^{2s-2nh}\sigma^2(s)\,ds, \quad n\geq 1.
\]
We show first that $\tilde{V}_h(n)\to 0$ a.s. as $n\to\infty$. 
By the fact that $f\circ F^{-1}$ is asymptotic to $\varphi\circ \Phi^{-1}$, there is $N=N(h)$ such that for all $n\geq N(h)$ we have 
\[
\tilde{v}_h^2(n)\leq \frac{1}{(\varphi\circ \Phi^{-1})(nh)^2} \int_{(n-1)h}^{nh} \sigma^2(s)\,ds \leq 4 \theta^2(n-1).
\]
Hence $\tilde{v}_h(n)\leq 2\theta(n-1)$ for $n\geq N(h)$. Also
$\mathbb{P}[|\tilde{V}_h(n)|>\epsilon]=2\mathbb{P}[\tilde{V}_h(n)/\tilde{v}_h(n)>\epsilon/\tilde{v}_h(n)]$, so
\[
\mathbb{P}[|\tilde{V}_h(n)|>\epsilon]=2\Psi(\epsilon/\tilde{v}_h(n)) \leq 2\Psi(\epsilon/\tilde{v}_h(n))\leq 2\Psi(\epsilon/2/\theta(n-1)), \quad n\geq N(h),
\]
since $\Psi$ is decreasing, and $\epsilon/v_h(n)\geq \epsilon/(2\theta(n-1))$ for $n\geq N(h)$. Now, as $S_f(\epsilon/2,h)<+\infty$, it follows that 
\[
\sum_{n=0}^\infty \mathbb{P}[|\tilde{V}_h(n)|>\epsilon] <+\infty
\] 
for every $\epsilon>0$, and hence, by the first Borel--Cantelli lemma, it follows that $\mathbb{P}[\lim_{n\to\infty} \tilde{V}_h(n)=0]=1$. 

Next, we note that as $Y$ is a solution of \eqref{Y}, it obeys 
\begin{align*}
Y(t) = e^{-t}\int_0^t e^s \sigma(s)\, dB(s),\quad t\geq 0.
\end{align*}
Notice that  
$Y((n+1)h) = e^{-h}Y(nh) + V_h(n+1)$ for $n\geq 0$. Now we define $\tilde{Y}(t) := Y(t)/(\varphi \circ \Phi^{-1})(t)$ for $t\geq 0$ and thus we have 
\begin{align*}
\tilde{Y}((n+1)h) &= e^{-h}\frac{Y(nh)}{(\varphi \circ \Phi^{-1})((n+1)h)} + \tilde{V}_h(n+1) \\
&= e^{-h}\frac{(\varphi \circ \Phi^{-1})(nh)}{(\varphi \circ \Phi^{-1})((n+1)h)}\tilde{Y}(nh) + \tilde{V}_h(n+1).
\end{align*}
Hence
\begin{align*}
\tilde{Y}((n+1)h) = a(nh) \tilde{Y}(nh) + \tilde{V}_h(n+1), \quad n\geq 0,
\end{align*}
where $a(nh) := e^{-h}(\varphi \circ \Phi^{-1})(nh) / (\varphi \circ \Phi^{-1})((n+1)h)$. We note that $a(nh) > 0$ for all $n \in \N$ and that $\lim_{n \to \infty}a(nh)=e^{-h}$. Notice that $(1+h/2)e^{-h}<1$ for all $h>0$. Since $a(nh)\to e^{-h}$ as $n\to\infty$, there exists $N_2(h)\in \mathbb{N}$ such that $a(nh)\leq (1+h/2)e^{-h}<1$ for all $n\geq N_2$. Next, we may write, for all $n > N_2(h)$
\begin{align*}
|\tilde{Y}((n+1)h)| &\leq a(nh) |\tilde{Y}(nh)| + |\tilde{V}_h(n+1)| \\ 
&\leq e^{-h}(1+h/2)|\tilde{Y}(nh)| + |\tilde{V}_h(n+1)|.
\end{align*}
From this inequality we define
\begin{gather*}
\bar{Y}((n+1)h) = e^{-h}(1+h/2)\bar{Y}(nh)+ |\tilde{V}_h(n+1)|, \quad n \geq  N_2(h)+1,\\ 
\bar{Y}(nh) = |\tilde{Y}(nh)| +1, \quad n = N_2(h)+1.
\end{gather*}
Therefore we have that $|\tilde{Y}(nh)|<\bar{Y}(nh)$ for $n\geq N_2(h)+1$. 
Since $\tilde{V}_h(n)\to 0$ as $n\to\infty$ a.s., it follows that $\bar{Y}(nh)\to 0$ as $n\to\infty$ a.s. Hence $\tilde{Y}(nh)\to 0$ as $n\to\infty$ a.s.

Next, let $t \in [nh, (n+1)h]$. Then
\begin{align*}
\frac{Y(t)}{(\varphi \circ \Phi^{-1})(t)} = \frac{1}{(\varphi \circ \Phi^{-1})(t)} Y(nh)e^{-(t-nh)} 
+ \frac{e^{-t}}{(\varphi \circ \Phi^{-1})(t)}\int_{nh}^{t} e^s\sigma(s)\,dB(s).
\end{align*}
Notice since $\varphi'(x)\to 0$ as $x\to 0$, and $(\Phi^{-1})'(t)=(\varphi\circ \Phi^{-1})(t)$ that 
$t\mapsto e^{-t}/(\varphi\circ \Phi^{-1})(t)$ is decreasing on $(T_3,\infty)$ for some $T_3>0$. 
Define $N_3\in \mathbb{N}$ such that $N_3h>T_3$. Also $t\mapsto (\varphi\circ \Phi^{-1})(t)$ is decreasing on $[0,\infty)$. 
Then we have for $n\geq N_3$ that $t\geq nh\geq N_3h>T_3$, and so 
\begin{align*}
\lefteqn{
\sup_{t\in [nh,(n+1)h]}
\frac{|Y(t)|}{(\varphi \circ \Phi^{-1})(t)}}
\\ 
&\leq  
\sup_{t\in [nh,(n+1)h]}
\frac{1}{(\varphi \circ \Phi^{-1})(t)} |Y(nh)|  
+ \sup_{t\in [nh,(n+1)h]}\frac{e^{-t}}{(\varphi \circ \Phi^{-1})(t)}\left|\int_{nh}^{t} e^s\sigma(s)\,dB(s)\right|\\
&\leq \frac{|Y(nh)|}{(\varphi \circ \Phi^{-1})((n+1)h)}  
+ \frac{e^{-(n+1)h}}{(\varphi \circ \Phi^{-1})((n+1)h)}\sup_{t\in [nh,(n+1)h]}\left|\int_{nh}^{t} e^s\sigma(s)\,dB(s)\right|.
\end{align*}
Since $Y(nh)/\varphi\circ \Phi^{-1}(nh)\to 0$ as $n\to\infty$ a.s. and $\varphi\circ \Phi^{-1}$ is in $\text{RV}_\infty(-\beta/(\beta-1))$ we have that 
the first term on the right--hand side has zero limit as $n\to\infty$ a.s. Therefore it remains to prove that 
\[
U(n+1):=\frac{e^{-(n+1)h}}{(\varphi \circ \Phi^{-1})((n+1)h)}\sup_{t\in [nh,(n+1)h]}\left|\int_{nh}^{t} e^s\sigma(s)\,dB(s)\right|.
\]
obeys $U(n)\to 0$ as $n\to\infty$ a.s., as this will demonstrate that 
\[
\lim_{n\to\infty} \sup_{t\in [nh,(n+1)h]} \frac{|Y(t)|}{(\varphi\circ \Phi^{-1})(t)}=0, \quad\text{a.s.}
\]
Next, we see that with $\rho(t)=\int_{nh}^{t} e^{2s}\sigma^2(s)\,ds$ for $t\in [nh,(n+1)h)$, by the martingale time change theorem, 
there exists a standard Brownian motion $B^\ast_n$ such that
\begin{align*}
\lefteqn{\mathbb{P}[U(n+1)>\epsilon]}\\
&=\mathbb{P}\left[\sup_{t\in [nh,(n+1)h]}\left|\int_{nh}^{t} e^s\sigma(s)\,dB(s)\right|>\epsilon \frac{e^{(n+1)h}}{(\varphi \circ \Phi^{-1})((n+1)h)} \right]\\
&=\mathbb{P}\left[\sup_{t\in [nh,(n+1)h]}|B^\ast_n(\rho(t))|>\epsilon \frac{e^{(n+1)h}}{(\varphi \circ \Phi^{-1})((n+1)h)} \right]\\
&=\mathbb{P}\left[\sup_{t\in [0,\rho((n+1)h)]}|B^\ast_n(t)|>\epsilon \frac{e^{(n+1)h}}{(\varphi \circ \Phi^{-1})((n+1)h)} \right].
\end{align*}
By standard arguments, we get that 
\begin{align*}
\mathbb{P}[U(n+1)>\epsilon]
&\leq 2\mathbb{P}\left[ \sup_{t\in [0,\rho((n+1)h)]} B^\ast_n(t) >\epsilon \frac{e^{(n+1)h}}{(\varphi \circ \Phi^{-1})((n+1)h)}  \right]\\
&= 2\mathbb{P}\left[|B^\ast_n(\rho((n+1)h))| >\epsilon \frac{e^{(n+1)h}}{(\varphi \circ \Phi^{-1})((n+1)h)}\right]\\
&=4\Psi\left(\epsilon \frac{e^{(n+1)h}}{(\varphi \circ \Phi^{-1})((n+1)h)} \frac{1}{\sqrt{\rho((n+1)h)}}  \right).
\end{align*}
Finally, we estimate the right--hand side of the above expression. Since 
\[
\sqrt{\rho((n+1)h)}=\left(\int_{nh}^{(n+1)h} e^{2s}\sigma^2(s)\,ds\right)^{1/2}\leq e^{(n+1)h}\left(\int_{nh}^{(n+1)h} \sigma^2(s)\,ds\right)^{1/2},
\]
we have
\begin{multline*}
\epsilon\frac{e^{(n+1)h}}{(\varphi \circ \Phi^{-1})((n+1)h)}\cdot 
\frac{1}{\sqrt{\rho((n+1)h)}}
\\
\geq 
\epsilon 
\frac{(f\circ F^{-1})(nh)}{(\varphi \circ \Phi^{-1})((n+1)h)}
\cdot\frac{1}{(f\circ F^{-1})(nh)} \left(\int_{nh}^{(n+1)h} \sigma^2(s)\,ds\right)^{-1/2}.
\end{multline*}
Next, the fact that $f\circ F^{-1}$ is asymptotic to $\varphi\circ\Phi^{-1}$ and that both are regularly varying functions 
means there is an $N_4(h)\in\mathbb{N}$ such that
\[
\frac{(f\circ F^{-1})(nh)}{(\varphi \circ \Phi^{-1})((n+1)h)}\geq \frac{1}{2}, \quad n\geq N_4(h). 
\]
Hence by the definition of $\theta(n)$, for $n\geq N_4(h)$, we get
\[
\epsilon\frac{e^{(n+1)h}}{(\varphi \circ \Phi^{-1})((n+1)h)}\cdot 
\frac{1}{\sqrt{\rho((n+1)h)}}
\geq 
\frac{\epsilon}{2} \frac{1}{\theta(n)},
\]
so as $\Psi$ is decreasing, we have for $n\geq N_4(h)$ that
\begin{align*}
\mathbb{P}[U(n+1)>\epsilon]
\leq 4\Psi\left(\epsilon \frac{e^{(n+1)h}}{(\varphi \circ \Phi^{-1})((n+1)h)} \frac{1}{\sqrt{\rho((n+1)h)}}  \right)
\leq 4\Psi\left(\frac{\epsilon/2}{\theta(n)}\right).
\end{align*}
Since $S_f(\epsilon/2,h)<+\infty$ for all $\epsilon>0$, it follows that 
\[
\sum_{n=1}^\infty 
\mathbb{P}[U(n+1)>\epsilon] < +\infty
\]
for every $\epsilon>0$, and therefore, by the first Borel--Cantelli lemma, it follows that $\mathbb{P}[\lim_{n\to\infty} U(n)=0]=1$.
As noted above, this is the remaining fact that guarantees that $\lim_{t\to\infty} Y(t)/(f\circ F^{-1}(t))=0$ a.s., as required.
\end{proof}  
We now have all the ingredients to prove Theorem~\ref{thm:Xderiv}.
\begin{proof}[Proof of Theorem~\ref{thm:Xderiv}]
Since $S_f(\epsilon,h) < +\infty$ for all $\epsilon > 0$, by Lemma~\ref{eq.YOUasy}, we have that 
\[
\lim_{t\to\infty} \frac{Y(t)}{(f\circ F^{-1})(t)}=0, \quad \text{a.s.}
\]
Consider $Z(t)=X(t)-Y(t)$ for $t\geq 0$. Since $(f\circ F^{-1})(nh)\to 0$ as $n\to\infty$, it follows that $S_f(\epsilon,h)<+\infty$ for all $\epsilon>0$ 
this implies that 
\[
S(\epsilon,h):=\sum_{n=1}^\infty \Psi\left(\frac{\epsilon}{\sqrt{\int_{nh}^{(n+1)h} \sigma^2(s)\,ds}}\right)<+\infty.
\]
This has been shown in \cite{appchengrod:2012} to give $X(t)\to 0$ as $t\to\infty$ a.s. Therefore, we have that $Z(t)\to 0$ as $t\to\infty$ a.s.
Next, we have that $Z'(t)=-f(Z(t)+Y(t))+Y(t)$ for $t\geq 0$. Now, given that $f$ is continuous, we have that 
\[
Z'(t)=-f(Z(t))+g(t), 
\]
where $g$ is continuous and is given by $g(t)=f(Z(t))-f(Z(t)+Y(t))+Y(t)$ for $t\geq 0$. Now, since $f$ is locally Lipschitz continuous, it follows that 
there is $K_\delta>0$ such that $|f(x)-f(y)|\leq K_1|x-y|$ for all $|x|,|y|\leq \delta$. Since $Z(t)\to 0$ and $Y(t)\to 0$ as $t\to\infty$, it follows that 
$|Z(t)|\leq \delta/2$, $|Y(t)|\leq \delta/2$ for all $t\geq T_1$. Therefore $|f(Z(t))-f(Z(t)+Y(t))|\leq K_\delta|Y(t)|$ for $t\geq T_1$. Hence  
$|g(t)|\leq (1+K_\delta)|Y(t)|$ for $t\geq T_1$. Therefore, we have that 
\[
\lim_{t\to\infty} \frac{g(t)}{(f\circ F^{-1})(t)}=0, \quad \text{a.s.}
\]
Thus, by Theorem~\ref{thm.detpressuff}, we have that for each outcome in an a.s. event we have that 
$Z(t,\omega)/F^{-1}(t)\to \lambda(\omega)\in \{-1,0,1\}$ as $t\to\infty$. Since $f(x)/x\to 0$ as $x\to 0$, it 
follows that $Y(t)/F^{-1}(t)\to 0$ as $t\to\infty$ a.s., so therefore we have that  
\begin{align*}
\lim_{t \to \infty}\frac{X(t,\omega)}{F^{-1}(t)} = \lambda(\omega)\in \{-1,0,1\}, 
\end{align*}
for every outcome $\omega$ in some a.s. event. This is the first limit in \eqref{eq:Xderiv}. Next let 
\begin{align*}
\Omega_0 := \left\{\omega : \text{ a solution $X(\cdot,\omega)$ exists and } \lim_{t \to \infty}X(t,\omega)=0 \right\}.
\end{align*}
We further define the events 
\begin{align*}
&A_0 := \left\{\omega : \lim_{t \to \infty}\frac{X(t,\omega)}{F^{-1}(t)}=0 \right\} \cap \Omega_0, \\
&A_1^+ := \left\{\omega : \lim_{t \to \infty}\frac{X(t,\omega)}{F^{-1}(t)}=1 \right\} \cap \Omega_0, \\
&A_1^- := \left\{\omega : \lim_{t \to \infty}\frac{X(t,\omega)}{F^{-1}(t)}=-1 \right\} \cap \Omega_0, 
\end{align*}
with $A_1 := A_1^+ \cup A_1^-$. Therefore we have that $\Omega_1 := A_0 \cup A_1$ is a.s. and since $S_f(\epsilon,h)<\infty$ for all $\epsilon>0$, by 
Lemma~\ref{sigmaIntegral:zero}, we have
\begin{align*}
\lim_{t \to \infty}\frac{\int_t^{t+h} \sigma(s)\,dB(s)}{(f \circ F^{-1})(t)}= 0, \quad \text{a.s.}
\end{align*}
Let the event on which this limit holds be $\Omega_2$ and let $\Omega_3 = \Omega_0 \cap \Omega_2$. For $\omega \in \Omega_3$ we have 
\begin{align*}
\lim_{t \to \infty}\frac{\int_t^{t+h}{f(X(s))ds}}{(f \circ F^{-1})(t)} = \lim_{t \to \infty}\frac{\int_t^{t+h}{\varphi(X(s))ds}}{(\varphi \circ F^{-1})(t)},
\end{align*}
where $\varphi$ is odd,  increasing and $\varphi \in RV_0(\beta)$. If $\omega \in A_0$, then $X(t,\omega)/F^{-1}(t) \rightarrow 0$ as $t \to \infty$. Thus 
\begin{align*}
\lim_{t \to \infty}\frac{|\varphi(X(t,\omega))|}{\varphi(F^{-1}(t))} = \lim_{t \to \infty}\frac{\varphi(|X(t,\omega)|)}{\varphi(F^{-1}(t))} = 0.
\end{align*}
Hence as $\varphi \circ F^{-1} \in \text{RV}_0(-\beta/(\beta-1))$, we have 
\begin{align*}
\lim_{t \to \infty}\frac{\int_t^{t+h}{\varphi(X(s,\omega))ds}}{\varphi(F^{-1}(t))} = 0,
\end{align*}
and therefore we have 
\begin{align*}
\lim_{t \to \infty}\frac{\int_t^{t+h}{f(X(s,\omega))ds}}{(f\circ F^{-1})(t)} = 0, \quad \omega \in A_0 \cap \Omega_3.
\end{align*}
Hence for $\omega \in A_0 \cap \Omega_3$, we have
\begin{align*}
\lim_{t\to\infty}
\frac{\frac{X(t+h,\omega)-X(t,\omega)}{h}}{(f \circ F^{-1})(t)} 
&= \lim_{t\to\infty} 
\frac{\frac{1}{h}\int_t^{t+h} -f(X(s))\,ds }{(f \circ F^{-1})(t)} + \frac{\frac{1}{h}\int_t^{t+h} \sigma(s)\,dB(s) }{(f \circ F^{-1})(t)} \\
&= 0 = -\lambda(\omega). 
\end{align*}

If $\omega \in A_1^+ \cap \Omega_3$, we have $\lim_{t \to \infty} X(t,\omega)/F^{-1}(t) = 1$, so as $\varphi$ is regularly varying, it follows that
$\lim_{t \to \infty} \varphi(X(t,\omega))/\varphi(F^{-1}(t)) = 1$. Hence 
\begin{align*}
\lim_{t \to \infty}\frac{\int_t^{t+h} \varphi(X(s,\omega))\,ds}{\varphi(F^{-1}(t))} = h,
\end{align*}
and so 
\begin{align*}
\lim_{t \to \infty}\frac{-\frac{1}{h}\int_t^{t+h} \varphi(X(s,\omega))\,ds}{\varphi(F^{-1}(t))} = -1.
\end{align*}
Therefore, for $\omega \in A_1^+ \cap \Omega_3$, we have 
\begin{align*}
\lim_{t \to \infty}\frac{\frac{X(t+h,\omega)-X(t,\omega)}{h}}{(f \circ F^{-1})(t)} = - 1 = -\lambda(\omega).
\end{align*}

Similarly, for $\omega \in A_1^- \cap \Omega_3$, we use the fact that $\varphi$ is odd to get
\begin{align*}
\lim_{t \to \infty}\frac{\varphi(X(t,\omega))}{\varphi(F^{-1}(t))} = \lim_{t \to \infty}\frac{-\varphi(-X(t,\omega))}{\varphi(F^{-1}(t))} = -1,
\end{align*}
from which we obtain, for $\omega \in A_1^- \cap \Omega_3$, 
\begin{align*}
\lim_{t \to \infty}\frac{\frac{X(t+h,\omega)-X(t,\omega)}{h}}{(f \circ F^{-1})(t)} = 1 = -\lambda(\omega).  
\end{align*}

Thus for $\omega\in  A_1 \cap \Omega_3$ we have  
\begin{align*}
\lim_{t \to \infty}\frac{\frac{X(t+h,\omega)-X(t,\omega)}{h}}{(f \circ F^{-1})(t)} =-\lambda(\omega),
\end{align*}
and so for all $\omega\in \Omega_1\cap\Omega_3$ we have 
\begin{equation*}
\lim_{t \to \infty}\frac{\frac{X(t+h,\omega)-X(t,\omega)}{h}}{(f \circ F^{-1})(t)} =-\lambda(\omega).
\end{equation*}
But since $\Omega_1$ and $\Omega_3$ are a.s. events, we have the second limit in \eqref{eq:Xderiv}, as required. 
\end{proof}

\subsection{Proof of Theorem~\ref{thm:Xderivconverse}}
Define $A=A_1\cup A_{-1}\cup A_0$. If we are in the case when $\omega\in A_1$, 
then 
\begin{align*}
\lim_{t \to \infty}\frac{X(t,\omega)}{F^{-1}(t)} = 1.
\end{align*}
Hence $X(t,\omega) \sim F^{-1}(t)$ as $t \to \infty$ and so $f(X(t,\omega)) \sim (f \circ F^{-1})(t)$ as $t \to \infty$. Therefore we have that
\begin{align*}
\lim_{t \to \infty}\frac{\frac{1}{h} \int_t^{t+h} f(X(s,\omega))\,ds}{(f \circ F^{-1})(t)} = 1.
\end{align*}
By hypothesis we know that 
\begin{align*}
\lim_{t \to \infty}\frac{-\frac{1}{h} \int_t^{t+h} f(X(s,\omega))\,ds}{(f \circ F^{-1})(t)} + 
\lim_{t \to \infty}\frac{\frac{1}{h} \int_t^{t+h} \sigma(s)\,dB(s)}{(f \circ F^{-1})(t)} = -\lambda(\omega) = -1, \text{ a.s. on $A_1$}.
\end{align*}
Thus we can conclude that 
\begin{align*}
\lim_{t \to \infty}\frac{\int_t^{t+h}{\sigma(s)\,dB(s)}}{(f \circ F^{-1})(t)} = 0, \text{ a.s. on $A_1$}.
\end{align*}
By the same argument (and using Lemma~\ref{asym_odd})) we can show that on $A_{-1}$, we get  
\begin{align*}
\lim_{t \to \infty}\frac{ \int_t^{t+h} \sigma(s)\,dB(s)}{(f \circ F^{-1})(t)} = 0, \text{ a.s. on  $A_{-1}$}.
\end{align*}
A similar limit applies for $A_0$:
\begin{align*}
\lim_{t \to \infty}\frac{ \int_t^{t+h} \sigma(s)\,dB(s)}{(f \circ F^{-1})(t)} = 0, \text{ a.s. on  $A_0$}.
\end{align*}
Therefore as the limit applies to $A_1, A_{-1}$ and $A_0$, it applies to all of $A$, a.s., and in particular along the sequence of times $nh$, for $n\geq 1$:
\begin{align} \label{integral_zero}
\lim_{t \to \infty}\frac{\int_{nh}^{(n+1)h} \sigma(s)\,dB(s)}{(f \circ F^{-1})(nh)} = 0, \text{ a.s. on $A$}.
\end{align}
Since $A$ is an event of positive probability and the random variables 
\begin{align*}
\tilde{Y_n} = \frac{\int_{nh}^{(n+1)h} \sigma(s)\,dB(s)}{(f \circ F^{-1})(nh)}
\end{align*}
are independent, the convergence in (\ref{integral_zero}) is a.s. by the Zero--One Law. Moreover, since the $\tilde{Y_n}$ are independent, the Borel--Cantelli Lemmas force $S_f(\epsilon,h) < +\infty$ for all $\epsilon>0$, as claimed.

\section{Proofs from Examples Section}
\subsection{Proof of Lemma~\ref{lemma.osyexamp}}
Let $j\in \mathbb{N}$ and consider 
\[
\int_{2j\pi}^{2(j+1)\pi} k_0(s)\,ds 
= \int_{0}^{\pi} k(u+2\pi j) \sin(u)\,du + \int_{\pi}^{2\pi} k(v+2\pi j) \sin(v)\,dv. 
\]
Now 
\begin{multline*}
\int_{\pi}^{2\pi} k(v+2\pi j) \sin(v)\,dv \\
= \int_{0}^\pi k(u+\pi+2\pi j)\sin(u+\pi)\,du = -\int_0^\pi k(u+2\pi j + \pi)\sin(u)\,du. 
\end{multline*}
Therefore
\[
\int_{2j\pi}^{2(j+1)\pi} k_0(s)\,ds 
= \int_{0}^{\pi} \{k(u+2\pi j)-k(u+2\pi j+\pi)\} \sin(u)\,du. 
\]
Hence 
\begin{align*}
\lefteqn{
\left|
\frac{1}{-k'(2j\pi)\cdot \pi} \int_{2j\pi}^{2(j+1)\pi} k_0(s)\,ds - \int_0^\pi \sin(u)\,du\right|}\\
&\leq 
\int_0^\pi \left|\frac{\{k(u+2\pi j)-k(u+2\pi j+\pi)\}}{-k'(2j\pi)\cdot \pi} -1\right| |\sin(u)|\,du\\
&\leq 
\pi \sup_{u\in [0,\pi]} \left|\frac{\{k(u+2\pi j)-k(u+2\pi j+\pi)\}}{-k'(2j\pi)\cdot \pi} -1\right|.
\end{align*}
By the mean value theorem, for any $u\in [0,\pi]$, there is $\xi_{j,u}\in [0,\pi]$ such that we have 
\begin{multline*}
\left|
\frac{k(u+2\pi j)-k(u+2\pi j+\pi)}{-k'(2j\pi)\cdot \pi}-1\right|=\left|\frac{k'(u+2\pi j+\xi_{u,j})}{k'(2j\pi)}-1\right|
\\\leq \sup_{v\in [0,2\pi]} \left|\frac{k'(v+2\pi j)}{k'(2j\pi)}-1\right|.
\end{multline*}
Since $\int_0^\pi \sin(u)\,du=2$, we have
\begin{align*}
\lefteqn{
\left|
\frac{1}{-k'(2j\pi)\cdot \pi} \int_{2j\pi}^{2(j+1)\pi} k_0(s)\,ds - 2\right|
}\\
&\leq 
\pi \sup_{v\in [0,2\pi]} \left|\frac{k'(v+2\pi j)}{k'(2j\pi)}-1\right|.
\end{align*}
By hypothesis, we therefore have that 
\begin{equation}\label{eq.k1}
\lim_{j\to\infty} \frac{1}{-k'(2j\pi)} \int_{2j\pi}^{2(j+1)\pi} k_0(s)\,ds = 2\pi.
\end{equation}
Next 
\[
\frac{k(2j\pi)-k(2j\pi+2\pi)}{-k'(2j\pi)} - 2\pi
=
\int_{2j\pi}^{2j\pi+2\pi} \left\{\frac{k'(s)}{k'(2j\pi)}-1\right\}\,ds,
\]
so 
\[
\lim_{j\to\infty} \frac{(k(2j\pi)-k(2j\pi+2\pi))}{-2\pi k'(2j\pi)} = 1.
\]
Combining this with \eqref{eq.k1} gives
\[
\lim_{j\to\infty} \frac{1}{k(2j\pi)-k(2j\pi+2\pi)} \int_{2j\pi}^{2(j+1)\pi} k_0(s)\,ds = 1.
\]
Since $k(t)\to 0$ as $t\to\infty$, by Toeplitz lemma,
\begin{equation}  \label{eq.k2}
\lim_{n\to\infty}
\frac{\int_{2n\pi}^\infty k_0(s)\,ds}{k(2\pi n)}
=
\lim_{n\to\infty}
\frac{\sum_{j=n}^\infty \int_{2j\pi}^{2(j+1)\pi} k_0(s)\,ds}{\sum_{j=n}^\infty \frac{1}{2}(k(2j\pi)-k(2j\pi+2\pi))}
=1.
\end{equation}
\eqref{eq.k2} demonstrates that the first part of (i) is valid. We now use it to prove part (ii). To do so, 
let $n(t)$ be the largest integer less than or equal to $t/(2\pi)$ i.e., $n(t)=\lfloor t/(2\pi) \rfloor$. Then
\[
\frac{\int_{2\pi(n(t)+1)}^\infty k_0(s)\,ds}{k(t)}
=\frac{\int_{2\pi(n(t)+1)}^\infty k_0(s)\,ds}{k(2\pi (n(t)+1))}\cdot \frac{k(2\pi (n(t)+1))}{k(t)} \to 1
\]
as $t\to\infty$. Also
\[
\lim_{t\to\infty} \left\{\frac{\int_t^{2\pi(n(t)+1)} k_0(s)\,ds}{k(t)}-\int_t^{2\pi(n(t)+1)} \sin(u)\,du\right\} = 0.
\]
Therefore, as $\int_t^{2\pi(n(t)+1)} \sin(u)\,du=\cos(t)-\cos(2\pi(n(t)+1))=\cos(t)-1$, we have 
\[
\lim_{t\to\infty} \left\{\frac{\int_t^\infty k_0(s)\,ds}{k(t)}-\cos(t)\right\}=0. 
\]
Therefore, we see that part (ii) is true. The proof of the second part of (i) 

Let $j\in \mathbb{N}$; noting that
\[
\int_{\pi}^{2\pi} k(v+2\pi j) |\sin(v)|\,dv = \int_0^\pi k(u+2\pi j + \pi)|\sin(u)|\,du. 
\]
we see that  
\begin{align*}
\int_{2j\pi}^{2(j+1)\pi} |k_0(s)|\,ds 
&= \int_{0}^{\pi} k(u+2\pi j) \sin(u)\,du + \int_{\pi}^{2\pi} k(v+2\pi j) |\sin(v)|\,dv\\
&= 2\int_0^{\pi} k(u+2\pi j) \sin(u)\,du.
\end{align*}
Arguing as before, we see that 
\[
\lim_{j\to\infty} \frac{\int_{2j\pi}^{2(j+1)\pi} |k_0(s)|\,ds}{k(2\pi j)}=2\int_0^\pi \sin(u)\,du=4.
\]
Also
\[
\lim_{j\to\infty} \frac{\int_{2j\pi}^{2(j+1)\pi} k(s)\,ds}{k(2\pi j)}=2\pi.
\]
Therefore
\[
\lim_{j\to\infty} \frac{\int_{2j\pi}^{2(j+1)\pi} |k_0(s)|\,ds}{\int_{2j\pi}^{2(j+1)\pi} k(s)\,ds}
=\frac{4}{2\pi}.
\] 
Hence by Toeplitz lemma, we have 
\[
\lim_{n\to\infty} 
\frac{\int_{0}^{2n\pi} |k_0(s)|\,ds}{\int_{0}^{2n\pi} k(s)\,ds} = 
\lim_{n\to\infty}
\frac{\sum_{j=0}^n \int_{2j\pi}^{2(j+1)\pi} |k_0(s)|\,ds}{\sum_{j=0}^n \int_{2j\pi}^{2(j+1)\pi} k(s)\,ds}
=\frac{4}{2\pi},
\]
and so $\lim_{t\to\infty} \int_0^t |k_0(s)|\,ds=+\infty$, as required.
Now 

\subsection{Proof of Theorem~\ref{thm.goscill}} 
Suppose $n$ is an integer such that $n\geq (2\beta-1)/(\beta-1)$, and let 
\[
g(t)=\Gamma(t) \sin\left(\left\{\int_0^t \Gamma(s)\,ds \right\}^n \right), \quad t\geq 0.
\]
Since $\Gamma$ is continuous, so is $g$. Moreover $I(t):=\int_0^t \Gamma(s)\,ds$ is in $C^1((0,\infty);(0,\infty))$ and it obeys $I(t)\to\infty$ as $t\to\infty$. Let $0\leq t<T$. Then, using integration by substitution, we obtain 
\[
\int_t^T g(s)\,ds =\int_{I(t)^n}^{I(T)^n} \frac{1}{n}u^{-(1-1/n)} \sin(u)\,du.
\]
If we identify $k(t)=t^{-(1-1/n)}/n$, and let $k_0(t)=k(t)\sin(t)$ for $t\geq 1$, it can be seen that $k$ obeys all the properties of Lemma~\ref{lemma.osyexamp}, and therefore that 
\[
\lim_{t\to\infty} \int_1^t k_0(s)\,ds =:K^\ast, \quad \lim_{t\to\infty} \int_1^t |k_0(s)|\,ds = +\infty.
\] 
Since $I(T)\to\infty$ as $T\to\infty$, $g$ is continuous on $[0,1]$ and 
\[
\int_1^T g(s)\,ds = \int_{I(1)^n}^{I(T)^n} k_0(u)\,du, \text{ and } \int_1^T |g(s)|\,ds = \int_{I(1)^n}^{I(T)^n} |k_0(u)|\,du,
\]
we have that $g$ obeys
\[
\lim_{t\to\infty} \int_0^t g(s)\,ds =K^\ast+\int_0^1 g(s)\,ds, \text{ and } \lim_{t\to\infty} \int_0^t |g(s)|\,ds = +\infty.
\]
Of course, the first limit implies that $\int_t^\infty g(s)\,ds\to 0$ as $t\to\infty$.
Clearly by construction $\limsup_{t\to\infty} |g(t)|/\Gamma(t)=1$, and $g$ has infinitely many changes of sign because $I(t)^n$, the argument of $\sin$ in $g$, tends to infinity as $t\to\infty$. 

Finally, we determine the asymptotic behaviour of $\int_t^\infty g(s)\,ds$ as $t\to\infty$. Since $I(t)\to\infty$ as $t\to\infty$, 
by Lemma~\ref{lemma.osyexamp}, we have that 
\[
\limsup_{t\to\infty} \frac{\left|\int_t^\infty g(s)\,ds\right|}
{k(I(t)^n)} 
=\limsup_{t\to\infty} \frac{\left|\int_{I(t)^n}^\infty k_0(u)\,du\right|}{k(I(t)^n)}=1.
\]
Since $k(I(t)^n)=I(t)^{-(n-1)}/n$, we have 
\begin{equation} \label{eq.gversusIn}
\limsup_{t\to\infty} \frac{\left|\int_t^\infty g(s)\,ds\right|}{\frac{1}{n}I(t)^{-(n-1)}} =1.
\end{equation}
Next, as $\Gamma(t)\to\infty$ as $t\to\infty$, it follows that $I(t)/t\to \infty$ as $t\to\infty$, so there exists $T_1$ such that $I(t)>t$ for all $t\geq T_1$.
Therefore $I(t)^{-(n-1)}<t^{-(n-1)}$ for $t\geq T_1$. Since $F^{-1}\in \text{RV}_\infty(-1/(\beta-1))$, we have that 
\[
\lim_{t\to\infty}\frac{\log F^{-1}(t)}{\log t}=-\frac{1}{\beta-1}.
\] 
Hence there is a $T_2>0$ such that 
$F^{-1}(t)>t^{-1/(\beta-1)-1/2}$ for $t\geq T_2(\epsilon)$. Now let $T_3=\max(T_1,T_2)$. For $t\geq T_3$, we have
\[
\frac{I(t)^{-(n-1)}}{F^{-1}(t)}\leq t^{-(n-1)} \cdot t^{1/(\beta-1)+1/2},
\]
and as $n\geq (2\beta-1)/(\beta-1)$, the right--hand side of this expression tends to zero as $t\to\infty$. Combining this with \eqref{eq.gversusIn} 
gives the second part of \eqref{eq.intgdivF}, as claimed. Therefore, Theorem~\ref{thm.detpressuff} applies to the solution $x$ of \eqref{eq.odepert}, 
as claimed.

\subsection{Proof of Lemma~\ref{lemma.gspikes}}
First we note some key properties of $h_s$ which will be used extensively:
\begin{align*}
&h(0,a,b) = 0, \,\,\frac{d}{dx}h_s(x,a,b)= 0 \text{ for } x = 0,a,b, \\
&\frac{d}{dx}h_s(x,a,b) > 0 \text{ for } x \in (0,a) \text{ and }
\frac{d}{dx}h_s(x,a,b) < 0 \text{ for } x \in (a,2a).
\end{align*}
Thus $h_s(x,a,b) \leq h_s(a,a,b) = b$ for $x \in [0,2a]$.
Now we proceed to show (\ref{max}):
\begin{align*}
k(n + \frac{w_n}{2}) &= k_s(n+\frac{w_n}{2}) + h_s(\frac{w_n}{2},\frac{w_n}{2},\Gamma_+(n+\frac{w_n}{2})-k_s(n+\frac{w_n}{2})) \\
&= k_s(n+\frac{w_n}{2}) + \Gamma_+(n+\frac{w_n}{2})-k_s(n+\frac{w_n}{2}) = \Gamma_+(n+\frac{w_n}{2}),
\end{align*}
which is valid since $\Gamma_+(t) > g_s(t), \, t \geq 0$. Therefore, with $t_n=n+\frac{w_n}{2}$,
\begin{align*}
\limsup_{t \to \infty}\frac{k(t)}{\Gamma_+(t)} \geq \limsup_{n \to \infty}\frac{k(t_n)}{\Gamma_+(t_n)} = 1.
\end{align*}
Considering $t \in [n, n+w_n]$, we obtain
\begin{align*}
k(t) &= k_s(t) + h_s(t-n,\frac{w_n}{2},\Gamma_+(t)-k_s(t)) \\
&\leq k_s(t) + \Gamma_+(t)-k_s(t) = \Gamma_+(t).
\end{align*}
For $t \in [n+w_n, n+1]$, $k(t) = k_s(t) < \Gamma_+(t).$ Therefore $k(t) \leq \Gamma_+(t)$ for all $t \geq 0$. Thus
\begin{align*}
1 \leq \limsup_{t \to \infty}\frac{k(t)}{\Gamma_+(t)} \leq 1, \text{ so } \limsup_{t \to \infty}\frac{k(t)}{\Gamma(t)} = 1,
\end{align*}
since $\Gamma_+(t) \sim \Gamma(t)$ as $t \to \infty$. \\[5pt] 
Given $k_s \in C^1(0,\infty)$, to show that $k \in C^1(0,\infty)$ we just need to show that it is $C^1$ at the points of transition. Write $h_s(x,a,b) = b \, \tilde{h_s(x,a)}$. Then
\begin{align*}
\frac{d}{dt}h_s(t-n,\frac{w_n}{2},\Gamma_+(t)-k_s(t)) &= \frac{\partial}{\partial x}h_s(t-n,\frac{w_n}{2},\Gamma_+(t)-k_s(t)) \\
&+ \frac{\partial}{\partial t}h_s(t-n,\frac{w_n}{2},\Gamma_+(t)-k_s(t)).(\Gamma_+'(t)-k_s'(t)) \\
&= (\Gamma_+(t)-k_s(t)).\tilde{h_s'}(t-n,\frac{w_n}{2}) \\
&+ \tilde{h_s}(t-n,\frac{w_n}{2})(\Gamma_+'(t)-k_s'(t)).
\end{align*}
Now as $\tilde{h_s'}(0,a)=\tilde{h_s'}(a,a)=\tilde{h_s'}(2a,a)=0$ and $\tilde{h_s}(0,a)=\tilde{h_s}(2a,a)=0$ we have 
\begin{align*}
k'(t) =
\begin{cases}
&k_s'(t) + (\Gamma_+(t)-k_s(t)).\tilde{h_s'}(t-n,\frac{w_n}{2}) 
+ \tilde{h_s}(t-n,\frac{w_n}{2})(\Gamma_+'(t)-k_s'(t)), \\ &t \in [n,n+w_n), \\
&k_s'(t), \, t \in [n+w_n,n+1].
\end{cases}
\end{align*}
Hence $\lim_{t \downarrow n}k'(t) = k'(n)$ and 
$\lim_{t \uparrow n+w_n}k'(t) = k'(n+w_n).$
Similarly, we have \\ $\lim_{t \downarrow n}k(t) = k_s(n) + h_s(0,\frac{w_n}{2},\Gamma_+(n)-k_s(n)) = k_s(n) = k(n)$ and 
\begin{align*}
\lim_{t \uparrow n+w_n}k(t) &= k_s(n+w_n) + h_s(w_n,\frac{w_n}{2},\Gamma_+(n+w_n)-k_s(n+w_n)) \\
&= k_s(n+w_n) + \{ \Gamma_+(n+w_n)-k_s(n+w_n) \}\tilde{h_s}(w_n,\frac{w_n}{2}) \\
&= k_s(n+w_n) = k(n+w_n).
\end{align*}
Thus we have $k \in C^1(0,\infty)$, as required.\\[5pt]
Finally we demonstrate that (\ref{intRatio}) holds. Suppose $t \in [n,n+w_n)$ and write
\begin{align*}
k(t) &= k_s(t) + h_s(t-n,\frac{w_n}{2},\Gamma_+(t)-k_s(t) )
\leq \Gamma_+(t) \leq \Gamma_+(n+1).
\end{align*}
Hence it can be shown that 
\begin{align*}
\int_t^{\infty}{k(u)du} \leq \int_t^\infty{k_s(u)du} + \sum_{j=n}^\infty{w_j \Gamma_+(j+1)}, \, t \in [n, n+w_n).
\end{align*}
Similarly, for $t \in [n+w_n,n+1]$ we have 
\begin{align*}
\int_t^{\infty}{k(u)du} \leq \int_t^\infty{k_s(u)du} + \sum_{j=n+1}^\infty{w_j \Gamma_+(j+1)}.
\end{align*}
Thus 
\begin{align*}
\int_t^{\infty}{k(u)du} \leq \int_t^\infty{k_s(u)du} + \sum_{j=n}^\infty{w_j \Gamma_+(j+1)}, \, t \in [n, n+1].
\end{align*}
Hence for $t \in [n, n+1]$,
\begin{align*}
\frac{\int_t^{\infty}{k(u)du}}{\int_t^\infty{k_s(u)du}} &\leq 1 + \frac{\sum_{j=n}^\infty{w_j \Gamma_+(j+1)}}{\int_t^\infty{k_s(u)du}}
\leq 1 + \frac{\sum_{j=n}^\infty{w_j \Gamma_+(j+1)}}{\int_{n+1}^\infty{k_s(u)du}}\\ 
&= 1 + \frac{\sum_{j=n}^\infty{w_j \Gamma_+(j+1)}}{\sum_{j=n}^{\infty}{\int_{j+1}^{j+2}{k_s(u)du}}}.
\end{align*}
Now using (\ref{w_j}) we have
\begin{align*}
0 \leq \limsup_{n \to \infty}\frac{w_n \Gamma_+(n+1)}{\int_{n+1}^{n+2}{k_s(u)du}} \leq \limsup_{n \to \infty}\frac{1}{n+1} = 0.
\end{align*}
Hence 
\begin{align*}
\lim_{n \to \infty}\frac{w_n \Gamma_+(n+1)}{\int_{n+1}^{n+2}{k_s(u)du}} = 0,
\end{align*}
and by Toeplitz lemma \cite{Toeplitz} 
\begin{align*}
\lim_{n \to \infty}\frac{\sum_{j=n}^\infty{w_j \Gamma_+(j+1)}}{\sum_{j=n}^\infty{\int_{j+1}^{j+2}{k_s(u)du}}} = 0.
\end{align*}
Hence with $n(t) \in \N$ defined by $t \leq n(t) < t+1$ we have
\begin{align*}
\limsup_{t \to \infty}\frac{\int_t^{\infty}{k(u)du}}{\int_t^\infty{k_s(u)du}} \leq 1 + \limsup_{t \to \infty}\frac{\sum_{j=n(t)}^\infty{w_j \Gamma_+(j+1)}}{\sum_{j=n(t)}^\infty{\int_{j+1}^{j+2}{k_s(u)du}}} = 1.
\end{align*}
Since $k(t) \geq k_s(t)$, for all $t \geq 0$, 
\begin{align*}
\liminf_{t \to \infty}\frac{\int_t^\infty{k(u)du}}{\int_t^\infty{k_s(u)du}} \geq 1,
\end{align*}
which completes the proof.

\end{document}